\documentclass[12pt,twosides]{amsart}
\usepackage{amssymb,amsmath,amsthm, amscd, enumerate, mathrsfs}
\usepackage{graphicx, hhline}
\usepackage[all]{xy}
\usepackage{braket}
\usepackage[usenames]{color}
\usepackage{hyperref}
\usepackage{fancyhdr}
\usepackage[top=30truemm,bottom=30truemm,left=30truemm,right=30truemm]{geometry}

\hypersetup{colorlinks=true}

\title{On minimal model theory for log abundant lc pairs}
\author{Kenta Hashizume, Zheng-Yu Hu}
\date{2019/08/25, version 0.56}
\address{Graduate School of Mathematical Sciences, 
The University of Tokyo, 3-8-1 Komaba Meguro-ku Tokyo 153-8914, Japan}
\email{hkenta@ms.u-tokyo.ac.jp}

\address{Mathematical Sciences Research Center, 
	Chongqing University of Technology, No.69 Hongguang Avenue, Chongqing, 400054, China}
\email{zhengyuhu16@gmail.com}

\newtheorem{thm}{Theorem}[section]

\newtheorem{lem}[thm]{Lemma}

\newtheorem{prop}[thm]{Proposition}

\theoremstyle{definition}
\newtheorem{defn}[thm]{Definition}
\newtheorem{rem}[thm]{Remark}

\newtheorem*{ack}{Acknowledgments} 
\newtheorem*{divisor}{Divisors and morphisms} 
\newtheorem*{sing}{Singularities of pairs} 
\newtheorem*{model}{Models}

\newtheorem{step1}{Step}
\newtheorem{step2}{Step}
\newtheorem{step3}{Step}
\newtheorem{step4}{Step}

\newtheorem*{claim*}{Claim}
\begin{document}

\maketitle

\begin{abstract}
Under the assumption of the minimal model theory for projective klt pairs of dimension $n$, we establish the minimal model theory for lc pairs $(X/Z,\Delta)$ such that the log canonical divisor is relatively log abundant and its restriction to any lc center has relative numerical dimension at most $n$. 
We also give another detailed proof of results by the second author, and study termination of log MMP with scaling. 
\end{abstract}

\tableofcontents

\section{Introduction}\label{sec1}

Throughout this paper we will work over the complex number field $\mathbb{C}$. \\
\textbf{Log abundant lc pairs.} In the first half of this paper, we study log MMP for log abundant lc pairs.  
Let $\pi \colon X\to Z$ be a projective morphism, and let $(X,\Delta)$ be an lc pair. 
Then the divisor $K_{X}+\Delta$ is $\pi$-{\em log abundant} with respect to $(X,\Delta)$ if it is abundant over $Z$ and for any lc center $S$ of $(X,\Delta)$ with the normalization $S^{\nu}$, the pullback of $K_{X}+\Delta$ to $S^{\nu}$ is abundant over $Z$. 
For precise definition, see Subsection \ref{subsec2.2}. 
In this paper, non-pseudo-effective divisors are abundant by convention. 

The following theorem is the main result of this paper. 

\begin{thm}\label{thmmain}
Assume the existence of good minimal models or Mori fiber spaces for all projective klt pairs of dimension $n$. 

Let $\pi \colon X\to Z$ be a projective morphism of normal quasi-projective varieties, and let $(X,\Delta)$ be an lc pair such that $K_{X}+\Delta$ is $\pi$-log abundant with respect to $(X,\Delta)$ and the inequality $\kappa_{\sigma}(S^{\nu}/Z, (K_{X}+\Delta)|_{S^{\nu}})\leq n$ holds for any lc center $S$ of $(X,\Delta)$ with the normalization $S^{\nu}$. 

Then, $(X,\Delta)$ has a good minimal model or a Mori fiber space over $Z$. 
\end{thm}

In Theorem \ref{thmmain}, $\kappa_{\sigma}(S^{\nu}/Z, \;\cdot\;)$ denotes the relative numerical dimension with respect to the morphism $S^{\nu}\to Z$ (see Subsection \ref{subsec2.2}). 

The class of log abundant lc pairs is known as an important class in the minimal model theory. 
The class is very useful for inductive arguments on dimension of varieties. 
For example, the abundance conjecture for abundant klt pairs and log abundant lc pairs are studied by a lot of people (\cite{kawamata-abund}, \cite{reid-shokurov}, \cite{fujino-abund-logbig}, \cite{fukuda}, \cite{fujino-abund-saturation}, \cite{fujino-gongyo}, \cite{haconxu}), and the conjecture is currently known in full generality (\cite{haconxu}, and \cite{fujino-gongyo} for projective case). 
In klt case, the existence of good minimal models or Mori fiber spaces is known (\cite{lai} and \cite{gongyolehmann}). 
On the other hand, little is known about log MMP for log abundant lc pairs because we cannot check from construction that log MMP for log abundant lc pairs produces a log abundant lc pair in each step. 
Ambro and Koll\'ar \cite{ambrokollar} study behavior of non-klt loci of lc pairs under log MMP.
Birkar and the second author showed in \cite[Example 5.2]{birkarhu-arg} that the minimal model theory for log big lc pairs implies the abundance theorem in a special situation. 
Thus, the minimal model theory for log abundant lc pairs is a very difficult problem. 
Nevertheless, Theorem \ref{thmmain} shows that it holds true if we assume the minimal model theory in lower dimensions and an assumption on the restriction of log canonical divisor to lc centers. 
We note that Theorem \ref{thmmain} is one of inductive arguments.

With Theorem \ref{thmmain}, we strengthen the main result of \cite{bchm} and give further applications. 

\begin{thm}[=Theorem \ref{thm--gentype}]\label{thm--gen-type}
Assume the existence of good minimal models or Mori fiber spaces for all projective klt pairs of dimension $n$. 

Let $\pi\colon X\to Z$ be a projective morphism of normal quasi-projective varieties, and let $(X,\Delta)$ be an lc pair such that
\begin{itemize}
\item
$K_{X}+\Delta$ is abundant over $Z $ or $\Delta$ is big over $Z$, and 
\item
for any lc center $S$ of $(X,\Delta)$, we have ${\rm dim}S-{\rm dim}\,\pi(S)\leq n$. 
\end{itemize}
Then, $(X,\Delta)$ has a good minimal model or a Mori fiber space over $Z$.  
\end{thm}

\begin{thm}\label{thm--calabiyau}
Let $(X,B)$ be a projective Calabi--Yau pair, that is, a projective lc pair $(X,B)$ such that $K_{X}+B\equiv0$. 
Suppose that $-K_{X}$ is big and any lc center of $(X,B)$ is at most $3$-dimensional. 

Then, for any effective $\mathbb{Q}$-Cartier divisor $D$ whose support contains no lc centers of $(X,B)$, the graded $\mathbb{C}$-algebra $\mathcal{R}(X,D)=\bigoplus_{m\geq0}H^{0}(X,\mathcal{O}_{X}(\llcorner mD\lrcorner))$ is finitely generated. 
In particular, if $K_{X}$ is $\mathbb{Q}$-Cartier and $(X,(1+t)B)$ is lc for some $t>0$, then  $\mathcal{R}(X,-K_{X})$ is finitely generated, and $X$ admits a birational contraction to an lc Fano variety. 
\end{thm}

In case when $-K_{X}$ is nef in Theorem \ref{thm--calabiyau}, the statement shows the semi-ampleness of anti-canonical divisor of lc weak Fano varieties under the assumptions on dimension of lc centers and existence of a Calabi--Yau pair whose lc centers coincide with lc centers of underlying variety.

\begin{thm}\label{thm--flat}
Assume the existence of good minimal models or Mori fiber spaces for all projective klt pairs of dimension $n$. 

Let $\pi\colon X\to Z$ be a projective surjective morphism of normal quasi-projective varieties such that ${\rm dim}X-{\rm dim}Z\leq n$ and all fibers of $\pi$ have the same dimension, and let $(X,\Delta)$ be an lc pair. 
Then $(X,\Delta)$ has a good minimal model or a Mori fiber space over $Z$. 
\end{thm}

Theorem \ref{thm--flat} shows a relation between the minimal model theory in the absolute setting and that in the relative setting.  
For proofs of those applications, see Section \ref{sec4}. \\
\\
{\textbf{Lc pairs with boundaries containing ample divisors.}}
Using ideas in the proof of Theorem \ref{thmmain}, in the latter half of this paper, we give  an alternative proof of results originally announced by the second author \cite{bh}. 

\begin{thm}\label{thm--lcample}
Let $\pi\colon X\to Z$ be a projective morphism of normal quasi-projective varieties, and let $(X,B)$ be an lc pair. 
Let $A$ be a $\pi$-ample $\mathbb{R}$-divisor on $X$ such that $(X,B+A)$ is lc. 

Then, $(X,B+A)$ has a good minimal model or a Mori fiber space over $Z$. 
\end{thm}

The key ingredient to prove Theorem \ref{thm--lcample} is Theorem \ref{thmmain--hu}, which is equivalent to a combination of \cite[Theorem 1.1]{bh} and \cite[Theorem 3.1]{bh} by the second author (see also \cite[Theorem 1.1]{has-mmp}). 
Most part of Section \ref{sec5} is devoted to prove Theorem \ref{thmmain--hu}, which is equivalent to giving another detailed proof of \cite[Theorem 1.1 and Theorem 3.1]{bh}. 

With Theorem \ref{thm--lcample}, we prove a special kind of canonical bundle formula for lc pairs. 

\begin{thm}\label{thm--lcfano}
Let $f\colon X\to Y$ be a contraction of normal projective varieties. Suppose that either $X$ or $Y$ is $\mathbb{Q}$-factorial.
If there is an effective $\mathbb{R}$-divisor $B$ such that $(X,B)$ is lc and $-(K_{X}+B)$ is ample, then there is an effective $\mathbb{R}$-divisor $B_{Y}$ such that $(Y,B_{Y})$ is lc and $-(K_{Y}+B_{Y})$ is ample. 
\end{thm}

For details of proof, see Section \ref{sec5}. 

In Section \ref{sec6}, we study termination of non-$\mathbb{Q}$-factorial log MMP with scaling of an ample divisor. 
We would like to note that termination of non-$\mathbb{Q}$-factorial log MMP is highly nontrivial and considerably important to our understandings of the minimal model theory, and the following theorem is essential to a thorough study of this task. 

\begin{thm}\label{thm--mmp-scaling}
Let $\pi\colon X\to Z$ be a projective morphism of normal quasi-projective varieties, and let $(X,B)$ be an lc pair. 
Suppose that $(X,B)$ has a log minimal model over $Z$ or $K_{X}+B$ is not pseudo-effective over $Z$. 
Let $A$ be an ample $\mathbb{R}$-divisor on $X$ such that $(X,B+A)$ is lc and $K_{X}+B+A$ is nef over $Z$.  
  
Then there is a sequence of birational contractions 
\begin{equation*}
(X,B)=(X_{1},B_{1})\dashrightarrow (X_{2},B_{2})\dashrightarrow \cdots \dashrightarrow (X_{l},B_{l})
\end{equation*}
of a non-$\mathbb{Q}$-factorial $(K_{X}+B)$-MMP over $Z$ with scaling of $A$ that terminates. 
The final lc pair $(X_{l},B_{l})$, where $X_{l}$ may not be $\mathbb{Q}$-factorial, satisfies one of the following two conditions:
\begin{enumerate}
\item
 ((non-$\mathbb{Q}$-factorial) log minimal model). $K_{X_{l}}+B_{l}$ is nef over $Z$. 
\item
 ((non-$\mathbb{Q}$-factorial) Mori fiber space). There is a contraction $X_{l}\to Y$ over $Z$ to a normal quasi-projective variety $Y$ such that ${\rm dim}Y<{\rm dim}X_{l}$, $-(K_{X_{l}}+B_{l})$ is ample over $Y$ and the relative Picard number $\rho(X_{l}/Y)$ is one. 
\end{enumerate}
\end{thm}

We prove Theorem \ref{thm--mmp-scaling} by applying Theorem \ref{thm--lcample}. 

To explain differences between our proof and proofs in \cite{bh}, we compare arguments in sections \ref{sec5} and \ref{sec6} with those in \cite{bh}. 
As stated before, the goal of Section \ref{sec5}--\ref{sec6} is to give an alternative proof of results in \cite{bh}, so our arguments are completely independent from \cite{bh}, and moreover we prove further results (Theorem \ref{thm--lcfano} and Theorem \ref{thm--cont-fanofib}). 
First, we compare proof of Theorem \ref{thm--lcample} with that of \cite[Theorem 1.1]{bh}. 
Both proofs use the same reduction, that is, a reduction to the special termination by using a special kind of log MMP (\cite{birkar-flip}, \cite{has-trivial}, \cite{has-mmp}, \cite{has-class}, \cite[Section 2]{bh}). 
But proofs after the reduction are different. 
In the proof of \cite[Theorem 1.1]{bh} in \cite[Section 5]{bh}, careful analysis of non-nef loci and behavior of open neighborhoods near the generic point of each lc center of lc pairs play crucial roles. 
On the other hand, our proof does not use these techniques, but we apply ideas of the  proof of Theorem \ref{thmmain}. 
More specifically, we apply results of log abundant log canonical divisors of lc pairs and \cite[Proof of Theorem 1.2]{birkar-09}, which are more purely minimal model theoretic and more natural to extend to a more general setting. 
Next, we compare Theorem \ref{thm--mmp-scaling} with \cite[Theorem 1.4 and Corollary 1.5]{bh}. 
We emphasize that Theorem \ref{thm--mmp-scaling} is in the framework of non-$\mathbb{Q}$-factorial pairs (compare Theorem \ref{thm--mmp-scaling} with \cite[Theorem 1.4 and Corollary 1.5]{bh}).  
Furthermore, in proof of Theorem \ref{thm--mmp-scaling} we make good use of length of extremal rays (\cite[Section 18]{fujino-fund}) 
and a result of termination of log MMP with scaling (\cite[Theorem 4.1]{birkar-flip}) without using the geography of models (see \cite[Section 6]{shokurov}) or theory of finitely generated adjoint rings (for example, linearity of asymptotic vanishing orders) established in \cite{cortilazic}, which are heavily used in \cite[Section 6]{bh}. 
But, compared with \cite[Theorem 1.4 and Corollary 1.5]{bh}, we only prove the existence of a sequence of steps of log MMP that terminates. 

It is natural to expect that all log MMP with scaling of an ample divisor terminate in the situation of Theorem \ref{thm--mmp-scaling}, and our result would be a feasible starting point. \\
\\
{\textbf{Contents of the paper.}}
This paper is divided into two parts: Section \ref{sec2}--\ref{sec4} and Section \ref{sec5}--\ref{sec6}. 
In Section \ref{sec2}, we collect definitions, notations and results on log MMP. 
In Section \ref{sec3}, we prove Theorem \ref{thmmain}. 
In Section \ref{sec4}, we prove Theorem \ref{thm--gen-type}, Theorem \ref{thm--calabiyau}, Theorem \ref{thm--flat} and the minimal model theory for lc pairs with log big boundary divisors in a special case (Theorem \ref{thm--logbig}). 
In Section \ref{sec5}, we prove Theorem \ref{thm--lcample}, Theorem \ref{thm--lcfano} and an application (Theorem \ref{thm--cont-fanofib}). 
In Section \ref{sec6}, we prove Theorem \ref{thm--mmp-scaling}. 

\begin{ack}
The first author was partially supported by JSPS KAKENHI Grant Number JP16J05875 and JP19J00046. 
Part of the work was done while the first author was visiting University of Cambridge in October--November 2018. The first author thanks staffs of the university, Professor Caucher Birkar, Dr.~Roberto Svaldi, and Dr.~Yanning Xu for their hospitality. 
He thanks audience of Algebraic Geometry Seminar in the University of Tokyo, and he  thanks Dr.~Kenta Sato for discussion on the proof of Lemma \ref{lem--abundant-general}. 
Part of the work was done when the second author was visiting CMS, Zhejiang University. He would like to thank Professor Kefeng Liu and Professor Hongwei Xu. 
The authors are grateful to Professors Hiromu Tanaka and Shinnosuke Okawa for advice and answering questions about Theorem \ref{thm--lcfano}. 
They thank Professor Osamu Fujino for reading the manuscript and giving comments. 
They thank the referee for reading the draft carefully and a lot of suggestion. 
\end{ack}

\section{Preliminaries}\label{sec2}
In this section, we collect definitions and some important results. 

\subsection{Notations and definitions}
We collect some notations and definitions, and we show some basic lemmas. 

\begin{divisor}
Let $X$ be a normal projective variety. 
We use the standard definitions of nef $\mathbb{R}$-divisor, ample $\mathbb{R}$-divisor, semi-ample $\mathbb{R}$-divisor, etc. 
In this paper, we do not assume big $\mathbb{R}$-divisors to be $\mathbb{R}$-Cartier.
For any morphism $f\colon X\to Y$ and any $\mathbb{R}$-Cartier divisor $D$ on $Y$, we sometimes denote $f^{*}D$ by $D|_{X}$. 

For a variety $X$ and an effective $\mathbb{R}$-divisor $D$ on it, a {\em log resolution of} $(X,{\rm Supp}D)$ (or a {\em log resolution of} $(X,D)$ if there is no risk of confusion) denotes a projective birational morphism $f\colon Y\to X$ from a smooth variety $Y$ such that the exceptional locus ${\rm Ex}(f)$ is pure codimension one and ${\rm Ex}(f)\cup {\rm Supp}f_{*}^{-1}D$ is a simple normal crossing divisor. 

Let $f\colon X\to Y$ be a projective morphism of varieties. 
Then $f$ is called a {\em contraction} if it is surjective and has connected fibers. 
Let $D$ be a semi-ample $\mathbb{R}$-divisor on $X$. 
We say that $f$ is a {\em contraction induced by} $D$ if $f$ is a contraction and we have $D\sim_{\mathbb{R}}f^{*}A$ for an ample $\mathbb{R}$-divisor $A$ on $Y$. 
We note that a contraction induced by $D$ always exists for any semi-ample $\mathbb{R}$-divisor $D$. 

Let $D$ be a semi-ample $\mathbb{R}$-divisor on a normal projective variety $X$. 
We can write $D=\sum_{i}r_{i}D_{i}$, where $r_{i}$ are positive real numbers and $D_{i}$ are semi-ample $\mathbb{Q}$-Cartier divisors.
Then $D$ is {\em general} (resp.~{\em sufficiently general}) if there is a sufficiently large and divisible integer $k>0$ such that $kD_{i}$ are base point free Cartier divisors and $D_{i}=\frac{1}{k}D'_{i}$, where $D'_{i}$ are general (resp.~sufficiently general) elements of $|kD_{i}|$. 
In particular, when $D$ is general, $D$ is effective and all coefficients of $D$ are less than one. 

\begin{lem}\label{lem--genmember}
Let $f\colon X\dashrightarrow X'$ be a birational contraction of projective $\mathbb{Q}$-factorial varieties. 
Let $S$ and $S'$ be normal projective subvarieties of $X$ and $X'$ respectively such that the indeterminacy locus of $f$ does not contain $S$ and the restriction of $f$ to $S$ induces a birational map $f_{S}\colon S\dashrightarrow S'$. 
Fix a common resolution $\phi\colon T\to S$ and $\phi'\colon T\to S'$ of  $f_{S}$.  
Let $D$ be a semi-ample $\mathbb{R}$-divisor on $X$. 

Then, any general $A\sim_{\mathbb{R}}D$ satisfies $A_{S}\geq 0$, $A_{S'}\geq0$, and $\phi^{*}A_{S}\leq \phi_{*}'^{-1}A_{S'}$, where $A_{S}=A|_{S}$ and $A_{S'}=(f_*A)|_{S'}$. 
\end{lem}

\begin{proof}
Let $Z$ be the smallest closed subset of $X$ such that $f$ is an isomorphism on $X\setminus Z$. 
Then, for any general $A\sim_{\mathbb{R}}D$ such that ${\rm Supp}A$ does not contain any component of $Z$ and $Z\cap S$,  we have $A_{S}\geq 0$, $A_{S'}\geq0$ and $f_{S*}A_{S}\leq A_{S'}$. 
Therefore, if we pick a general $A\sim_{\mathbb{R}}D$ so that $A$ satisfies the above conditions, $\phi^{*}A_{S}=\phi_{*}^{-1}A_{S}$, and ${\rm Supp}\phi^{*}A_{S}$ does not contain any $\phi'$-exceptional prime divisors, then $A$ satisfies all the conditions in Lemma \ref{lem--genmember}. 
Note that $A_{S}$ and $A_{S'}$ are $\mathbb{R}$-Cartier since $X$ and $X'$ are $\mathbb{Q}$-factorial.   
\end{proof}

\begin{lem}\label{lem--genmember2}
Let $f\colon X\dashrightarrow X'$ be a birational map of normal projective varieties.
Let $\Delta'$ be a reduced divisor on $X'$, and let $E'$ be the reduced $f^{-1}$-exceptional divisor. 
We fix a log resolution $g'\colon W\to X'$ of $(X',{\rm Supp}(\Delta'+E'))$ such that the induced birational map $g\colon W\to X$ is a morphism. 
Let $D$ be a semi-ample $\mathbb{R}$-divisor on $X$. 

Then, for any general $A\sim_{\mathbb{R}}D$, the morphism $g'\colon W\to X'$ is a log resolution of $(X',{\rm Supp}(\Delta'+E'+f_{*}A))$. 
\end{lem}

\begin{proof}
We note that $g'$ is projective, $W$ is smooth and ${\rm Ex}(g')$ is pure codimension one. 
So we only have to prove that ${\rm Supp}g'^{-1}_{*}(\Delta'+E'+f_{*}A)\cup {\rm Ex}(g')$ is a simple normal crossing divisor. 
By construction of $g'$ and $E'$, this condition is satisfied when ${\rm Supp}g^{*}A$ intersects ${\rm Supp}g'^{-1}_{*}(\Delta'+E')\cup {\rm Ex}(g')$ transversely.
By an application of Bertini type theorem, we see that any general $A\sim_{\mathbb{R}}D$ satisfies this condition. 
\end{proof}
\end{divisor}

\begin{sing}
A {\em pair} $(X,\Delta)$ consists of a normal variety $X$ and a boundary $\mathbb{R}$-divisor $\Delta$ on $X$ such that $K_{X}+\Delta$ is $\mathbb{R}$-Cartier. 

Let $(X,\Delta)$ be a pair, and let $P$ be a prime divisor over $X$. 
Then $a(P,X,\Delta)$ denotes the discrepancy of $P$ with respect to $(X,\Delta)$. 
We use standard definitions of Kawamata log terminal (klt, for short) pair, log canonical (lc, for short) pair and divisorially log terminal (dlt, for short) pair as in \cite{kollar-mori}. 
In \cite{kollar-mori}, pairs and classes of singularity are defined in the framework of $\mathbb{Q}$-divisors. 
But, we can similarly define those classes of singularity for pairs of a normal variety and a boundary $\mathbb{R}$-divisor. 
When $(X,\Delta)$ is an lc pair, an {\em lc center} of $(X,\Delta)$ is the image on $X$ of a prime divisor $P$ over $X$ whose discrepancy $a(P,X,\Delta)$ is equal to $-1$. 
\end{sing}

\begin{lem}\label{lem--discre}
Let $(X,\Delta)$ be a pair, and let $f\colon Y\to X$ be a log resolution of $(X,\Delta)$. 
We define $\Gamma\geq0$ and $E\geq0$ so that $K_{Y}+\Gamma=f^{*}(K_{X}+\Delta)+E$ and $\Gamma$ and $E$ have no common components. 
Then $a(P,Y,\Gamma)={\rm min}\{0, a(P,X,\Delta)\}$ for any prime divisor $P$ on $Y$.  
\end{lem}

\begin{proof}
It is clear from definition of discrepancy. 
\end{proof}

\begin{lem}\label{lem--adjunction}
Let $(X,\Delta)$ be a dlt pair and $S$ a component of $\llcorner \Delta \lrcorner$. 
Fix a log resolution $f\colon Y\to X$ of $(X,\Delta)$ and we write $K_{Y}+T+\Gamma=f^{*}(K_{X}+\Delta)+E$, where $T$ is the birational transform of $S$ on $Y$ and $\Gamma\geq 0$ and $E\geq 0$   have no common components. 
Then $\Gamma|_{T}$ and $E|_{T}$ have no common components and $E|_{T}$ is exceptional over $S$.  
\end{lem}

\begin{proof}
Suppose that there is a common component $D$ of $\Gamma|_{T}$ and $E|_{T}$. 
Then there are components $\Gamma'$ and $E'$ of $\Gamma$ and $E$ respectively such that $D\subset T\cap \Gamma'\cap E'$. 
But the left hand side has codimension $2$ in $Y$ and the right hand side is codimension at least $3$ in $Y$ since $Y\to X$ is a log resolution of $(X,\Delta)$, so we get a contradiction. 

We define $\Delta_{S}$ by adjunction $K_{S}+\Delta_{S}=(K_{X}+\Delta)|_{S}$. 
Then $\Delta_{S}\geq 0$ (see, for example, \cite[Section 14]{fujino-fund}) and $\Delta_{S}=f|_{T*}(\Gamma|_{T})-f|_{T*}(E|_{T})$. 
Since the morphism $f|_{T}\colon T\to S$ is birational, $f|_{T*}(\Gamma|_{T})$ and $f|_{T*}(E|_{T})$ have no common components. 
Thus $f|_{T*}(E|_{T})=0$, and hence $E|_{T}$ is exceptional over $S$.  
\end{proof}

\begin{model}\label{deflogbir}
In this paper, we freely use definition of models as in \cite[Definition 2.2]{has-trivial}. 
We write down the precise definition here. 

Let $\pi\colon X \to Z$ be a projective morphism from a normal variety to a variety, and let $(X,\Delta)$ be an lc pair.  
Let $\pi '\colon X' \to Z$ be a projective morphism from a normal variety to $Z$ and  $\phi\colon X \dashrightarrow X'$ a birational map over $Z$. 
Let $E$ be the reduced $\phi^{-1}$-exceptional divisor on $X'$, that is, $E=\sum E_{j}$ where $E_{j}$ are $\phi^{-1}$-exceptional prime divisors on $X'$. 
Then the pair $(X', \Delta'=\phi_{*}\Delta+E)$ is called a {\it log birational model} of $(X,\Delta)$ over $Z$. 
A log birational model $(X', \Delta')$ of $(X,\Delta)$ over $Z$ is a {\it weak log canonical model} ({\it weak lc model}, for short) if 
\begin{itemize}
\item
$K_{X'}+\Delta'$ is nef over $Z$, and 
\item
for any prime divisor $D$ on $X$ which is exceptional over $X'$, we have
\begin{equation*}
a(D, X, \Delta) \leq a(D, X', \Delta').
\end{equation*}
\end{itemize}
A weak lc model $(X',\Delta')$ of $(X,\Delta)$ over $Z$ is a {\it log minimal model} if 
\begin{itemize}
\item
$X'$ is $\mathbb{Q}$-factorial, and 
\item
the above inequality on discrepancies is strict. 

\end{itemize}
A log minimal model $(X',\Delta')$ of $(X, \Delta)$ over $Z$ is called a {\it good minimal model} if $K_{X'}+\Delta'$ is semi-ample over $Z$. 

A log birational model $(X',\Delta')$ of $(X, \Delta)$ over $Z$ is called a {\it Mori fiber space} if $X'$ is $\mathbb{Q}$-factorial and there is a contraction $X' \to W$ over $Z$ with ${\rm dim}W<{\rm dim}X'$ such that 
\begin{itemize}
\item
the relative Picard number $\rho(X'/W)$ is one and $-(K_{X'}+\Delta')$ is ample over $W$, and 
\item
for any prime divisor $D$ over $X$, we have
\begin{equation*}
a(D,X,\Delta)\leq a(D,X',\Delta')
\end{equation*}
and strict inequality holds if $D$ is a divisor on $X$ and exceptional over $X'$.
\end{itemize} 
In particular, we do not assume that log minimal models and Mori fiber spaces are dlt. 

For any lc pair $(X,\Delta)$ on a normal quasi-projective variety $X$, we freely construct a dlt blow-up $(Y,\Gamma)\to (X,\Delta)$ as in \cite[Theorem 10.4]{fujino-fund} or \cite[Theorem 3.1]{kollarkovacs}. 
In this paper, we call $(Y,\Gamma)$ a {\em dlt model}. 
\end{model}

\begin{rem}\label{remmodels}
Let $(X,\Delta)$ be an lc pair and $(X',\Delta')$ be a log minimal model of $(X,\Delta)$. 
Let $(X'',\Delta'')$ be a $\mathbb{Q}$-factorial lc pair such that $K_{X''}+\Delta''$ is nef, $X''$ and $X'$ are isomorphic in codimension one, and $\Delta''$ is the birational transform of $\Delta'$ on $X''$. 
Then $(X'',\Delta'')$ is also a log minimal model of $(X,\Delta)$. 
Moreover, if $(X',\Delta')$ is a good minimal model of $(X,\Delta)$, then $(X'',\Delta'')$ is also a good minimal model of $(X,\Delta)$. 
\end{rem}

\subsection{Invariant Iitaka dimension and numerical dimension}\label{subsec2.2}

\begin{defn}[Invariant Iitaka dimension]\label{defn--inv-iitaka-dim}
Let $X$ be a normal projective variety, and let $D$ be an $\mathbb{R}$-Cartier $\mathbb{R}$-divisor on $X$. 
We define the {\em invariant  Iitaka dimension} of $D$, denoted by $\kappa_{\iota}(X,D)$, as follows (see also \cite[Definition 2.5.5]{fujino-book}):  
If there is an $\mathbb{R}$-divisor $E\geq 0$ such that $D\sim_{\mathbb{R}}E$, set $\kappa_{\iota}(X,D)=\kappa(X,E)$. 
Here, the right hand side is the usual Iitaka dimension of $E$. 
Otherwise, we set $\kappa_{\iota}(X,D)=-\infty$. 
We can check that $\kappa_{\iota}(X,D)$ is well-defined, i.e., when there is $E\geq 0$ such that $D\sim_{\mathbb{R}}E$, the invariant Iitaka dimension $\kappa_{\iota}(X,D)$ does not depend on the choice of $E$. 
By definition, we have $\kappa_{\iota}(X,D)\geq0$ if and only if $D$ is $\mathbb{R}$-linearly equivalent to an effective $\mathbb{R}$-divisor. 

Let $X\to Z$ be a projective morphism from a normal variety to a variety, and let $D$ be an $\mathbb{R}$-Cartier $\mathbb{R}$-divisor on $X$. 
Then the {\em relative invariant  Iitaka dimension} of $D$, denoted by $\kappa_{\iota}(X/Z,D)$, is similarly defined: If there is an $\mathbb{R}$-divisor $E\geq 0$ such that $D\sim_{\mathbb{R},Z}E$, set $\kappa_{\iota}(X/Z,D)=\kappa_{\iota}(F,D|_{F})$, where $F$ is a sufficiently general fiber of the Stein factorization of $X\to Z$, and otherwise we set $\kappa_{\iota}(X/Z,D)=-\infty$. 
When there is $E\geq0$ such that $D\sim_{\mathbb{R},Z}E$, by semi-continuity of dimension of cohomology of flat coherent sheaves, we can check that $\kappa_{\iota}(X/Z,D)$ does not depend on the choice of $E$ and $F$. 
By definition, we have $\kappa_{\iota}(X/Z,D)\geq0$ if and only if $D\sim_{\mathbb{R},Z}E$ for some $E\geq0$. 
\end{defn}

\begin{defn}[Numerical dimension]\label{defn--num-dim}
Let $X$ be a normal projective variety, and let $D$ be an $\mathbb{R}$-Cartier $\mathbb{R}$-divisor on $X$. 
We define the {\em numerical dimension} of $D$, denoted by $\kappa_{\sigma}(X,D)$, as follows (see also \cite[V, 2.5 Definition]{nakayama}): 
For any Cartier divisor $A$ on $X$, we set
\begin{equation*}
\sigma(D;A)={\rm max}\!\Set{\!k\in \mathbb{Z}_{\geq0} | \underset{m\to \infty}{\rm lim\,sup}\frac{{\rm dim}H^{0}(X,\mathcal{O}_{X}(\llcorner mD \lrcorner+A))}{m^{k}}>0\!}
\end{equation*}
if ${\rm dim}H^{0}(X,\mathcal{O}_{X}(\llcorner mD \lrcorner+A))>0$ for infinitely many $m\in \mathbb{Z}_{>0}$, and otherwise we set $\sigma(D;A):=-\infty$. 
Then, we define 
\begin{equation*}
\kappa_{\sigma}(X,D):={\rm max}\!\set{\sigma(D;A) | \text{$A$ is a Cartier divisor on $X$}\!}.
\end{equation*}

Let $X\to Z$ be a projective morphism from a normal variety to a variety, and let $D$ be an $\mathbb{R}$-Cartier $\mathbb{R}$-divisor on $X$. 
Then, the {\em relative numerical dimension} of $D$ over $Z$ is defined by $\kappa_{\sigma}(F,D|_{F})$, where $F$ is a sufficiently general fiber of the Stein factorization of $X\to Z$.  
We note that the value $\kappa_{\sigma}(F,D|_{F})$ does not depend on the choice of $F$, so the relative numerical dimension is well-defined. 
In this paper, we denote the relative numerical dimension by $\kappa_{\sigma}(X/Z,D)$. 
\end{defn}

\begin{rem}\label{remdiv}
We write down basic properties of the invariant Iitaka dimension and the numerical dimension. 
\begin{enumerate}
\item
Let $D_{1}$ and $D_{2}$ be $\mathbb{R}$-Cartier $\mathbb{R}$-divisors on a normal projective variety $X$. 
\begin{itemize}
\item
Suppose that $D_{1}\sim_{\mathbb{R}}D_{2}$ holds. 
Then we have $\kappa_{\iota}(X,D_{1})=\kappa_{\iota}(X,D_{2})$ and $\kappa_{\sigma}(X,D_{1})=\kappa_{\sigma}(X, D_{2})$. 
\item
Suppose that $D_{1}\sim_{\mathbb{R}}N_{1}$ and $D_{2}\sim_{\mathbb{R}}N_{2}$ for some $\mathbb{R}$-divisors $N_{1}\geq0$ and $N_{2}\geq0$ such that ${\rm Supp}N_{1}={\rm Supp}N_{2}$. 
Then we have $\kappa_{\iota}(X,D_{1})=\kappa_{\iota}(X,D_{2})$ and $\kappa_{\sigma}(X,D_{1})=\kappa_{\sigma}(X, D_{2})$. 
\end{itemize}
\item
Let $f\colon Y \to X$ be a surjective morphism of normal projective varieties, and let $D$ be an $\mathbb{R}$-Cartier $\mathbb{R}$-divisor on $X$. 
\begin{itemize}
\item
We have $\kappa_{\iota}(X,D)=\kappa_{\iota}(Y,f^{*}D)$ and $\kappa_{\sigma}(X,D)=\kappa_{\sigma}(Y,f^{*}D)$.
\item
Suppose that $f$ is birational. 
Let $D'$ be an $\mathbb{R}$-Cartier $\mathbb{R}$-divisor on $Y$ such that  $D'=f^{*}D+E$ for some effective $f$-exceptional divisor $E$. 
Then we have $\kappa_{\iota}(X,D)=\kappa_{\iota}(Y,D')$ and $\kappa_{\sigma}(X,D)=\kappa_{\sigma}(Y, D')$.
\end{itemize}
\end{enumerate}
\end{rem}

\begin{defn}[Relatively abundant divisor and relatively log abundant divisor]\label{defn--abund}
Let $\pi\colon X\to Z$ be a projective morphism from a normal variety to a variety, and let $D$ be an $\mathbb{R}$-Cartier $\mathbb{R}$-divisor on $X$. 
We say that $D$ is $\pi$-{\em abundant} (or {\em abundant over} $Z$) if the equality $\kappa_{\iota}(X/Z,D)=\kappa_{\sigma}(X/Z,D)$ holds. 
When $Z$ is a point, we simply say that $D$ is {\em abundant}. 

Let $\pi\colon X\to Z$ and $D$ be as above, and let $(X,\Delta)$ be an lc pair. 
We say that $D$ is $\pi$-{\em log abundant} (or {\em log abundant over} $Z$) with respect to $(X,\Delta)$ if $D$ is abundant over $Z$ and for any lc center $S$ of $(X,\Delta)$ with the normalization $S^{\nu}\to S$, the pullback $D|_{S^{\nu}}$ is abundant over $Z$. 
\end{defn}

The following lemma is well-known in the case of $\mathbb{Q}$-divisors, but, to the best of our knowledge, we cannot find the $\mathbb{R}$-divisors case in the literature. 
We write a detailed proof for the reader's convenience. 

\begin{lem}\label{lem--abundant-general}
Let $\pi\colon X\to Z$ be a projective morphism from a normal variety to a quasi-projective variety, and let $D$ be an $\mathbb{R}$-Cartier divisor on $X$. 
If $\kappa_{\iota}(F_{0},D|_{F_{0}})\geq0$ for a sufficiently general fiber $F_{0}$ of the Stein factorization of $\pi$, we have $\kappa_{\iota}(X/Z,D)\geq0$. 
In particular, $D$ is abundant over $Z$ if and only if $D|_{F}$ is abundant for any sufficiently general fiber $F$ of the Stein factorization of $\pi$. 
\end{lem}

\begin{proof}
By Definition \ref{defn--inv-iitaka-dim}, we see that the first assertion shows $\kappa_{\iota}(X/Z,D)=\kappa_{\iota}(F,D|_{F})$ for any sufficiently general fiber $F$ of the Stein factorization of $\pi$. 
By Definition \ref{defn--num-dim} the equality $\kappa_{\sigma}(X/Z,D)=\kappa_{\sigma}(F,D|_{F})$ holds, so the second assertion follows from the first assertion. 
Therefore, it is sufficient to prove the first assertion. 
Assume $\kappa_{\iota}(F_{0},D|_{F_{0}})\geq0$ for a sufficiently general fiber $F_{0}$ of the Stein factorization of $\pi$. 

Let $\pi'\colon X\to Z'$ be the Stein factorization of $\pi$, and let $U\subset Z'$ be an open subset over which $\pi'$ is flat. 
Since the morphism $Z'\to Z$ is finite, if there is $E'\geq0$ such that $D|_{\pi'^{-1}(U)}\sim_{\mathbb{R},U}E'$ then we can find $E'_{+}\geq0$ and $E'_{-}\geq0$ on $X$ such that $D\sim_{\mathbb{R}}E'_{+}-E'_{-}$ and codimension of $\pi({\rm Supp}E'_{-})$ in $Z$ is at least one. 
Since $Z$ is quasi-projective, there is a Cartier divisor $A\geq 0$ on $Z$ such that ${\rm Supp}\pi^{*}A\supset {\rm Supp}E'_{-}$. 
Rescaling $A$ if necessary, we may assume $\pi^{*}A-E'_{-}\geq0$. 
Then $D\sim_{\mathbb{R},Z}E'_{+}+(\pi^{*}A-E'_{-})\geq0$, and so we have $\kappa_{\iota}(X/Z,D)\geq0$. 
We may assume $\pi'(F_{0})\in U$ since $F_{0}$ is sufficiently general. 
So we may replace $X\to Z$ by $\pi'^{-1}(U)\to U$, and we may assume that $\pi$ is a contraction and flat. 

Fix a representation $D=\sum_{i=1}^{n}r_{i}D_{i}$ of $D$ as an $\mathbb{R}$-linear combination of Cartier divisors. 
Since $F_{0}$ is sufficiently general, we may assume that $D_{i}|_{F_{0}}$ are well-defined as Weil divisors. 
For any closed point $z\in Z$, we denote by $X_{z}$ the fiber over $z$ of $\pi$. 
For any rational point $\boldsymbol{p}=(p_{1},\cdots,p_{n})\in \mathbb{R}^{n}$, we set $D_{\boldsymbol{p}}=\sum_{i=1}^{n}p_{i}D_{i}$ and we fix a positive integer $k_{\boldsymbol{p}}$ such that $k_{\boldsymbol{p}}D_{\boldsymbol{p}}$ is Cartier. 
In addition, for any $\boldsymbol{p}$ and $m\in \mathbb{Z}_{>0}$, we define
\begin{equation*}
S_{\boldsymbol{p},m}=\set{\! z\in Z|{\rm dim}H^{0}(X_{z},\mathcal{O}_{X_{z}}(mk_{\boldsymbol{p}}D_{\boldsymbol{p}}|_{X_{z}}))=0\!},
\end{equation*}
which is empty or open by upper semi-continuity of dimension of cohomology of flat coherent sheaves. 
We set $J=\set{\!(\boldsymbol{p},m)|\boldsymbol{p}\in\mathbb{Q}^{n},m\in \mathbb{Z}_{>0}, S_{\boldsymbol{p},m}\neq \emptyset}$. 
Then the set
\begin{equation*}
W=\bigcap_{(\boldsymbol{p},m)\in J}S_{\boldsymbol{p},m}
\end{equation*}
is an intersection of countably many open subsets in $Z$, and we may assume $F_{0}=X_{z_{0}}$ for some $z_{0}\in W$. 
Then, for any $\mathbb{Q}$-Cartier divisor $D_{\boldsymbol{p}'}$ associated to a rational point $\boldsymbol{p}'=(p'_{1},\cdots,p'_{n})\in \mathbb{R}^{n}$, an inequality $\kappa(F_{0},D_{\boldsymbol{p}'}|_{F_{0}})\geq0$ holds if and only if $D_{\boldsymbol{p}'}\sim_{\mathbb{Q},Z}E_{\mathbb{\boldsymbol{p}'}}$ for some $E_{\mathbb{\boldsymbol{p}'}}\geq0$. 
Indeed, $\kappa(F_{0},D_{\boldsymbol{p}'}|_{F_{0}})\geq0$ if and only if $z_{0}\not\in S_{\boldsymbol{p}',m}$ for some $m$. 
But, since $z_{0}\in W$ and by construction of $W$ and $J$, the condition $z_{0}\not\in S_{\boldsymbol{p}',m}$ is equivalent to $S_{\boldsymbol{p}',m}=\emptyset$. 
Since $D_{\boldsymbol{p}'}$ is a $\mathbb{Q}$-divisor, by construction of $S_{\boldsymbol{p}',m}$, it is easy to check that $S_{\boldsymbol{p}',m}=\emptyset$ for some $m$ if and only if $D_{\boldsymbol{p}'}\sim_{\mathbb{Q},Z}E_{\mathbb{\boldsymbol{p}'}}$ for some $E_{\mathbb{\boldsymbol{p}'}}\geq0$. 
We note that the assumption $\kappa_{\iota}(F_{0},D|_{F_{0}})\geq 0$ is not used in this paragraph and the previous paragraph. 

From our assumption that $\kappa_{\iota}(F_{0},D|_{F_{0}})\geq 0$, there are positive real numbers $a_{1},\cdots,a_{s}$, effective Weil divisors $E_{1},\cdots, E_{s}$ on $F_{0}$, real numbers $b_{1},\cdots, b_{t}$, and rational functions $\phi_{1},\cdots, \phi_{t}$ on $F_{0}$ such that $D|_{F_{0}}=\sum_{i=1}^{n}r_{i}D_{i}|_{F_{0}}=\sum_{j=1}^{s}a_{j}E_{j}+\sum_{k=1}^{t}b_{k}\cdot{\rm div}(\phi_{k})$ as $\mathbb{R}$-divisors. 
We consider the set 
\begin{equation*}
\Biggl\{\boldsymbol{v}'=\bigl((r'_{i})_{i},(a'_{j})_{j},(b'_{k})_{k}\bigr)\in \mathbb{R}^{n}\times (\mathbb{R}_{\geq0})^{s}\times \mathbb{R}^{t}\Bigg{|}\sum_{i=1}^{n}r'_{i}D_{i}|_{F_{0}}=\sum_{j=1}^{s}a'_{j}E_{j}+\sum_{k=1}^{t}b'_{k}\cdot{\rm div}(\phi_{k})\Biggr\},
\end{equation*}
which contains $\boldsymbol{v}:=\bigl((r_{i})_{i},(a_{j})_{j},(b_{k})_{k}\bigr)$. 
Since all $D_{i}|_{F_{0}}$ are well-defined as Weil divisors, by an argument of convex geometry, we can find positive real numbers $\alpha_{1},\cdots, \alpha_{l_{0}}$ and rational points $\boldsymbol{v}_{1},\cdots, \boldsymbol{v}_{l_{0}}$ in the above set such that 
$\sum_{l=1}^{l_{0}}\alpha_{l}=1$ and $\sum_{l=1}^{l_{0}}\alpha_{l}\boldsymbol{v}_{l}=\boldsymbol{v}$. 
This shows there are $\mathbb{Q}$-Cartier divisors $D^{(1)},\cdots ,D^{(l_{0})}$ on $X$ such that $\sum_{l=1}^{l_{0}}\alpha_{l}D^{(l)}=D$ and $\kappa(F_{0},D^{(l)}|_{F_{0}})\geq0$ for any $1\leq l\leq l_{0}$. 
By the argument in the previous paragraph,  for any $1\leq l\leq l_{0}$, there is a $\mathbb{Q}$-divisor $E^{(l)}\geq0$ on $X$ such that $D^{(l)}\sim_{\mathbb{Q},Z}E^{(l)}$. 
Putting $E$ as $E=\sum_{l=1}^{l_{0}}\alpha_{l}E^{(l)}$, we have $D\sim_{\mathbb{R},Z}E\geq0$. 

Thus we obtain $\kappa_{\iota}(X/Z,D)\geq0$ assuming $\kappa_{\iota}(F_{0},D|_{F_{0}})\geq0$, so Lemma \ref{lem--abundant-general} holds. 
\end{proof}

\begin{lem}\label{lem--iitakafib} Let $(X,\Delta)$ be a projective lc pair with an $\mathbb{R}$-divisor $\Delta$. 
Suppose that $K_{X}+\Delta$ is abundant and there is an effective $\mathbb{R}$-divisor $D\sim_{\mathbb{R}}K_{X}+\Delta$. 
Let $X\dashrightarrow V$ be the Iitaka fibration associated to $D$. 
Pick a log resolution $f\colon Y\to X$ of $(X,\Delta)$ such that the induced map $Y\dashrightarrow V$ is a morphism, and let $(Y,\Gamma)$ be a projective lc pair such that we can write $K_{Y}+\Gamma=f^{*}(K_{X}+\Delta)+E$ for an effective $f$-exceptional divisor $E$. 

Then, we have $\kappa_{\sigma}(Y/V,K_{Y}+\Gamma)=0$. 
\end{lem}

\begin{proof}
We denote the morphism $Y\to V$ by $\psi$. 
It is clear that $\kappa_{\sigma}(Y/V,K_{Y}+\Gamma)\geq0$. 
We prove the inverse inequality $\kappa_{\sigma}(Y/V,K_{Y}+\Gamma)\leq0$. 
Since the map $X\dashrightarrow V$ is the Iitaka fibration associated to $D$, there is an ample $\mathbb{Q}$-divisor $A$ on $V$ and an effective $\mathbb{R}$-divisor $E'$ on $Y$ such that $f^{*}D=\psi^{*}A+E'$. 
Then we have $K_{Y}+\Gamma+k\psi^{*}A\sim_{\mathbb{R}}(1+k)\psi^{*}A+E'+E$ for any $k>0$. 
By Remark \ref{remdiv}, we have 
\begin{equation*}
\kappa_{\sigma}(Y,K_{Y}+\Gamma+k\psi^{*}A)=\kappa_{\sigma}(Y,K_{Y}+\Gamma)={\rm dim}V
\end{equation*}
for any $k>0$. 
Let $g\colon \bar{V}\to V$ be a resolution of $V$. 
Replacing $(Y,\Gamma)$ by a higher model, we may assume that the induced map $\bar{\psi}\colon Y\dashrightarrow \bar{V}$ is a morphism. 
Fix an integer $k>0$ such that $K_{\bar{V}}+kg^{*}A$ is big. 
Applying the inequality in \cite[(3.3)]{fujino-subadd-correct} to $\bar{\psi}\colon Y\to \bar{V}$ (Put $D=K_{Y}-\bar{\psi}^{*}K_{\bar{V}}+\Gamma$ and  $Q=K_{\bar{V}}+kg^{*}A$ in the notations of \cite[(3.3)]{fujino-subadd-correct}), we obtain
\begin{equation*}
\begin{split}
\kappa_{\sigma}(Y,K_{Y}+\Gamma+k\bar{\psi}^{*}g^{*}A)&\geq \kappa_{\sigma}(Y/\bar{V},K_{Y}-\bar{\psi}^{*}K_{\bar{V}}+\Gamma)+\kappa(\bar{V},K_{\bar{V}}+kg^{*}A)\\
&=\kappa_{\sigma}(Y/\bar{V},K_{Y}+\Gamma)+{\rm dim}\bar{V}
=\kappa_{\sigma}(Y/V,K_{Y}+\Gamma)+{\rm dim}V. 
\end{split}
\end{equation*}
We note that we can apply the inequality in \cite[(3.3)]{fujino-subadd-correct} even if $\Gamma$ is an $\mathbb{R}$-divisor. 
Since we have $\bar{\psi}^{*}g^{*}A=\psi^{*}A$ and $\kappa_{\sigma}(Y,K_{Y}+\Gamma+k\psi^{*}A)=\kappa_{\sigma}(Y,K_{Y}+\Gamma)={\rm dim}V$, we obtain $\kappa_{\sigma}(Y/V,K_{Y}+\Gamma)\leq0$. 
So $\kappa_{\sigma}(Y/V,K_{Y}+\Gamma)=0$, and Lemma \ref{lem--iitakafib} holds. 
\end{proof}

\begin{rem}\label{rem--iitakafib}
Quite recently, Lesieutre \cite{lesieutre} gave an example of $\mathbb{R}$-divisor $D$ such that $\kappa_{\sigma}(D)\neq \kappa_{\nu}(D)$, where $\kappa_{\nu}(\;\cdot\;)$ is an other notion of numerical dimension defined by numerical domination. 
So Lemma \ref{lem--iitakafib} does not directly follow from \cite[Theorem 6.1]{leh}. 
Fortunately, argument by Nakayama \cite[Section V]{nakayama} works in the proof of Lemma \ref{lem--iitakafib}, and therefore \cite[V, 4.2 Corollary]{nakayama} holds true (see \cite{fujino-subadd-correct} for details). 
By \cite[(3.3)]{fujino-subadd-correct}, results of the minimal model theory proved with Lemma \ref{lem--iitakafib} (in particular, \cite[Theorem 4.3]{gongyolehmann} and results in \cite{has-mmp}) can be recovered without any troubles.
\end{rem}

The following statement is the $\mathbb{R}$-divisor version of \cite[Theorem 4.3]{gongyolehmann}. 

\begin{lem}\label{lem--abundantmmp}
Let $\pi\colon X\to Z$ be a projective morphism of normal quasi-projective varieties, and let $(X,\Delta)$ be a klt pair. 
If $K_{X}+\Delta$ is abundant over $Z$, then $(X,\Delta)$ has a good minimal model or a Mori fiber space over $Z$. 
\end{lem}

\begin{proof}
We may assume that $K_{X}+\Delta$ is pseudo-effective over $Z$ and $\pi$ is a contraction. 
Then the statement follows by Lemma \ref{lem--iitakafib}, \cite[Proposition 3.3]{has-mmp} and taking the relative Iitaka fibration over $Z$ associated to an effective $\mathbb{R}$-divisor $D\sim_{\mathbb{R},Z}K_{X}+\Delta$.  
\end{proof}

\subsection{Results related to log MMP}
In this subsection, we collect known results on log MMP or existence of good minimal models, and we prove two lemmas. 

\begin{thm}[cf.~{\cite[Theorem 4.1]{birkar-flip}}]\label{thmtermi}
Let $(X,\Delta)$ be a $\mathbb{Q}$-factorial lc pair such that $(X,0)$ is klt, and let $\pi\colon X \to Z$ be a projective morphism of normal quasi-projective varieties. 
If there exists a log minimal model of $(X,\Delta)$ over $Z$, then any $(K_{X}+\Delta)$-MMP over $Z$ with scaling of an ample divisor terminates. 
\end{thm}

\begin{lem}[{\cite[Lemma 2.15]{has-mmp}}]\label{lembirequiv}
Let $\pi\colon X \to Z$ be a projective morphism of normal quasi-projective varieties, and let $(X,\Delta)$ be an lc pair. 
Let $(Y,\Gamma)$ be an lc pair with a projective birational morphism $f\colon Y\to X$ such that $K_{Y}+\Gamma=f^{*}(K_{X}+\Delta)+E$ for an effective $f$-exceptional divisor $E$. 

Then, $(X,\Delta)$ has a weak lc model (resp.~a log minimal model, a good minimal model) over $Z$ if and only if  $(Y,\Gamma)$ has a weak lc model (resp.~a log minimal model, a good minimal model) over $Z$. 
\end{lem}

\begin{lem}\label{lem--extraction}
Let $X$ be a normal quasi-projective variety, and let $\mathcal{T}$ be an empty set or a finite set of exceptional prime divisors over $X$.  
Suppose that there is an $\mathbb{R}$-divisor $\Delta$ on $X$ such that $(X,\Delta)$ is dlt, and suppose in addition that if $\mathcal{T}\neq \emptyset$ then we have $-1< a(P_{i},X,\Delta)< 0$ for all $P_{i}\in \mathcal{T}$. 

Then, there is a crepant model $f\colon (X_{\mathcal{T}},\Delta_{\mathcal{T}})\to (X,\Delta)$ such that the pair $(X_{\mathcal{T}},\Delta_{\mathcal{T}})$ is $\mathbb{Q}$-factorial dlt  and $f$-exceptional prime divisors are exactly elements of $\mathcal{T}$. 
\end{lem}

\begin{proof} It immediately follows from \cite[Corollary 1.37]{kollar-mmp} and \cite[Corollary 1.38]{kollar-mmp}. \end{proof}

\begin{lem}\label{lem--abundant-birat-2}
Assume the existence of good minimal models or Mori fiber spaces for all projective klt pairs of dimension $n$. 

Let $\pi\colon X\to Z$ be a projective morphism of normal quasi-projective varieties, and let $(X,\Delta)$ be an lc pair such that $K_{X}+\Delta$ is abundant over $Z$ and $\kappa_{\sigma}(X/Z,K_{X}+\Delta)\leq n$. 
Let $Y\to Z$ be a projective morphism from a normal quasi-projective variety $Y$, and let $X\dashrightarrow Y$ be a birational map over $Z$. 
Let $(Y,\Gamma)$ be an lc pair such that inequality $a(P,X,\Delta)\leq a(P,Y,\Gamma)$ holds for any prime divisor $P$ on $X$. 

Then, $K_{Y}+\Gamma$ is abundant over $Z$ and $\kappa_{\sigma}(Y/Z,K_{Y}+\Gamma)\leq n$. 
\end{lem}

\begin{proof}
Let $g\colon Y_{0}\to Y$ be a log resolution of $(Y,\Gamma)$ such that the induced birational map $f\colon Y_{0}\dashrightarrow X$ is a morphism. 
We may write $K_{Y_{0}}+\Gamma_{0}=g^{*}(K_{Y}+\Gamma)+E_{0}$ with $\Gamma_{0}\geq0$ and $E_{0}\geq0$ which have no common components. 
Since any prime divisor $P$ on $X$ is also a prime divisor on $Y_{0}$, we see that $a(P,Y_{0},\Gamma_{0})={\rm min}\{0, a(P,Y,\Gamma)\}$ by Lemma \ref{lem--discre}. 
Since $a(P,X,\Delta)\leq0$, using the hypothesis of Lemma \ref{lem--abundant-birat-2} we have $a(P,X,\Delta)\leq a(P,Y_{0},\Gamma_{0})$ for any prime divisor $P$ on $X$. 
Therefore, $(Y_{0},\Gamma_{0})$ satisfies the condition of Lemma \ref{lem--abundant-birat-2}. 
It is easy to check that we may replace $(Y,\Gamma)$ with $(Y_{0},\Gamma_{0})$. 
Replacing $(Y,\Gamma)$, we may assume that $(Y,\Gamma)$ is $\mathbb{Q}$-factorial dlt and the map $f\colon Y\dashrightarrow X$ is a morphism. 

If $K_{Y}+\Gamma$ is not pseudo-effective over $Z$, then there is nothing to prove.  
So we may assume that $K_{Y}+\Gamma$ is pseudo-effective over $Z$. 
Put $M=K_{Y}+\Gamma-f^{*}(K_{X}+\Delta)$. 
By the condition on discrepancies of prime divisors on $X$, we see that the effective part of $M$ is $f$-exceptional. 
We run a $(K_{Y}+\Gamma)$-MMP over $X$ with scaling of an ample divisor. 
By the negativity lemma of very exceptional divisors \cite[Theorem 3.5]{birkar-flip}, we reach a model $(Y',\Gamma')$ over $X$ such that $M'\leq0$, where $M'$ is the birational transform of $M$ on $Y'$. 
By construction, $K_{Y'}+\Gamma'-M'$ is the pullback of $K_{X}+\Delta$, so the pair $(Y', \Gamma'-M')$ is lc. 
In this way, replacing $(X,\Delta)$ and $(Y,\Gamma)$ by $(Y', \Gamma'-M')$ and $(Y',\Gamma')$ respectively, we may assume $X=Y$. 
Then $\Delta \geq \Gamma$ and $K_{X}+\Delta$ is pseudo-effective over $Z$. 

Let $F$ be a sufficiently general fiber of the Stein factorization of $\pi$. 
Since $K_{X}+\Delta$ is abundant over $Z$ and $\kappa_{\sigma}(X/Z,K_{X}+\Delta)\leq n$, the restriction $K_{F}+\Delta|_{F}=(K_{X}+\Delta)|_{F}$ is abundant and $\kappa_{\sigma}(F,K_{F}+\Delta|_{F})\leq n$ by Lemma \ref{lem--abundant-general}. 
Applying Lemma \ref{lem--abundant-general} again, it is sufficient to prove that $K_{F}+\Gamma|_{F}=(K_{X}+\Gamma)|_{F}$ is abundant and $\kappa_{\sigma}(F,K_{F}+\Gamma|_{F})\leq n$. 
Therefore, by restricting $(X,\Delta)$ and $(X,\Gamma)$ to $F$, we may assume that $Z$ is a point. 

We have $\kappa_{\iota}(X,K_{X}+\Delta)=\kappa_{\sigma}(X,K_{X}+\Delta)\geq 0$ by our assumptions. 
So there is an effective $\mathbb{R}$-divisor $D\sim_{\mathbb{R}}K_{X}+\Delta$. 
We take the Iitaka fibration $X\dashrightarrow V$ of $D$, and let $\bar{f}\colon \bar{X} \to X$ be a log resolution of $(X,\Delta)$ such that the induced map $\bar{X}\dashrightarrow V$ is a morphism. 
We may write 
\begin{equation*}
K_{\bar{X}}+\bar{\Delta}=\bar{f}^{*}(K_{X}+\Delta)+E_{1} \quad{\rm and} \quad K_{\bar{X}}+\bar{\Gamma}=\bar{f}^{*}(K_{X}+\Gamma)+E_{2},
\end{equation*}
where $\bar{\Delta}\geq 0$ and $E_{1}\geq0$ have no common components, and $\bar{\Gamma}\geq0$ and $E_{2}\geq0$ have no common components. 
By Lemma \ref{lem--iitakafib}, we see that $\kappa_{\sigma}(\bar{X}/V,K_{\bar{X}}+\bar{\Delta})=0$. 
We also see that ${\rm dim}V=\kappa_{\sigma}(\bar{X},K_{\bar{X}}+\bar{\Delta})\leq n$ and $\bar{\Gamma}\leq \bar{\Delta}$ by construction, and $K_{\bar{X}}+\bar{\Gamma}$ is pseudo-effective. 
So we have $\kappa_{\sigma}(\bar{X}/V,K_{\bar{X}}+\bar{\Gamma})=0$. 
From this construction, to prove Lemma \ref{lem--abundant-birat-2}, it is sufficient to prove the following claim:
\begin{claim*}
Assume the existence of good minimal models or Mori fiber spaces for all projective klt pairs of dimension $n$. 
Let $(\bar{X},\bar{\Gamma})$ be a projective lc pair with a contraction $\bar{X}\to V$ such that ${\rm dim}V\leq n$ and $\kappa_{\sigma}(\bar{X}/V, K_{\bar{X}}+\bar{\Gamma})=0$. 
If $K_{\bar{X}}+\bar{\Gamma}$ is pseudo-effective, then $(\bar{X},\bar{\Gamma})$ has a good minimal model. 
\end{claim*}
Thanks to \cite[Lemma 3.1]{has-mmp}, we only have to prove the claim in a special case when $K_{\bar{X}}+\bar{\Gamma}\sim_{\mathbb{R},V}0$. 
But the case of the claim follows from \cite[Theorem 1.5]{has-mmp}. 
\end{proof}

\begin{lem}[cf.~{\cite[Proposition 3.2 (5)]{birkar-existII}}]\label{lem--polytope-1}
Let $\pi\colon X\to Z$ be a morphism of normal projective varieties, and let $(X,\Delta_{0})$ be a $\mathbb{Q}$-factorial lc pair such that $K_{X}+\Delta_{0}$ is nef over $Z$. 
Let $\mathcal{W}\ni \Delta_{0}$ be a finite-dimensional $\mathbb{R}$-vector space in ${\rm WDiv}_{\mathbb{R}}(X)$ spanned by prime divisors. 
Let $\mathcal{L}\ni \Delta_{0}$ be a rational polytope in $\mathcal{W}$ such that for any $\Delta \in \mathcal{L}$ the pair $(X,\Delta)$ is lc. 
Fix a norm $||\cdot||$ in $\mathcal{W}$. 

Then, there is a real number $\gamma>0$ such that $\gamma$ satisfies the following property:

\begin{itemize}
\item
Pick any $\Delta \in \mathcal{L}$ such that $||\Delta-\Delta_{0}||<\gamma$. 
Let $(X,\Delta)\dashrightarrow (Y,\Gamma)$ be a sequence of steps of a $(K_{X}+\Delta)$-MMP over $Z$ to a good minimal model, and let $Y\to V$ be the contraction over $Z$ induced by $K_{Y}+\Gamma$. 
Then $K_{Y}+\Gamma_{0}$ is nef over $Z$ and $K_{Y}+\Gamma_{0}\sim_{\mathbb{R},V}0$, where $\Gamma_{0}$ is the birational transform of $\Delta_{0}$ on $Y$. 
\end{itemize}
\end{lem}

\begin{proof}
By the standard argument of length of extremal rays (\cite[Proposition 3.2]{birkar-existII}, see also \cite[Proof of Lemma 2.12]{has-mmp} and \cite[Remark 2.13]{has-mmp}), we can find $\gamma_{0}>0$ such that for any $\Delta \in \mathcal{L}$ satisfying $||\Delta-\Delta_{0}||<\gamma_{0}$ and for all steps of any $(K_{X}+\Delta)$-MMP over $Z$, the birational transform of $K_{X}+\Delta_{0}$ is trivial over each extremal contraction.

By the argument of Shokurov polytopes, we can find $\mathbb{Q}$-divisors $\Delta^{(1)},\cdots,\Delta^{(m)}$ and positive real numbers $\alpha_{1},\cdots, \alpha_{m}$ such that $\sum_{i=1}^{m}\alpha_{i}=1$, $\sum_{i=1}^{m}\alpha_{i}\Delta^{(i)}=\Delta_{0}$, $(X,\Delta^{(i)})$ are lc and $K_{X}+\Delta^{(i)}$ are nef over $Z$ for all $1\leq i\leq m$.  
Put $d={\rm dim}X$ and $\alpha={\rm min} \{\alpha_{i}\}_{i}$, and let $k$ be a positive integer such that $k(K_{X}+\Delta^{(i)})$ are all Cartier.  
We pick $\gamma\in (0,\gamma_{0}]$ such that for any $\Delta \in \mathcal{L}$ satisfying $||\Delta-\Delta_{0}||<\gamma$, we have $\Delta_{0}+(1+\frac{3dk}{\alpha})(\Delta-\Delta_{0})\in \mathcal{L}$. 
We show this $\gamma$ satisfies the condition of the lemma. 

Pick any $\Delta\in \mathcal{L}$ such that $||\Delta-\Delta_{0}||<\gamma$, and let $(X,\Delta)\dashrightarrow (Y,\Gamma)$ be a sequence of steps of a $(K_{X}+\Delta)$-MMP over $Z$ to a good minimal model. 
Since $\gamma\leq \gamma_{0}$, the birational transform of $K_{X}+\Delta_{0}$ is trivial over each extremal contraction of the $(K_{X}+\Delta)$-MMP over $Z$. 
We put $\Delta'=\Delta_{0}+(1+\frac{3dk}{\alpha})(\Delta-\Delta_{0})$. 
Let $\Gamma_{0}$, $\Gamma'$ and $\Gamma^{(i)}$ be the birational transforms of $\Delta_{0}$, $\Delta'$ and $\Delta^{(i)}$ on $Y$, respectively. 
The argument of length of extremal rays (\cite[Proposition 3.2]{birkar-existII}, \cite[Proof of Lemma 2.12]{has-mmp} and \cite[Remark 2.13]{has-mmp}) also shows that all $(Y,\Gamma^{(i)})$ are lc and all $k(K_{Y}+\Gamma^{(i)})$ are nef over $Z$ and Cartier. 
Therefore $K_{Y}+\Gamma_{0}=\sum_{i=1}^{m}\alpha_{i}(K_{Y}+\Gamma^{(i)})$ is nef over $Z$. 
Moreover, for any curve $\xi$ on $Y$ contracted by the morphism $Y\to Z$, if an inequality $(K_{Y}+\Gamma_{0})\cdot \xi>0$ holds then $(K_{Y}+\Gamma_{0})\cdot \xi\geq \frac{\alpha}{k}$. 

Let $Y\to V$ be the contraction over $Z$ induced by $K_{Y}+\Gamma$. 
Since $K_{Y}+\Gamma\sim_{\mathbb{R},V}0$, we have $K_{Y}+\Gamma'\sim_{\mathbb{R},V}-\frac{3dk}{\alpha}(K_{Y}+\Gamma_{0})$. 
We prove $K_{Y}+\Gamma'$ is nef over $V$. 
Since the birational transform of $K_{X}+\Delta_{0}$ is trivial over each extremal contraction of the $(K_{X}+\Delta)$-MMP over $Z$, the birational map $X\dashrightarrow Y$ is a sequence of steps of the $(K_{X}+\Delta')$-MMP over $Z$. 
So $(Y,\Gamma')$ is lc. 
If $K_{Y}+\Gamma'$ is not nef over $V$, by \cite[Theorem 4.5.2]{fujino-book}, there is a curve $\xi_{0}$ contracted by the morphism $Y\to V$ such that $0<-(K_{Y}+\Gamma')\cdot \xi_{0} \leq 2d$. 
On the other hand, the inequality $(K_{Y}+\Gamma')\cdot \xi_{0}<0$ implies $(K_{Y}+\Gamma_{0})\cdot \xi_{0}>0$, and therefore $-(K_{Y}+\Gamma')\cdot \xi_{0}=\frac{3dk}{\alpha}(K_{Y}+\Gamma_{0})\cdot \xi_{0}\geq3d$ because $(K_{Y}+\Gamma_{0})\cdot \xi_{0}>0$ implies $(K_{Y}+\Gamma_{0})\cdot \xi_{0}\geq \frac{\alpha}{k}$. 
Then $3d\leq-(K_{Y}+\Gamma')\cdot \xi_{0} \leq 2d$, a contradiction. 
In this way, we see that $K_{Y}+\Gamma'$ is nef over $V$. 
Since $\frac{3dk}{\alpha}(K_{Y}+\Gamma_{0})\sim_{\mathbb{R},V}-(K_{Y}+\Gamma')$ and since $K_{Y}+\Gamma_{0}$ and $K_{Y}+\Gamma'$ are both nef over $V$, the divisor $K_{Y}+\Gamma_{0}$ is numerically trivial over $V$. 

By \cite[Lemma 3.4]{has-class}, $K_{Y}+\Gamma_{0}$ is semi-ample over $V$. 
So we have $K_{Y}+\Gamma_{0}\sim_{\mathbb{R},V}0$. 
\end{proof}

\section{Proof of Theorem \ref{thmmain}}\label{sec3}
 
\begin{proof}[Proof of Theorem \ref{thmmain}]
We prove Theorem \ref{thmmain} by induction on the dimension of $X$. 
We may assume that $K_{X}+\Delta$ is pseudo-effective over $Z$. 
By taking the Stein factorization of $\pi$, we may assume that $\pi$ is a contraction. 
We prove Theorem \ref{thmmain} in several steps. 
From Step \ref{step1relative} to Step \ref{step9relative} we prove that $(X,\Delta)$ has a log minimal model over $Z$, and in Step \ref{step10relative} we prove that $(X,\Delta)$ has a good minimal model over $Z$. 

\begin{step1}\label{step1relative}
In this step, we reduce Theorem \ref{thmmain} to the case where $X$ and $Z$ are projective. 

Let $Z\hookrightarrow Z^{c}$ be an open immersion to a normal projective variety $Z^{c}$. 
Thanks to \cite[Corollary 1.3]{has-mmp}, there is an lc closure $(X^{c},\Delta^{c})$ of $(X,\Delta)$, that is, a projective lc pair $(X^{c},\Delta^{c})$ such that $X$ is an open subset of $X^{c}$ and $(X^{c}|_{X},\Delta^{c}|_{X})=(X,\Delta)$, and there is a projective morphism $\pi^{c}\colon X^{c}\to Z^{c}$. 
By construction of lc closures, we have $\pi^{c}|_{X}=\pi$ and ${\pi^{c}}^{-1}(Z)=X$. 
Furthermore, we can construct $(X^{c},\Delta^{c})$ so that any lc center $S^{c}$ of $(X^{c},\Delta^{c})$ intersects $X$ (see \cite[Corollary 1.3]{has-mmp}). 
Then the morphism $(X^{c},\Delta^{c})\to Z^{c}$ satisfies hypothesis of Theorem \ref{thmmain} because the relative invariant Iitaka dimension and the relative numerical dimension of $\mathbb{R}$-Cartier divisors are determined by its restriction to a sufficiently general fiber of the Stein factorization. 
If $(X^{c},\Delta^{c})$ has a good minimal model over $Z^{c}$, then its restriction over $Z$ is a good minimal model over $Z$. 

In this way, we may replace $(X,\Delta)\to Z$ by $(X^{c},\Delta^{c})\to Z^{c}$, and therefore we may assume that $X$ and $Z$ are projective. 
\end{step1}

\begin{step1}\label{step2relative}
In this step and the next step, we replace $(X,\Delta)$ with a crepant model so that $K_{X}+\Delta$ is $\mathbb{R}$-linearly equivalent to an effective $\mathbb{R}$-divisor which has good properties. 
The idea is very similar to \cite[Step 2--4 in the proof of Theorem 1.2]{has-mmp}. 

In this step, we construct a rational map associated to $K_{X}+\Delta$ and study properties of an lc pair on a resolution of the rational map. 
Pick an effective $\mathbb{R}$-divisor $D$ such that $K_{X}+\Delta\sim_{\mathbb{R},Z}D$, and take the relative Iitaka fibration $X\dashrightarrow V$ over $Z$ associated to $D$. 
We have ${\rm dim}V-{\rm dim}Z=\kappa_{\sigma}(X/Z, K_{X}+\Delta)$. 
Let $\bar{f}\colon \bar{X}\to X$ be a log resolution of $(X,\Delta)$ such that the induced map  $\bar{X}\dashrightarrow V$ is a morphism. 
We can construct a log smooth lc pair $(\bar{X},\bar{\Delta})$ such that 
\begin{itemize}
\item[(i)]
$K_{\bar{X}}+\bar{\Delta}=\bar{f}^{*}(K_{X}+\Delta)+\bar{E}$ for an effective $\bar{f}$-exceptional divisor $\bar{E}$, and 
\item[(ii)]
$\kappa_{\sigma}(\bar{X}/Z,K_{\bar{X}}+\bar{\Delta})={\rm dim}V-{\rm dim}Z$ and $\kappa_{\sigma}(\bar{X}/V,K_{\bar{X}}+\bar{\Delta})=0$ (Lemma \ref{lem--iitakafib}).
\end{itemize}
By construction of $X\dashrightarrow V$ and $(\bar{X},\bar{\Delta})$, the divisor $K_{\bar{X}}+\bar{\Delta}$ is $\mathbb{R}$-linearly equivalent to the sum of an effective $\mathbb{R}$-divisor and the pullback of a relatively ample divisor on $V$. 
So we can find an effective $\mathbb{R}$-divisor $\bar{D}\sim_{\mathbb{R},Z}K_{\bar{X}}+\bar{\Delta}$ such that ${\rm Supp}\bar{D}$ contains all lc centers of $(\bar{X},\bar{\Delta})$ which are vertical over $V$. 
By taking a log resolution of $(\bar{X},{\rm Supp}(\bar{\Delta}+\bar{D}))$ and by replacing $(\bar{X},\bar{\Delta})$ and $\bar{D}$, we may assume that $(\bar{X},{\rm Supp}(\bar{\Delta}+\bar{D}))$ is log smooth. 
By applying the argument as in \cite[Proof of Lemma 2.10]{has-trivial} to the morphism $(\bar{X},\bar{\Delta})\to V$ and by replacing $(\bar{X},\bar{\Delta})$ and $\bar{D}$ again, we may assume that
\begin{itemize}
\item[(iii)]
$\bar{\Delta}=\bar{\Delta}'+\bar{\Delta}''$, where $\bar{\Delta}'$ is effective and $\bar{\Delta}''$ is a reduced divisor, such that $\bar{\Delta}''$ is vertical over $V$ and all lc centers of $(\bar{X},\bar{\Delta}-\bar{\Delta}'')$ dominate $V$. 
\end{itemize}
Note that we may have $\bar{\Delta}''=0$. 
We have ${\rm Supp}\bar{D}\supset {\rm Supp}\bar{\Delta}''$ since ${\rm Supp}\bar{D}$ contains all lc centers of $(\bar{X},\bar{\Delta})$ which are vertical over $V$. 
By decomposing $\bar{D}$ appropriately, we obtain effective $\mathbb{R}$-divisors $\bar{G}$ and $\bar{H}$ such that 
\begin{itemize}
\item[(iv)]
$K_{\bar{X}}+\bar{\Delta}\sim_{\mathbb{R},Z}\bar{G}+\bar{H}$, 
\item[(v)]
${\rm Supp}\bar{\Delta}''\subset {\rm Supp}\bar{G}\subset {\rm Supp}\llcorner \bar{\Delta} \lrcorner$, and 
\item[(vi)]
no component of $\bar{H}$ is a component of $\llcorner \bar{\Delta} \lrcorner$.   
\end{itemize}
We fix a real number $t_{0}\in(0,1)$ such that $\bar{\Delta}-t_{0}\bar{G}\geq 0$. 
For any $t\in(0,t_{0}]$, we consider the pair $(\bar{X},\bar{\Delta}-t\bar{G})$. 
By conditions (iii) and (v), we see that any lc center of $(\bar{X},\bar{\Delta}-t\bar{G})$ dominates $V$. 
Since we have $K_{\bar{X}}+\bar{\Delta}\sim_{\mathbb{R},Z}\bar{G}+\bar{H}$ and $K_{\bar{X}}+\bar{\Delta}-t\bar{G}\sim_{\mathbb{R},Z}(1-t)\bar{G}+\bar{H}$, by applying  Remark \ref{remdiv} (1), we obtain 
$\kappa_{\sigma}(\bar{X}/Z,K_{\bar{X}}+\bar{\Delta}-t\bar{G})=\kappa_{\sigma}(\bar{X}/Z,K_{\bar{X}}+\bar{\Delta}).$ 
Combining this with condition (ii), we obtain $\kappa_{\sigma}(\bar{X}/Z,K_{\bar{X}}+\bar{\Delta}-t\bar{G})={\rm dim}V-{\rm dim}Z$. 
Similarly, we obtain $\kappa_{\sigma}(\bar{X}/V,K_{\bar{X}}+\bar{\Delta}-t\bar{G})=0$. 
Thus the morphisms $(\bar{X},\bar{\Delta}-t\bar{G})\to V$ and $V\to Z$ satisfy all conditions of \cite[Proposition 3.3]{has-mmp}. 
By \cite[Proposition 3.3]{has-mmp}, the pair $(\bar{X},\bar{\Delta}-t\bar{G})$ has a good minimal model over $Z$ for any $t\in(0,t_{0}]$. 
\end{step1}

\begin{step1}\label{step3relative}
We use the notations and conditions (i)--(vi) in Step \ref{step2relative}. 

We run a $(K_{\bar{X}}+\bar{\Delta})$-MMP over $X$, and we get a morphism $\widetilde{f}\colon(\widetilde{X},\widetilde{\Delta})\to (X,\Delta)$ such that $K_{\widetilde{X}}+\widetilde{\Delta}=\widetilde{f}^{*}(K_{X}+\Delta)$. 
Let $\widetilde{G}$ and $\widetilde{H}$ be the birational transforms of $\bar{G}$ and $\bar{H}$ on $\widetilde{X}$, respectively. 
By condition (iv), we have $K_{\widetilde{X}}+\widetilde{\Delta}\sim_{\mathbb{R},Z}\widetilde{G}+\widetilde{H}$. 
By condition (vi) and because $(\bar{X},{\rm Supp}(\bar{\Delta}+\bar{D}))$ is log smooth, by replacing  $t_{0}$ with a smaller one, we may assume that $(\bar{X},\bar{\Delta}+t_{0}\bar{H})$ is dlt. 
By construction of the map $\bar{X}\dashrightarrow \widetilde{X}$, by replacing $t_{0}$ again, we may assume that $\bar{X}\dashrightarrow \widetilde{X}$ is a sequence of steps of a $(K_{\bar{X}}+\bar{\Delta}+t\bar{H})$-MMP and a sequence of steps of a $(K_{\bar{X}}+\bar{\Delta}-t\bar{G})$-MMP for any $t\in(0,t_{0}]$. 
Then $(\widetilde{X},\widetilde{\Delta}+t_{0}\widetilde{H})$ is dlt, and the pairs $(\widetilde{X},\widetilde{\Delta}-t\widetilde{G})$ and $(\bar{X},\bar{\Delta}-t\bar{G})$ have the same good minimal model over $Z$ for all $t\in(0,t_{0}]$. 
By running a $(K_{\widetilde{X}}+\widetilde{\Delta}-t\widetilde{G})$-MMP over $Z$, we obtain a good minimal model $(\widetilde{X}_{t},\widetilde{\Delta}_{t}-t\widetilde{G}_{t})$ of $(\widetilde{X},\widetilde{\Delta}-t\widetilde{G})$ and $(\bar{X},\bar{\Delta}-t\bar{G})$. 
Let $\widetilde{X}_{t}\to \widetilde{V}_{t}$ be the contraction over $Z$ induced by $K_{\widetilde{X}_{t}}+\widetilde{\Delta}_{t}-t\widetilde{G}_{t}$. 
We have the following diagrams:
\begin{equation*}
\xymatrix@R=12pt
{
(\bar{X},\bar{\Delta})\ar[d]\ar[dr]_(0.4){\bar{f}}\ar@{-->}[r]&(\widetilde{X},\widetilde{\Delta})\ar[d]^{\widetilde{f}}&(\bar{X},\bar{\Delta}-t\bar{G})\ar[d]\ar@{-->}[r]& (\widetilde{X},\widetilde{\Delta}-t\widetilde{G})\ar[dd]\ar@{-->}[r]&(\widetilde{X}_{t},\widetilde{\Delta}_{t}-t\widetilde{G}_{t})\ar[d]\\
V\ar[dr]&(X,\Delta)\ar[d]\ar@{-->}[l]&V\ar[dr]&&\widetilde{V}_{t}\ar[dl]\\
&Z&&Z
}
\end{equation*}
We recall that the morphisms $(\bar{X},\bar{\Delta}-t\bar{G})\to V$ and $V\to Z$ satisfies all conditions of \cite[Proposition 3.3]{has-mmp}, so the morphism $(\widetilde{X}_{t},\widetilde{\Delta}_{t}-t\widetilde{G}_{t})\to\widetilde{V}_{t}$ has a property stated in \cite[Proposition 3.3]{has-mmp}, that is, any lc center of $(\widetilde{X}_{t},\widetilde{\Delta}_{t}-t\widetilde{G}_{t})$ dominates $\widetilde{V}_{t}$. 

It is easy to check that we can replace $(X,\Delta)$ with $(\widetilde{X},\widetilde{\Delta})$. 
In this way, by replacing $(X,\Delta)$ we may assume that there is an effective $\mathbb{R}$-divisors $G$ and $H$ such that 
\begin{itemize}
\item[(I)]
$K_{X}+\Delta\sim_{\mathbb{R},Z}G+H$, 
\item[(II)]
${\rm Supp}G\subset {\rm Supp}\llcorner \Delta \lrcorner$, and
\item[(III)]
there is a real number $t_{0}>0$ such that for any $t\in(0,t_{0}]$, followings hold true: 
\begin{itemize}
\item[(III-a)]
The pair $(X,\Delta+tH)$ is $\mathbb{Q}$-factorial dlt, and 
\item[(III-b)]
there is a sequence of steps of a $(K_{X}+\Delta-tG)$-MMP over $Z$ terminating with a good minimal model
$(\widetilde{X}_{t},\widetilde{\Delta}_{t}-t\widetilde{G}_{t})$, and the contraction $\widetilde{X}_{t}\to \widetilde{V}_{t}$ over $Z$ induced by $K_{\widetilde{X}_{t}}+\widetilde{\Delta}_{t}-t\widetilde{G}_{t}$ satisfies the property that any lc center of $(\widetilde{X}_{t},\widetilde{\Delta}_{t}-t\widetilde{G}_{t})$ dominates $\widetilde{V}_{t}$. 
\end{itemize}
\end{itemize}
In the rest of the proof, we only use notations and properties in (I), (II), (III), (III-a) and (III-b). 
We may ignore the other notations and conditions in Step \ref{step2relative} and this step. 
\end{step1}

\begin{step1}\label{step4relative}
From this step to Step \ref{step6relative}, we reduce Theorem \ref{thmmain} to the special termination of a log MMP. 
We mimic arguments as in \cite[Step 3--5 in the proof of Theorem 3.5]{has-class}.  

In this step, we construct a strictly decreasing sequence $\{e_{i}\}_{i\geq1}$ of real numbers and a sequence of birational maps
\begin{equation*}
X \dashrightarrow X_{1}\dashrightarrow X_{2}\dashrightarrow \cdots \dashrightarrow X_{i}\dashrightarrow \cdots
\end{equation*}
over $Z$ such that
\begin{enumerate}
\item
$0<e_{i}<t_{0}$ for any $i$ and ${\rm lim}_{i\to \infty}e_{i}=0$, 
\item
the map $X \dashrightarrow X_{1}$ is a sequence of steps of a $(K_{X}+\Delta+e_{1}H)$-MMP over $Z$ to a good minimal model over $Z$, 
\item
for each $i\geq 1$, the pair $(X_{i},\Delta_{i}+e_{i}H_{i})$ is a good minimal model over $Z$ of both $(X,\Delta+e_{i}H)$ and $(X_{1},\Delta_{1}+e_{i}H_{1})$, and
\item
for each $i\geq 1$, the map $X_{i}\dashrightarrow X_{i+1}$ is an isomorphism or a sequence of steps of a $(K_{X_{i}}+\Delta_{i}+e_{i+1}H_{i})$-MMP over $Z$ with scaling of $(e_{i}-e_{i+1})H_{i}$. 
\end{enumerate}
Here, $\Delta_{i}$ and $H_{i}$ are the birational transforms of $\Delta$ and $H$ on $X_{i}$, respectively. 
By (4), the sequence of birational maps $X_{1}\dashrightarrow \cdots$ is a sequence of steps of a $(K_{X_{1}}+\Delta_{1})$-MMP over $Z$ with scaling of $e_{1}H_{1}$. 

By (I) in Step \ref{step3relative}, for any real number $t$ 
we have 
\begin{equation*}\tag{$\spadesuit$}
K_{X}+\Delta+tH\sim_{\mathbb{R},Z}(1+t)\Bigl(K_{X}+\Delta-\frac{t}{1+t}G\Bigr). 
\end{equation*}
Pick a strictly decreasing infinite sequence $\{e_{i}\}_{i\geq1}$ of positive real numbers such that $e_{i}<t_{0}$ for any $i\geq1$ and ${\rm lim}_{i\to \infty}e_{i}=0$. 
Then $(X,\Delta+e_{i}H)$ are $\mathbb{Q}$-factorial dlt. 
For each pair $(X,\Delta-\frac{e_{i}}{1+e_{i}}G)$, there is a sequence of steps of a $(K_{X}+\Delta-\frac{e_{i}}{1+e_{i}}G)$-MMP over $Z$ to a good minimal model as in (III-b) in Step \ref{step3relative}. 
For simplicity of notation, we denote the good minimal model over $Z$ by $(\widetilde{X}_{i},\widetilde{\Delta}_{i}-\frac{e_{i}}{1+e_{i}}\widetilde{G}_{i})$, where $\widetilde{\Delta}_{i}$ and $\widetilde{G}_{i}$ are the birational transforms of $\Delta$ and $G$ on $\widetilde{X}_{i}$, respectively. 
Let $\widetilde{H}_{i}$ be the birational transform of $H$ on $\widetilde{X}_{i}$. 
By ($\spadesuit$), the map
\begin{equation*}
(X,\Delta+e_{i}H)\dashrightarrow(\widetilde{X}_{i},\widetilde{\Delta}_{i}+e_{i}\widetilde{H}_{i})
\end{equation*}
is a sequence of steps of a $(K_{X}+\Delta+e_{i}H)$-MMP over $Z$ and $(\widetilde{X}_{i},\widetilde{\Delta}_{i}+e_{i}\widetilde{H}_{i})$ is a good minimal model of $(X,\Delta+e_{i}H)$ over $Z$. 
By (I) in Step \ref{step3relative}, prime divisors contracted by the log MMP are components of $G+H$. 
By replacing $\{e_{i}\}_{i\geq1}$ with a subsequence, we may assume that all the maps $X \dashrightarrow \widetilde{X}_{i}$ contract the same divisors. 
Then all $\widetilde{X}_{i}$ are isomorphic in codimension one. 

We put $X_{1}=\widetilde{X}_{1}$, $\Delta_{1}=\widetilde{\Delta}_{1}$, $G_{1}=\widetilde{G}_{1}$, and $H_{1}=\widetilde{H}_{1}$. 
By the above argument, it is clear that the sequence $\{e_{i}\}_{i\geq1}$ and the variety $X_{1}$ satisfy conditions (1) and (2) stated at the start of this step. 

We show that $(X_{1},\Delta_{1}+tH_{1})$ has a good minimal model over $Z$ for any $t\in(0,e_{1})$. 
We put $t'=\frac{t}{1+t}$. 
By using ($\spadesuit$), we see that it is sufficient to prove the existence of a good minimal model over $Z$ of the pair $(X_{1},\Delta_{1}-t'G_{1})$. 
Now the pair $(X_{1},\Delta_{1}-\frac{e_{1}}{1+e_{1}}G_{1})$ is a good minimal model of $(X,\Delta-\frac{e_{1}}{1+e_{1}}G)$ over $Z$ as in (III-b) in Step \ref{step3relative}. 
Let $X_{1}\to V_{1}$ be the contraction over $Z$ induced by $K_{X_{1}}+\Delta_{1}-\frac{e_{1}}{1+e_{1}}G_{1}$. 
By (III-b) in Step \ref{step3relative}, all lc centers of $(X_{1},\Delta_{1}-\frac{e_{1}}{1+e_{1}}G_{1})$ dominate $V_{1}$. 
Since $t'>0$ and $(X_{1},\Delta_{1})$ is lc, any lc center of $(X_{1},\Delta_{1}-t'G_{1})$ is an lc center of $(X_{1},\Delta_{1}-\frac{e_{1}}{1+e_{1}}G_{1})$, hence any lc center of $(X_{1},\Delta_{1}-t'G_{1})$ dominates $V_{1}$. 
Moreover, by $(\spadesuit)$ again, we obtain $K_{X_{1}}+\Delta_{1}-\frac{e_{1}}{1+e_{1}}G_{1}\sim_{\mathbb{R},Z}\frac{1}{1+e_{1}}G_{1}+H_{1}$ and $K_{X_{1}}+\Delta_{1}-t'G_{1}\sim_{\mathbb{R},Z}(1-t')G_{1}+H_{1}$. 
By Remark \ref{remdiv} (1), we obtain 
\begin{equation*}
\kappa_{\sigma}(X_{1}/Z,K_{X_{1}}+\Delta_{1}-t'G_{1})=\kappa_{\sigma}(X_{1}/Z,K_{X_{1}}+\Delta_{1}-\tfrac{e_{1}}{1+e_{1}}G_{1})={\rm dim}V_{1}-{\rm dim}Z.
\end{equation*}
By a similar calculation, we obtain $\kappa_{\sigma}(X_{1}/V_{1},K_{X_{1}}+\Delta_{1}-t'G_{1})=0$. 
Then we can apply \cite[Proposition 3.3]{has-mmp} to $(X_{1},\Delta_{1}-t'G_{1})\to V_{1}\to Z$, and we see that $(X_{1},\Delta_{1}-t'G_{1})$ has a good minimal model over $Z$. 
From this fact, we see that $(X_{1},\Delta_{1}+tH_{1})$ has a good minimal model over $Z$ for any $t\in(0,e_{1})$. 

By \cite[Lemma 2.14]{has-mmp}, we can construct a sequence of steps of a $(K_{X_{1}}+\Delta_{1})$-MMP over $Z$ with scaling of $e_{1}H_{1}$
\begin{equation*}
(X_{1},\Delta_{1})=:(X'_{1},\Delta'_{1})\dashrightarrow (X'_{2},\Delta'_{2})\dashrightarrow\cdots \dashrightarrow (X'_{j},\Delta'_{j})\dashrightarrow \cdots
\end{equation*}
such that if we set $\lambda_{j}={\rm inf}\!\set{\!\mu\in\mathbb{R}_{\geq0} | \text{$K_{X'_{j}}+\Delta'_{j}+\mu H'_{j}$ is nef over $Z$}\!}$, where $H'_{j}$ is the birational transform of $H_{1}$ on $X'_{j}$, then the $(K_{X'_{1}}+\Delta'_{1})$-MMP over $Z$ terminates after finitely many steps or we have ${\rm lim}_{j\to \infty}\lambda_{j}=0$ when it does not terminate. 

For each $i\geq1$, pick the minimum $k_{i}$ such that $K_{X'_{k_{i}}}+\Delta'_{k_{i}}+e_{i}H'_{k_{i}}$ is nef over $Z$. 
Such $k_{i}$ exists since ${\rm lim}_{j\to \infty}\lambda_{j}=0$, and $k_{1}=1$. 
We put 
\begin{equation*}
X_{i}=X'_{k_{i}}, \quad \Delta_{i}=\Delta'_{k_{i}} \quad \text{and}\quad H_{i}=H'_{k_{i}}.
\end{equation*} 
Then the pair $(X_{i},\Delta_{i}+e_{i}H_{i})$ is a good minimal model of $(X_{1}, \Delta_{1}+e_{i}H_{1})$ over $Z$ by construction. 
We show that $(X_{i},\Delta_{i}+e_{i}H_{i})$ is also a good minimal model of $(X, \Delta+e_{i}H)$ over $Z$ for any $i$. 
Recall that the map $(X,\Delta+e_{i}H)\dashrightarrow(\widetilde{X}_{i},\widetilde{\Delta}_{i}+e_{i}\widetilde{H}_{i})$ is in particular a birational contraction to a good minimal model over $Z$, which was constructed at the third paragraph of this step. 
Recall also that all $\widetilde{X}_{i}$ are isomorphic in codimension one. 
Since $X_{1}=\widetilde{X}_{1}$, we see that $X_{1}$ and $\widetilde{X}_{i}$ are isomorphic in codimension one. 
Furthermore, since ${\rm lim}_{i\to \infty}e_{i}=0$, the divisor $K_{X_{1}}+\Delta_{1}$ is the limit of movable divisors over $Z$. 
Then the $(K_{X_{1}}+\Delta_{1})$-MMP contains only flips, which shows that $X_{i}$ and $\widetilde{X}_{i}$ are isomorphic in codimension one. 
By Remark \ref{remmodels}, the pair $(X_{i},\Delta_{i}+e_{i}H_{i})$ is a good minimal model of $(X,\Delta+e_{i}H)$ over $Z$ for any $i$. 
In this way, we see that the sequence of birational maps
$(X_{1},\Delta_{1})\dashrightarrow\cdots \dashrightarrow (X_{j},\Delta_{j})\dashrightarrow \cdots$ over $Z$
satisfies conditions (3) and (4) stated at the start of this step. 

We have constructed a sequence of positive real numbers $\{e_{i}\}_{i\geq1}$ and a sequence of birational maps 
\begin{equation*}
X \dashrightarrow X_{1}\dashrightarrow X_{2}\dashrightarrow \cdots \dashrightarrow X_{i}\dashrightarrow \cdots
\end{equation*}
over $Z$.
It is clear that they satisfy conditions (1)--(4) stated at the start of this step. 
\end{step1}

\begin{step1}\label{step5relative}
Suppose that the $(K_{X_{1}}+\Delta_{1})$-MMP over $Z$ with scaling of $e_{1}H_{1}$ terminates. 
Then $X_{l}\simeq X_{l+1}\simeq \cdots$ for some $l$.
By (3) in Step \ref{step4relative}, the pair $(X_{l},\Delta_{l}+e_{i}H_{l})$ is a good minimal model of $(X,\Delta+e_{i}H)$ for any $i\geq l$. 
Then $a(P,X,\Delta+e_{i}H)\leq a(P,X_{l},\Delta_{l}+e_{i}H_{l})$ for any prime divisor $P$ over $X$. 
By considering the limit $i\to \infty$, we have an inequality
$a(P,X,\Delta)\leq a(P,X_{l},\Delta_{l})$ 
 for any prime divisor $P$ over $X$. 
Therefore $(X_{l},\Delta_{l})$ is a weak lc model of $(X,\Delta)$ over $Z$, which implies that $(X,\Delta)$ has a log minimal model over $Z$. 

In this way, to prove the existence of log minimal model of $(X,\Delta)$, we only have to prove the termination of the $(K_{X_{1}}+\Delta_{1})$-MMP over $Z$. 
\end{step1}

\begin{step1}\label{step6relative}
Since we have $K_{X_{1}}+\Delta_{1}+e_{1}H_{1}\sim_{\mathbb{R},Z}(1+e_{1})(K_{X_{1}}+\Delta_{1}-\frac{e_{1}}{1+e_{1}}G_{1})$ by ($\spadesuit$) in Step \ref{step4relative} and since we have ${\rm Supp}G_{1}\subset {\rm Supp}\llcorner \Delta_{1} \lrcorner$ (cf.~(II) in Step \ref{step3relative}), the $(K_{X_{1}}+\Delta_{1})$-MMP over $Z$ occurs only in ${\rm Supp}\llcorner \Delta_{1}\lrcorner$ (cf.~\cite[Step 2 in the proof of Proposition 5.4]{has-trivial}). 

Note that $(X_{1},\Delta_{1}+e_{1}H_{1})$ is $\mathbb{Q}$-factorial dlt and any lc center of the pair is an lc center of $(X_{1},\Delta_{1})$. 
So, for any $i$, the pair $(X_{i},\Delta_{i})$ is $\mathbb{Q}$-factorial dlt and any lc center of the pair  is normal. 
There is $m>0$ such that for any lc center $S_{m}$ of $(X_{m},\Delta_{m})$ and  any $i \geq m$, the indeterminacy locus of the birational map $X_{m}\dashrightarrow X_{i}$ does not contain $S_{m}$ and the restriction of the map to $S_{m}$ induces a birational map $S_{m}\dashrightarrow S_{i}$ to an lc center $S_{i}$ of $(X_{i},\Delta_{i})$. 
We define an $\mathbb{R}$-divisor $\Delta_{S_{i}}$ on $S_{i}$ by adjunction $K_{S_{i}}+\Delta_{S_{i}}=(K_{X_{i}}+\Delta_{i})|_{S_{i}}$. 
From this step to Step \ref{step9relative}, we prove that for any lc center $S_{m}$ of $(X_{m},\Delta_{m})$, there is $i_{0}\geq m$ such that the induced birational map $(S_{i}, \Delta_{S_{i}})\dashrightarrow (S_{i+1}, \Delta_{S_{i+1}})$ is an isomorphism for any $i\geq i_{0}$. 
More strongly, we prove the following statement:

\begin{claim*}
For any lc center $S_{m}$ of $(X_{m},\Delta_{m})$, there is $i_{0}\geq m$ such that $K_{S_{i_{0}}}+\Delta_{S_{i_{0}}}$ is abundant over $Z$, the inequality $\kappa_{\sigma}(S_{i_{0}}/Z, K_{S_{i_{0}}}+\Delta_{S_{i_{0}}})\leq n$ holds, and the induced birational map $(S_{i}, \Delta_{S_{i}})\dashrightarrow (S_{i+1}, \Delta_{S_{i+1}})$ is an isomorphism for any $i\geq i_{0}$. 
\end{claim*}

Assuming this, then we see that the $(K_{X_{1}}+\Delta_{1})$-MMP over $Z$ terminates by the same argument as in \cite{fujino-sp-ter}, and we get a contradiction. 

We therefore prove the claim, and we prove it by induction on the dimension of $S_{m}$. 
Let $\Upsilon_{m}$ be an lc center of $(X_{m},\Delta_{m})$ such that $\Upsilon_{m}\subset S_{m}$. 
By arguments as in \cite{fujino-sp-ter}, by the induction hypothesis of the claim and replacing $m$, we may assume that for any $i\geq m$, if $\Upsilon_{m}\subsetneq S_{m}$ then $K_{\Upsilon_{m}}+\Delta_{\Upsilon_{m}}$ is abundant over $Z$, $\kappa_{\sigma}(\Upsilon_{m}/Z, K_{\Upsilon_{m}}+\Delta_{\Upsilon_{m}})\leq n$, and the map $(\Upsilon_{m}, \Delta_{\Upsilon_{m}})\dashrightarrow (\Upsilon_{i}, \Delta_{\Upsilon_{i}})$ is an isomorphism.  
By arguments as in \cite{fujino-sp-ter} and replacing $m$ again, we may assume that if $\Upsilon_{m}= S_{m}$ then the birational map $\Upsilon_{m}\dashrightarrow \Upsilon_{i}$ is small and the birational transform of $\Delta_{\Upsilon_{m}}$ on $\Upsilon_{i}$ is equal to $\Delta_{\Upsilon_{i}}$. 
\end{step1}

\begin{step1}\label{step7relative}
The basic strategy is similar to \cite[Proof of Theorem 1.2]{birkar-09}. 
In the rest of the proof, unless otherwise stated all $i$ are assumed to be $i\geq m$. 
In this step, we define some varieties and divisors used in the rest of the proof. 
At the start of the next step, we state all notations and facts we will use. 

By construction of $(X,\Delta+e_{i}H)\dashrightarrow(X_{i},\Delta_{i}+e_{i}H_{i})$ (see (2) and (4) in Step \ref{step4relative}), there is an lc center $S$ of $(X,\Delta)$ such that there is an induced birational map $S\dashrightarrow S_{i}$. 
We set $H_{S}=H|_{S}$ and $H_{S_{i}}=H_{i}|_{S_{i}}$, and we define $\Delta_{S}$ by adjunction $K_{S}+\Delta_{S}=(K_{X}+\Delta)|_{S}$. 
Then we have $H_{S}\geq0$ and $H_{S_{i}}\geq0$, and $H_{S_{i}}$ is equal to the birational transform of $H_{S_{m}}$ on $S_{i}$. 
By (2) and (4) in Step \ref{step4relative}, there is a common resolution $\overline{X}\to X$ and $\overline{X}\to {X_{i}}$ and a subvariety $\overline{S}\subset \overline{X}$ birational to $S$ and $S_{i}$ such that the induced morphisms $\overline{S}\to S$ and $\overline{S}\to S_{i}$ form a common resolution of the map $S\dashrightarrow S_{i}$. 
By (3) in Step \ref{step4relative}, comparing coefficients of divisors $(K_{X}+\Delta+e_{i}H)|_{\overline{S}}$ and $(K_{X_{i}}+\Delta_{i}+e_{i}H_{i})|_{\overline{S}}$, we obtain 
\begin{equation*}
a(Q,S,\Delta_{S}+e_{i}H_{S})\leq a(Q,S_{i},\Delta_{S_{i}}+e_{i}H_{S_{i}})
\end{equation*}
 for any prime divisor $Q$ over $S$. 

For each $i$, we set
 \begin{equation*}
 \mathcal{C}_{i}=\Set{ Q'
 | \begin{array}{l}\text{$Q'$ is a component of $H_{S}$ which is exceptional over $S_{i}$}\\ \text{such that $-1<a(Q',S_{i}, \Delta_{S_{i}}+e_{i}H_{S_{i}})<0$}
 \end{array}\!}.
 \end{equation*}
By Lemma \ref{lem--extraction}, there is a crepant model $\psi_{i}\colon (T_{i}, \Psi_{i})\to (S_{i},\Delta_{S_{i}}+e_{i}H_{S_{i}})$ such that the pair $(T_{i}, \Psi_{i})$ is $\mathbb{Q}$-factorial dlt and $\psi_{i}$-exceptional prime divisors are exactly elements of $\mathcal{C}_{i}$. 
By (4) in Step \ref{step4relative} and \cite[Lemma 4.2.10]{fujino-sp-ter}, we have 
\begin{equation*}
\begin{split}
a(Q, S_{i},\Delta_{S_{i}}+e_{i}H_{S_{i}})&\leq a(Q,S_{i},\Delta_{S_{i}}+e_{i+1}H_{S_{i}})\leq a(Q,S_{i+1},\Delta_{S_{i+1}}+e_{i+1}H_{S_{i+1}})
\end{split}
\end{equation*}
for any prime divisor $Q$ over $S_{i}$. 
Therefore, the induced birational map $T_{i}\dashrightarrow T_{i+1}$ is a birational contraction. 
By replacing $m$, we may assume that $T_{m}\dashrightarrow T_{i}$ is isomorphic in codimension one for any $i$. 
Let $\Psi_{i}^{(m)}$ be the birational transform of $\Psi_{i}$ on $T_{m}$. 
By the above relation and since $K_{T_{i}}+\Psi_{i}=\psi_{i}^{*}(K_{S_{i}}+\Delta_{S_{i}}+e_{i}H_{S_{i}})$, we have 
\begin{equation*}
\Psi_{m}\geq \Psi_{m+1}^{(m)}\geq \cdots \geq \Psi_{i}^{(m)} \geq \cdots \geq0.
\end{equation*}
Thus the limit $\Psi_{\infty}^{(m)}:={\rm lim}_{i \to \infty}\Psi_{i}^{(m)}$ exists as an $\mathbb{R}$-divisor. 
Then the pair $(T_{m},\Psi_{\infty}^{(m)})$ is $\mathbb{Q}$-factorial dlt because $(T_{m},\Psi_{m})$ is $\mathbb{Q}$-factorial dlt and $\Psi_{\infty}^{(m)}\leq \Psi_{m}$. 

By (4) in Step \ref{step4relative}, the map $(X_{m},\Delta_{m}+e_{i}H_{m})\dashrightarrow (X_{i},\Delta_{i}+e_{i}H_{i})$ is a sequence of steps of a $(K_{X_{m}}+\Delta_{m}+e_{i}H_{m})$-MMP over $Z$ with scaling of $(e_{m}-e_{i})H_{m}$ to a good minimal model over $Z$. 
By \cite[Lemma 4.2.10]{fujino-sp-ter}, we have 
\begin{equation*}
\begin{split}
a(Q, S_{m},\Delta_{S_{m}}+e_{i}H_{S_{m}})\leq a(Q,S_{i},\Delta_{S_{i}}+e_{i}H_{S_{i}})=a(Q,T_{i},\Psi_{i})
\end{split}
\end{equation*}
for any prime divisor $Q$ over $S_{m}$. 
Since $T_{m}\dashrightarrow T_{i}$ is isomorphic in codimension one, we obtain $\psi_{m}^{*}(K_{S_{m}}+\Delta_{S_{m}}+e_{i}H_{S_{m}})\geq K_{T_{m}}+\Psi_{i}^{(m)}$ for any $i$. 
By considering the limit $i\to \infty$, we have 
$a(Q,S_{m},\Delta_{S_{m}})\leq  a(Q,T_{m}, \Psi_{\infty}^{(m)})$ for any prime divisor $Q$ over $S_{m}$. 
\end{step1}

\begin{step1}\label{step8relative}
We have constructed the diagram
\begin{equation*}
\xymatrix
{
&&T_{m} \ar[d]_{\psi_{m}} &
\cdots  & T_{i} \ar[d]_{\psi_{i}}& T_{i+1} \ar[d]_{\psi_{i+1}}&\cdots \\
S\ar@{-->}[rr] &&S_{m}  \ar@{-->}[r]&
\cdots  \ar@{-->}[r]& S_{i} \ar@{-->}[r]& S_{i+1} \ar@{-->}[r]&\cdots 
}
\end{equation*}
over $Z$ and $\mathbb{Q}$-factorial dlt pairs $(T_{i},\Psi_{i})$ and $(T_{m},\Psi_{\infty}^{(m)})$ such that
\begin{enumerate}
\item[(a)]
for each $i$, the birational map $\psi_{i}^{-1}$ exactly extracts components $Q'$ of $H_{S}$ which are exceptional over $S_{i}$ such that $-1<a(Q',S_{i},\Delta_{S_{i}}+e_{i}H_{S_{i}} )<0$,
\item[(b)]
for any $i$, the birational map $T_{m}\dashrightarrow T_{i}$ is isomorphic in codimension one,  
\item[(c)]
 for any $i$, the divisor $K_{T_{i}}+\Psi_{i}=\psi_{i}^{*}(K_{S_{i}}+\Delta_{S_{i}}+e_{i}H_{S_{i}})$ is semi-ample over $Z$, 
\item[(d)]
we have $a(Q,S,\Delta_{S}+e_{i}H_{S})\leq a(Q,S_{i},\Delta_{S_{i}}+e_{i}H_{S_{i}})=a(Q,T_{i},\Psi_{i})$
for any $i$ and any prime divisor $Q$ over $S$, 
\item[(e)]
$\Psi_{i}^{(m)}$ is the birational transform of $\Psi_{i}$ on $T_{m}$, and the inequality $\Psi_{i}^{(m)}\geq \Psi_{i+1}^{(m)}\geq 0$ holds for any $i$,   
\item[(f)]
$\Psi_{\infty}^{(m)}:={\rm lim}_{i \to \infty}\Psi_{i}^{(m)}$, and 
\item[(g)]
we have $a(Q,S_{m},\Delta_{S_{m}})\leq  a(Q,T_{m}, \Psi_{\infty}^{(m)})$ for any prime divisor $Q$ over $S_{m}$. 
\end{enumerate}
In this step, we prove that $(T_{m}, \Psi_{\infty}^{(m)})$ has a good minimal model over $Z$. 
Note that the divisor $K_{T_{m}}+\Psi_{\infty}^{(m)}$ is the limit of movable divisors $K_{T_{m}}+\Psi_{i}^{(m)}$ over $Z$ (see (b), (c), (e) and (f)), so it is pseudo-effective over $Z$. 
By the induction hypothesis of Theorem \ref{thmmain}, it is sufficient to prove that 
$K_{T_{m}}+\Psi_{\infty}^{(m)}$ is log abundant over $Z$ and for any lc center $\Pi_{m}$ of $(T_{m}, \Psi_{\infty}^{(m)})$ the inequality $\kappa_{\sigma}(\Pi_{m}/Z, (K_{T_{m}}+\Psi_{\infty}^{(m)})|_{\Pi_{m}})\leq n$ holds. 

First, we prove that $K_{T_{m}}+\Psi_{\infty}^{(m)}$ is abundant over $Z$. 
If $a(Q',S,\Delta_{S})\leq a(Q',T_{m}, \Psi_{\infty}^{(m)})$ holds for any prime divisor $Q'$ on $S$, by applying Lemma \ref{lem--abundant-birat-2} to 
the birational map $(S,\Delta_{S})\dashrightarrow(T_{m}, \Psi_{\infty}^{(m)})$ over $Z$, we see that the divisor $K_{T_{m}}+\Psi_{\infty}^{(m)}$ is abundant over $Z$. 
Pick any prime divisor $Q'$ on $S$. 
If $Q'$ is not a component of $H_{S}$, then we have $a(Q',S,\Delta_{S})=a(Q',S,\Delta_{S}+e_{m}H_{S})\leq a(Q',S_{m},\Delta_{S_{m}}+e_{m}H_{S_{m}})$ by (d), and therefore  
\begin{equation*}
\begin{split}
a(Q',S,\Delta_{S})\leq a(Q',S_{m},\Delta_{S_{m}}+e_{m}H_{S_{m}})\leq a(Q',S_{m},\Delta_{S_{m}})
&\leq  a(Q',T_{m}, \Psi_{\infty}^{(m)}) 
\end{split}
\end{equation*}
by (g). 
So we may assume that $Q'$ is a component of $H_{S}$. 
Suppose that $Q'$ is exceptional over $T_{m}$. 
We note that we have
$-1<a(Q',S,\Delta_{S}+e_{m}H_{S})\leq a(Q',S_{m},\Delta_{S_{m}}+e_{m}H_{S_{m}})$ 
by (d). 
Since $Q'$ is exceptional over $S_{m}$ and it is not extracted by $\psi_{m}^{-1}$,  by (a), we obtain $a(Q',S_{m},\Delta_{S_{m}}+e_{m}H_{S_{m}})\geq0$. 
From these relations and (g), we have 
\begin{equation*}
\begin{split}
a(Q',S,\Delta_{S})\leq 0\leq a(Q',S_{m},\Delta_{S_{m}}+e_{m}H_{S_{m}})\leq a(Q',S_{m},\Delta_{S_{m}})\leq  a(Q',T_{m}, \Psi_{\infty}^{(m)}).
\end{split}
\end{equation*}
Finally, suppose that $Q'$ is a divisor on $T_{m}$. 
Since $T_{m}\dashrightarrow T_{i}$ is small for any $i$, we have 
\begin{equation*}
\begin{split}
a(Q',S,\Delta_{S}+e_{i}H_{S})\leq a(Q',T_{i},\Psi_{i})=a(Q',T_{m}, \Psi_{i}^{(m)}).
\end{split}
\end{equation*}
Here, the first inequality follows from (d).   
By considering the limit $i\to \infty$, we obtain $a(Q',S,\Delta_{S})\leq a(Q',T_{m}, \Psi_{\infty}^{(m)})$. 
In any case, we have $a(Q',S,\Delta_{S})\leq a(Q', T_{m}, \Psi_{\infty}^{(m)})$ for any divisor $Q'$ on $S$. 
By Lemma \ref{lem--abundant-birat-2}, the divisor $K_{T_{m}}+\Psi_{\infty}^{(m)}$ is abundant over $Z$. 

To complete this step, we need to prove that  $(K_{T_{m}}+\Psi_{\infty}^{(m)})|_{\Pi_{m}}$ is abundant over $Z$ and the inequality $\kappa_{\sigma}(\Pi_{m}/Z, (K_{T_{m}}+\Psi_{\infty}^{(m)})|_{\Pi_{m}})\leq n$ holds for any lc center $\Pi_{m}$ of $(T_{m}, \Psi_{\infty}^{(m)})$. 
Fix any lc center $\Pi_{m}$ of $(T_{m}, \Psi_{\infty}^{(m)})$. 
By (g), we have $\psi_{m}^{*}(K_{S_{m}}+\Delta_{S_{m}})\geq K_{T_{m}}+\Psi_{\infty}^{(m)}$. 
So the image of $\Pi_{m}$ on $S_{m}$, which we denote $\Upsilon'_{m}$, is an lc center of $(S_{m},\Delta_{S_{m}})$, and $\Upsilon'_{m}$ is also an lc center of $(X_{m},\Delta_{m})$ satisfying $\Upsilon'_{m} \subsetneq S_{m}$. 
By the induction hypothesis of the claim in Step \ref{step6relative} in this proof, the divisor $K_{\Upsilon'_{m}}+\Delta_{\Upsilon'_{m}}=(K_{S_{m}}+\Delta_{S_{m}})|_{\Upsilon'_{m}}$ is abundant over $Z$ and the inequality $\kappa_{\sigma}(\Upsilon'_{m}/Z, K_{\Upsilon'_{m}}+\Delta_{\Upsilon'_{m}})\leq n$ holds. 
We put $\psi_{\Pi_{m}}=\psi_{m}|_{\Pi_{m}}\colon \Pi_{m}\to \Upsilon'_{m}$ and define $\Psi_{\Pi_{m}}$ by adjunction $K_{\Pi_{m}}+\Psi_{\Pi_{m}}=(K_{T_{m}}+\Psi_{\infty}^{(m)})|_{\Pi_{m}}$. 
Since $\Pi_{m}$ is an lc center of $(T_{m}, \Psi_{\infty}^{(m)})$, with the relation $\psi_{m}^{*}(K_{S_{m}}+\Delta_{S_{m}})\geq K_{T_{m}}+\Psi_{\infty}^{(m)}$ and adjunction we can find $\Delta_{\Pi_{m}}\geq \Psi_{\Pi_{m}}$ such that $(\Pi_{m},\Delta_{\Pi_{m}})$ is lc  and $K_{\Pi_{m}}+\Delta_{\Pi_{m}}\sim_{\mathbb{R}}\psi_{\Pi_{m}}^{*}(K_{\Upsilon'_{m}}+\Delta_{\Upsilon'_{m}})$. 
The relation $K_{\Pi_{m}}+\Delta_{\Pi_{m}}\sim_{\mathbb{R}}\psi_{\Pi_{m}}^{*}(K_{\Upsilon'_{m}}+\Delta_{\Upsilon'_{m}})$ shows that $K_{\Pi_{m}}+\Delta_{\Pi_{m}}$ is abundant over $Z$ and $\kappa_{\sigma}(\Pi_{m}/Z, K_{\Pi_{m}}+\Delta_{\Pi_{m}})\leq n$. 
It is clear that $a(Q'', \Pi_{m},\Delta_{\Pi_{m}})\leq a(Q'', \Pi_{m}, \Psi_{\Pi_{m}})$ for any prime divisor $Q''$ on $\Pi_{m}$. 
Applying Lemma \ref{lem--abundant-birat-2} to the map $(\Pi_{m},\Delta_{\Pi_{m}})\dashrightarrow (\Pi_{m}, \Psi_{\Pi_{m}})$ over $Z$, we see that the divisor $(K_{T_{m}}+\Psi_{\infty}^{(m)})|_{\Pi_{m}}=K_{\Pi_{m}}+\Psi_{\Pi_{m}}$ is abundant over $Z$ and $\kappa_{\sigma}(\Pi_{m}/Z, (K_{T_{m}}+\Psi_{\infty}^{(m)})|_{\Pi_{m}})\leq n$ for any lc center $\Pi_{m}$ of $(T_{m}, \Psi_{\infty}^{(m)})$. 

In this way, we see that $K_{T_{m}}+\Psi_{\infty}^{(m)}$ is log abundant over $Z$ and for any lc center $\Pi_{m}$ of $(T_{m}, \Psi_{\infty}^{(m)})$ the inequality $\kappa_{\sigma}(\Pi_{m}/Z, (K_{T_{m}}+\Psi_{\infty}^{(m)})|_{\Pi_{m}})\leq n$ holds. 
By the induction hypothesis of Theorem \ref{thmmain}, the pair $(T_{m}, \Psi_{\infty}^{(m)})$ has a good minimal model over $Z$. 
Thus we complete this step. 
\end{step1}

\begin{step1}\label{step9relative}
With this step we complete the proof of the existence of a log minimal model of $(X,\Delta)$ over $Z$. 
In other words, we prove the claim stated in Step \ref{step6relative} in this proof. 

By Step \ref{step8relative} and Theorem \ref{thmtermi}, by running a $(K_{T_{m}}+ 
\Psi_{\infty}^{(m)})$-MMP over $Z$, we get a good minimal model over $Z$
\begin{equation*}
(T_{m}, \Psi_{\infty}^{(m)})\dashrightarrow (T',\Psi'_{\infty}).
\end{equation*} 
This log MMP contains only flips because $K_{T_{m}}+\Psi_{\infty}^{(m)}$ is the limit of movable divisors $K_{T_{m}}+\Psi_{i}^{(m)}$ over $Z$ ((b), (c), (e) and (f) in Step \ref{step8relative}). 
Therefore, $T'$ and $T_{m}$ are isomorphic in codimension one. 
Let $\Psi'_{i}$ be the birational transform of $\Psi_{i}$ on $T'$. 
Then $\Psi'_{i}\geq \Psi'_{i+1}\geq 0$ for any $i$ and $\Psi'_{\infty}={\rm lim}_{i \to \infty}\Psi'_{i}$ ((e) and (f) in Step \ref{step8relative}). 
Since $T_{m}\dashrightarrow T'$ is a sequence of steps of a $(K_{T_{m}}+ \Psi_{i}^{(m)})$-MMP for any $i\gg m$, the pair $(T',\Psi'_{i})$ is lc for any $i\gg m$. 
We recall that $K_{T_{i}}+\Psi_{i}$ is semi-ample over $Z$ (see (c) in Step \ref{step8relative}). 
Because $K_{T'}+\Psi'_{i}$ is the birational transform of $K_{T_{i}}+\Psi_{i}$ and since $T_{i}$ and $T'$ are isomorphic in codimension one, $(T_{i}, \Psi_{i})$ is a weak lc model of $(T',\Psi'_{i})$ over $Z$ with relatively semi-ample log canonical divisor.  
So $(T',\Psi'_{i})$ has a good minimal model over $Z$ for any $i\gg m$. 
By Theorem \ref{thmtermi}, there is a sequence of steps of a $(K_{T'}+\Psi'_{i})$-MMP over $Z$ to a good minimal model. 
Fix an $i_{0}\gg m$ such that 
$\Psi'_{i_{0}}-\Psi'_{\infty}$ is sufficiently small so that $\Psi'_{i_{0}}$ satisfies the property of Lemma \ref{lem--polytope-1}, that is, for any sequence of steps of any $(K_{T'}+\Psi'_{i_{0}})$-MMP over $Z$ to a good minimal model 
\begin{equation*}
(T',\Psi'_{i_{0}})\dashrightarrow(T'',\Psi''_{i_{0}})
\end{equation*}
and the contraction $T''\to W$ over $Z$ induced by $K_{T''}+\Psi''_{i_{0}}$, the divisor $K_{T''}+\Psi''_{\infty}$ is nef over $Z$ and $K_{T''}+\Psi''_{\infty}\sim_{\mathbb{R},W}0$, where $\Psi''_{\infty}$ is the birational transform of $\Psi'_{\infty}$ on $T''$. 
Note that this log MMP over $Z$ also contains only flips since $K_{T'}+\Psi'_{i_{0}}$ is movable over $Z$. 
Therefore, $T''$ and $T'$ are isomorphic in codimension one, hence $T''$ and $T_{i_{0}}$ are isomorphic in codimension one. 
We focus on the following diagram over $Z$.
\begin{equation*}
\xymatrix
{
(T'',\Psi''_{i_{0}})\ar[d]&(T_{i_{0}}, \Psi_{i_{0}})\ar[d]^{\psi_{i_{0}}} \ar@{-->}[l]\\
W&(S_{i_{0}},\Delta_{S_{i_{0}}}+e_{i_{0}}H_{S_{i_{0}}})\ar@{-->}[l]
}
\end{equation*}
Here, $S_{i_{0}}\dashrightarrow W$ is the induced map. 
Recall that $K_{T_{i_{0}}}+\Psi_{i_{0}}=\psi_{i_{0}}^{*}(K_{S_{i_{0}}}+\Delta_{S_{i_{0}}}+e_{i_{0}}H_{S_{i_{0}}})$ is semi-ample over $Z$ ((c) in Step \ref{step8relative}), and $T''\to W$ is the contraction over $Z$ induced by $K_{T''}+\Psi''_{i_{0}}$. 
Since $T''$ and $T_{i_{0}}$ are isomorphic in codimension one, by taking a common resolution of $T''\dashrightarrow T_{i_{0}}$ and by the negativity lemma, we see that the map $S_{i_{0}}\dashrightarrow W$ is a morphism. 
Furthermore, by definition of $\Psi_{\infty}^{(m)}$ and since ${\rm lim}_{i \to \infty}e_{i}=0$, the birational transform of $K_{T''}+\Psi''_{\infty}$ on $S_{i_{0}}$ is $K_{S_{i_{0}}}+\Delta_{S_{i_{0}}}$. 
Since we have $K_{T''}+\Psi''_{\infty}\sim_{\mathbb{R},W}0$, the divisor $K_{T''}+\Psi''_{\infty}$ is $\mathbb{R}$-linearly equivalent to the pullback of a relatively nef divisor on $W$. 
From these facts, $K_{S_{i_{0}}}+\Delta_{S_{i_{0}}}$ is $\mathbb{R}$-linearly equivalent to the pullback of a relatively nef divisor on $W$, hence it is nef over $Z$. 

There is $i_{0}\geq m$ such that $K_{S_{i_{0}}}+\Delta_{S_{i_{0}}}$ is nef over $Z$. 
Arguments as in \cite{fujino-sp-ter} show the birational map $(S_{i}, \Delta_{S_{i}})\dashrightarrow (S_{i+1}, \Delta_{S_{i+1}})$ is an isomorphism for any $i\geq i_{0}$. 
This is the final condition of the claim in Step \ref{step6relative} in this proof.  
Moreover, by (d) in Step \ref{step8relative}, we have 
\begin{equation*}
a(Q,S,\Delta_{S}+e_{i}H_{S})\leq a(Q,S_{i},\Delta_{S_{i}}+e_{i}H_{S_{i}})=a(Q,S_{i_{0}},\Delta_{S_{i_{0}}}+e_{i}H_{S_{i_{0}}})
\end{equation*}
for any prime divisor $Q$ over $S_{i_{0}}$ and any $i\geq i_{0}$. 
Since we have ${\rm lim}_{i\to \infty}e_{i}=0$, we obtain $a(Q,S,\Delta_{S})\leq a(Q,S_{i_{0}},\Delta_{S_{i_{0}}})$. 
Now we can apply Lemma \ref{lem--abundant-birat-2} to $(S,\Delta_{S})\dashrightarrow(S_{i_{0}},\Delta_{S_{i_{0}}})$ over $Z$ because $K_{S}+\Delta_{S}$ is abundant over $Z$ and $\kappa_{\sigma}(S/Z, K_{S}+\Delta_{S})\leq n$ by hypothesis of Theorem \ref{thmmain}. 
By Lemma \ref{lem--abundant-birat-2}, we see that $K_{S_{i_{0}}}+\Delta_{S_{i_{0}}}$ is abundant over $Z$ and we obtain $\kappa_{\sigma}(S_{i_{0}}/Z, K_{S_{i_{0}}}+\Delta_{S_{i_{0}}})\leq n$. 
Therefore, the claim in Step \ref{step6relative} in this proof holds for any lc center $S_{m}$ of $(X_{m},\Delta_{m})$, which implies that $(X,\Delta)$ has a log minimal model over $Z$ by special termination (see Step \ref{step6relative} in this proof). 
\end{step1}

\begin{step1}\label{step10relative}
Finally, we prove that $(X,\Delta)$ has a good minimal model over $Z$. 

By running a $(K_{X}+\Delta)$-MMP over $Z$ with scaling, we obtain a log minimal model $(X,\Delta)\dashrightarrow(Y,\Gamma)$ over $Z$. 
We need to show that $K_{Y}+\Gamma$ is log abundant over $Z$ because we do not know the log abundance of $K_{Y}+\Gamma$ over $Z$ directly from construction of log MMP.
By construction, $K_{Y}+\Gamma$ is abundant over $Z$. 
Pick any lc center $S_{Y}$ of $(Y,\Gamma)$. 
Then there is an induced birational map $S\dashrightarrow S_{Y}$ from an lc center $S$ of $(X,\Delta)$.  
We define $\Delta_{S}$ and $\Gamma_{S_{Y}}$ by adjunctions $K_{S}+\Delta_{S}=(K_{X}+\Delta)|_{S}$ and $K_{S_{Y}}+\Gamma_{S_{Y}}=(K_{Y}+\Gamma)|_{S_{Y}}$, respectively. 
By \cite[Lemma 4.2.10]{fujino-sp-ter}, we have $a(Q,S,\Delta_{S})\leq a(Q,S_{Y},\Gamma_{S_{Y}})$ for any prime divisor $Q$ on $S$. 
By applying Lemma \ref{lem--abundant-birat-2} to $(S,\Delta_{S})\dashrightarrow(S_{Y},\Gamma_{S_{Y}})$ over $Z$, we see that $K_{S_{Y}}+\Gamma_{S_{Y}}$ is abundant over $Z$, which implies that $K_{Y}+\Gamma$ is log abundant over $Z$. 
By \cite[Lemma 3.4]{has-class}, we see that $K_{Y}+\Gamma$ is semi-ample over $Z$. 
\end{step1}

We complete the proof. 
\end{proof}

\section{Abundant lc pairs and lc pairs with big boundary divisors}\label{sec4}

In this section, we prove results on the minimal model theory for lc pairs such that its log canonical divisor is abundant or boundary divisor is big.  

\begin{thm}\label{thm--abund-gentype}
Assume the existence of good minimal models or Mori fiber spaces for all projective klt pairs of dimension $n$. 
Let $(X,\Delta)$ be a projective lc pair such that
\begin{itemize}
\item
$\Delta$ is big, and 
\item
any lc center of $(X,\Delta)$ is at most $n$-dimensional. 
\end{itemize}
Then $K_{X}+\Delta$ is abundant. 
\end{thm}

\begin{proof}
We take a dlt blow-up $\pi\colon(X_{0},\Delta_{0})\to (X,\Delta)$. 
Then $\Delta_{0}$ is big. 
We prove that $K_{X_{0}}+\Delta_{0}$ is abundant in several steps. 
We may assume that $K_{X_{0}}+\Delta_{0}$ is pseudo-effective.

\begin{step2}\label{step1abund}
In this step, we prove Theorem \ref{thm--abund-gentype} when $K_{X_{0}}+\Delta_{0}-\epsilon \llcorner \Delta_{0}\lrcorner$ is pseudo-effective for some $\epsilon >0$. 

Renaming $\epsilon$ by $\frac{\epsilon}{2}$, we may assume that $K_{X_{0}}+\Delta_{0}-2\epsilon \llcorner \Delta_{0}\lrcorner$ is pseudo-effective. 
We may also assume that $\Delta_{0}-2\epsilon \llcorner \Delta_{0}\lrcorner$ is big. 
By \cite{bchm}, we have
\begin{equation*}
\begin{split}
\kappa_{\iota}(X_{0}, K_{X_{0}}+\Delta_{0}-2\epsilon \llcorner \Delta_{0}\lrcorner)=\kappa_{\sigma}(X_{0}, K_{X_{0}}+\Delta_{0}-2\epsilon \llcorner \Delta_{0}\lrcorner)\geq0
\end{split}
\end{equation*}
and similarly we see that $K_{X_{0}}+\Delta_{0}-\epsilon\llcorner \Delta_{0}\lrcorner$ is abundant. 
Then $K_{X_{0}}+\Delta_{0}-2\epsilon \llcorner \Delta_{0}\lrcorner$ is $\mathbb{R}$-linearly equivalent to an $\mathbb{R}$-divisor $G\geq0$, hence we have $K_{X_{0}}+\Delta_{0}\sim_{\mathbb{R}}G+2\epsilon\llcorner \Delta_{0}\lrcorner$ and $K_{X_{0}}+\Delta_{0}-\epsilon\llcorner \Delta_{0}\lrcorner \sim_{\mathbb{R}}G+\epsilon\llcorner \Delta_{0}\lrcorner$.
By Remark \ref{remdiv} (1), we obtain
\begin{equation*}
\begin{split}
\kappa_{\sigma}(X_{0},K_{X_{0}}+\Delta_{0})&=\kappa_{\sigma}(X_{0},K_{X_{0}}+\Delta_{0}-\epsilon\llcorner \Delta_{0}\lrcorner )=\kappa_{\iota}(X_{0},K_{X_{0}}+\Delta_{0}-\epsilon\llcorner \Delta_{0}\lrcorner )\\
&=\kappa_{\iota}(X_{0},K_{X_{0}}+\Delta_{0}), 
\end{split}
\end{equation*}
and therefore $K_{X_{0}}+\Delta_{0}$ is abundant. 

So we may assume that $K_{X_{0}}+\Delta_{0}-\epsilon \llcorner \Delta_{0}\lrcorner$ is not pseudo-effective for any $\epsilon >0$. 
\end{step2}

\begin{step2}\label{step2abund}
By Step \ref{step1abund}, we can find a component $S_{0}$ of $\llcorner \Delta_{0}\lrcorner$ such that $K_{X_{0}}+\Delta_{0}-\epsilon S_{0}$ is not pseudo-effective for any $\epsilon>0$. 
By \cite[Lemma 3.1]{gongyo-nonvanishing}, we can construct a birational contraction $(X_{0},\Delta_{0})\dashrightarrow (X',\Delta')$, which is not necessarily $(K_{X_{0}}+\Delta_{0})$-non-positive, and a contraction $X'\to Z'$ such that $(X',\Delta')$ is lc, $K_{X'}+\Delta'\sim_{\mathbb{R},Z'}0$ and the birational transform of $S_{0}$ on $X'$ is a prime divisor which dominates $Z'$. 
We take a log resolution $\phi\colon \bar{X} \to X_{0}$ of $(X_{0},\Delta_{0})$ such that the induced map $\psi\colon \bar{X}\dashrightarrow X'$ is a morphism. 
Then we can write 
\begin{equation*}\tag{$*$}
K_{\bar{X}}+\bar{\Delta}=\phi^{*}(K_{X_{0}}+\Delta_{0})+\bar{E}
\end{equation*}
with $\bar{\Delta}\geq0$ and $\bar{E}\geq0$ such that $\bar{\Delta}$ and $\bar{E}$ have no common components. 
We can also write $K_{\bar{X}}+\bar{\Delta}=\psi^{*}(K_{X'}+\Delta')+\bar{E}_{+}-\bar{E}_{-}$
with $\psi$-exceptional $\mathbb{R}$-divisors $\bar{E}_{+}\geq0$ and $\bar{E}_{-}\geq0$ which have no common components. 
We run a $(K_{\bar{X}}+\bar{\Delta})$-MMP over $X'$ with scaling of an ample divisor. 
By \cite[Theorem 3.5]{birkar-flip}, we obtain a model $\psi'\colon (\bar{X}',\bar{\Delta}')\to X'$ such that $K_{\bar{X}'}+\bar{\Delta}'+\bar{E}'_{-}=\psi'^{*}(K_{X'}+\Delta')$, where $\bar{\Delta}'$ (resp.~$\bar{E}'_{-}$) is the birational transform of  $\bar{\Delta}$ (resp.~$\bar{E}_{-}$) on $\bar{X}'$. 
Then the pair $(\bar{X}',\bar{\Delta}'+\bar{E}'_{-})$ is lc and 
$K_{\bar{X}'}+\bar{\Delta}'+\bar{E}'_{-}\sim_{\mathbb{R},Z'}0$. 
By \cite[Theorem 1.1]{has-mmp}, we can run a $(K_{\bar{X}'}+\bar{\Delta}')$-MMP over $Z'$ and obtain a good minimal model $(\bar{X}'',\bar{\Delta}'')\to Z'$. 
Let $\bar{X}''\to Z''$ be the contraction over $Z'$ induced by $K_{\bar{X}''}+\bar{\Delta}''$. 
Then $K_{\bar{X}''}+\bar{\Delta}''\sim_{\mathbb{R},Z''}0$, and $Z''\to Z'$ is birational. 
\end{step2}

\begin{step2}\label{step3abund}
We have the following diagram.
\begin{equation*}
\xymatrix@R=12pt{
\bar{X}\ar[d]_{\phi}\ar@{-->}[r]&\bar{X}''\ar[d]\\
X_{0}\ar[d]_{\pi}&Z''\\
X
}
\end{equation*}

The birational map $(\bar{X},\bar{\Delta})\dashrightarrow (\bar{X}'',\bar{\Delta}'')$ is a sequence of  steps of a $(K_{\bar{X}}+\bar{\Delta})$-MMP. 
So it is sufficient to prove that $K_{\bar{X}''}+\bar{\Delta}''$ is abundant. 
We recall that $S_{0}$ is a component of $\llcorner \Delta_{0}\lrcorner$ such that $K_{X_{0}}+\Delta_{0}-\epsilon S_{0}$ is not pseudo-effective for any $\epsilon>0$. 
Let $\bar{S}$ be the birational transform of $S_{0}$ on $\bar{X}$. 
Then, the divisor $K_{\bar{X}}+\bar{\Delta}-\epsilon \bar{S}$ is not pseudo-effective for any $\epsilon>0$, hence $\bar{S}$ is not contracted by the map $\bar{X}\dashrightarrow \bar{X}''$. 
Moreover, its birational transform $\bar{S}''$ on $\bar{X}''$ dominates $Z''$. 
It is because the morphism $Z''\to Z'$ is birational and the birational transform of $S_{0}$ on $X'$ is a prime divisor which dominates $Z'$ (for the second fact, see the second sentence of Step \ref{step2abund} in this proof). 
Let $\Delta_{\bar{S}''}$ be an $\mathbb{R}$-divisor on $\bar{S}''$ defined by adjunction $K_{\bar{S}''}+\Delta_{\bar{S}''}=(K_{\bar{X}''}+\bar{\Delta}'')|_{\bar{S}''}$. 
Then the pair $(\bar{S}'', \Delta_{\bar{S}''})$ is dlt because $(\bar{X}'',\bar{\Delta}'')$ is $\mathbb{Q}$-factorial dlt. 
Since we have $K_{\bar{X}''}+\bar{\Delta}''\sim_{\mathbb{R},Z''}0$ and $\bar{S}''\to Z''$ is surjective, by Remark \ref{remdiv}, it is sufficient to prove that $K_{\bar{S}''}+\Delta_{\bar{S}''}$ is abundant.  
\end{step2}

\begin{step2}\label{step4abund}
We set $\phi_{\bar{S}}=\phi|_{\bar{S}}$. 
Now we have the following diagram.
\begin{equation*}
\xymatrix@R=12pt{
\bar{S}\ar[d]_{\phi_{\bar{S}}}\ar@{-->}[rr]&&\bar{S}''\ar[d]\\
S_{0}&&Z'' 
}
\end{equation*}
We define an $\mathbb{R}$-divisor $\Delta_{\bar{S}}$ on $\bar{S}$ by adjunction $K_{\bar{S}}+\Delta_{\bar{S}}=(K_{\bar{X}}+\bar{\Delta})|_{\bar{S}}$. 
We also define an $\mathbb{R}$-divisor $\Delta_{S_{0}}$ on $S_{0}$ by $K_{S_{0}}+\Delta_{S_{0}}=(K_{X_{0}}+\Delta_{0})|_{S_{0}}$. 
Then $K_{S_{0}}+\Delta_{S_{0}}$ is abundant and $\kappa_{\sigma}(S_{0},K_{S_{0}}+\Delta_{S_{0}})\leq n$ because of \cite[Theorem 1.5]{has-mmp} and the hypothesis that all lc centers of $(X,\Delta)$ are at most $n$-dimensional. 
As in the argument in Step \ref{step8relative} in the proof of Theorem \ref{thmmain}, by Lemma \ref{lem--abundant-birat-2}, to prove that $K_{\bar{S}''}+\Delta_{\bar{S}''}$ is abundant it is sufficient to prove $a(Q,S_{0},\Delta_{S_{0}})\leq a(Q,\bar{S}'',\Delta_{\bar{S}''})$ for any prime divisor $Q$ on $S_{0}$. 
We have 
\begin{equation*}
K_{\bar{S}}+\Delta_{\bar{S}}=\phi_{\bar{S}}^{*}(K_{S_{0}}+\Delta_{S_{0}})+\bar{E}|_{\bar{S}}
\end{equation*}
by ($*$) in Step \ref{step2abund}. 
Since $\phi\colon\bar{X}\to X_{0}$ is a log resolution of $(X_{0},\Delta_{0})$, by Lemma \ref{lem--adjunction}, the divisors $\Delta_{\bar{S}}$ and $\bar{E}|_{\bar{S}}$ have no common components and $\bar{E}|_{\bar{S}}$ is $\phi_{\bar{S}}$-exceptional, and hence we have $a(Q,S_{0}, \Delta_{S_{0}})=a(Q, \bar{S}, \Delta_{\bar{S}})$ for any prime divisor $Q$ on $S_{0}$. 
Moreover, since the birational map $(\bar{X},\bar{\Delta})\dashrightarrow (\bar{X}'',\bar{\Delta}'')$ is a sequence of steps of a $(K_{\bar{X}}+\bar{\Delta})$-MMP, we have $a(Q,\bar{S}, \Delta_{\bar{S}})\leq a(Q, \bar{S}'', \Delta_{\bar{S}''})$ by \cite[Lemma 4.2.10]{fujino-sp-ter}. 
So $a(Q,S_{0}, \Delta_{S_{0}})\leq a(Q, \bar{S}'', \Delta_{\bar{S}''})$ for any prime divisor $Q$ on $S_{0}$. 
By Lemma \ref{lem--abundant-birat-2}, we see that $K_{\bar{S}''}+\Delta_{\bar{S}''}$ is abundant, and  $K_{X_{0}}+\Delta_{0}$ is also abundant.
\end{step2}

So we complete the proof. 
\end{proof}

\begin{thm}\label{thm--gentype}
Assume the existence of good minimal models or Mori fiber spaces for all projective klt pairs of dimension $n$. 
Let $\pi\colon X\to Z$ be a projective morphism of normal quasi-projective varieties, and let $(X,\Delta)$ be an lc pair such that
\begin{itemize}
\item
$K_{X}+\Delta$ is abundant over $Z $ or $\Delta$ is big over $Z$, and 
\item
for any lc center $S$ of $(X,\Delta)$, we have ${\rm dim}S-{\rm dim}\,\pi(S)\leq n$. 
\end{itemize}
Then, $(X,\Delta)$ has a good minimal model or a Mori fiber space over $Z$. 
\end{thm}

\begin{proof}
When $\Delta$ is big over $Z$, by restricting $(X,\Delta)$ to a sufficiently general fiber of the Stein factorization of $\pi$ and applying Theorem \ref{thm--abund-gentype} and Lemma \ref{lem--abundant-general}, we see that $K_{X}+\Delta$ is abundant over $Z$. 
Therefore, it is sufficient to prove Theorem \ref{thm--gentype} in the case when $K_{X}+\Delta$ is abundant over $Z$. 
We take a dlt blow-up $(Y,\Gamma)$ of $(X,\Delta)$. 
For any lc center $S$ of $(X,\Delta)$, pick an lc center $T$ of $(Y,\Gamma)$ mapped to $S$ surjectively. 
We define $\Gamma_{T}$ by adjunction $K_{T}+\Gamma_{T}=(K_{Y}+\Gamma)|_{T}$. 
We have a morphism $(T,\Gamma_{T})\to S\to \pi(S)$ such that $K_{T}+\Gamma_{T}\sim_{\mathbb{R},S}0$. 
Take the Stein factorization $T\to S'$ of $T\to S$. 
After that, take the Stein factorization $S'\to S''$ of $S'\to \pi(S)$. 
By construction $T\to S''$ is the Stein factorization of $T\to \pi(S)$. 
Restricting the morphism $T\to S'\to S''$ over a sufficiently general point $z\in S''$, we get a morphism $T_{z}\to S'_{z}$ of projective varieties and an lc pair $(T_{z}, \Gamma_{T_{z}})$ such that ${\rm dim}S'_{z}\leq n$ and $K_{T_{z}}+\Gamma_{T_{z}}\sim_{\mathbb{R}, S'_{z}}0$. 
By \cite[Theorem 1.5]{has-mmp}, $(T_{z}, \Gamma_{T_{z}})$ has a good minimal model or a Mori fiber space. 
In particular, the divisor $K_{T_{z}}+\Gamma_{T_{z}}$ is abundant, and we have $\kappa_{\sigma}(T_{z}, K_{T_{z}}+\Gamma_{T_{z}})\leq n$ because $K_{T_{z}}+\Gamma_{T_{z}}\sim_{\mathbb{R}, S'_{z}}0$ and ${\rm dim}S'_{z}\leq n$. 
By Lemma \ref{lem--abundant-general}, we see that $K_{T}+\Gamma_{T}$ is abundant over $Z$ and $\kappa_{\sigma}(T/Z, K_{T}+\Gamma_{T})\leq n$. 
In this way, we see that $(Y,\Gamma)$ satisfies all the conditions of Theorem \ref{thmmain}. 
By Theorem \ref{thmmain}, $(Y,\Gamma)$ has a good minimal model or a Mori fiber space over $Z$, and so does $(X,\Delta)$.  
\end{proof}

\begin{proof}[Proof of Theorem \ref{thm--calabiyau}]
Note that $K_{X}+B\sim_{\mathbb{R}}0$. 
For any effective $\mathbb{Q}$-Cartier divisor $D$ satisfying the hypothesis, there is a rational number $t>0$ such that $(X,B+tD)$ is lc and all lc centers of the pair are at most $3$-dimensional. 
By Theorem \ref{thm--gentype}, $(X,B+tD)$ has a good minimal model. 
Since $tD\sim_{\mathbb{R}} K_{X}+B+tD$, the assertion holds true. 
\end{proof}

\begin{proof}[Proof of Theorem \ref{thm--flat}] Let $S$ be an lc center of $(X,\Delta)$. Since all fibers of $\pi$ have the same dimension and ${\rm dim}X-{\rm dim}Z\leq n$, we have ${\rm dim}S-{\rm dim}\,\pi(S)\leq n$. 
So the morphism $(X,\Delta)\to Z$ satisfies the conditions of Theorem \ref{thm--gentype}. By Theorem \ref{thm--gentype}, $(X,\Delta)$ has a good minimal model or a Mori fiber space over $Z$. \end{proof}

Finally, we prove a result on lc pairs whose boundary divisors contain log big divisors. 

\begin{thm}\label{thm--logbig}
Assume the existence of good minimal models or Mori fiber spaces for all projective klt pairs of dimension $n$. 

Let $(X,\Delta)$ be a projective lc pair, and let $A\geq0$ be a log big $\mathbb{R}$-Cartier divisor with respect to $(X,\Delta)$, i.e., $A$ is big and $A|_{S^{\nu}}$ is big for any lc center $S$ of $(X,\Delta)$ with the normalization $S^{\nu}$. 
Suppose that 
\begin{itemize}
\item
an inequality $\kappa_{\sigma}(S^{\nu},(K_{X}+\Delta+A)|_{S^{\nu}})\leq n$ holds for any lc center $S$ of $(X,\Delta)$ with the normalization $S^{\nu}$, and
\item
the pair $(X,\Delta+(1+t)A)$ is lc for some $t>0$. 
\end{itemize}
Then, $(X,\Delta+A)$ has a good minimal model or a Mori fiber space. 
\end{thm}

\begin{proof}
We prove the following statement, the relative setting of the theorem.
\begin{itemize}
\item[($*$)]
Assume the existence of good minimal models or Mori fiber spaces for all projective klt pairs of dimension $n$. 

Let $(X,\Delta)$ be a projective $\mathbb{Q}$-factorial dlt pair and $\pi\colon X\to Z$ a surjective morphism of normal projective varieties. 
Let $A\geq0$ be a big $\mathbb{R}$-Cartier divisor on $Z$. 
Suppose that 
\begin{itemize}
\item[$\bullet$]
$K_{X}+\Delta\sim_{\mathbb{R},Z}0$, 
\item[$\bullet$]
$\kappa_{\sigma}(S,(K_{X}+\Delta+\pi^{*}A)|_{S})\leq n$ and $A|_{\pi(S)^{\nu}}$ is big for any lc center $S$ of $(X,\Delta)$, where $\pi(S)^{\nu}$ is the normalization of $\pi(S)$, and 
\item[$\bullet$]
the pair $(X,\Delta+(1+t)\pi^{*}A)$ is lc for some $t>0$. 
\end{itemize}
Then $K_{X}+\Delta+\pi^{*}A$ is log abundant.
\end{itemize}
Assuming this, then we see that $K_{X}+\Delta+A$ in Theorem \ref{thm--logbig} is log abundant by applying $(*)$ to a dlt model of $(X,\Delta+A)$, hence $(X,\Delta+A)$ has a good minimal model or a Mori fiber space by Theorem \ref{thmmain}. 
We therefore prove $(*)$, and we prove it by induction on the dimension of $X$. 
By the third condition, any lc center of $(X,\Delta+\pi^{*}A)$ is an lc center of $(X,\Delta)$. 
By restricting $K_{X}+\Delta+\pi^{*}A$ to each lc center of $(X,\Delta+\pi^{*}A)$ and applying the induction hypothesis, we see that it is sufficient to prove that $K_{X}+\Delta+\pi^{*}A$ is abundant. 
We may assume that $\pi$ is a contraction and $K_{X}+\Delta+\pi^{*}A$ is pseudo-effective. 

Suppose that $K_{X}+\Delta-\epsilon\llcorner \Delta \lrcorner+\pi^{*}A$ is pseudo-effective for some $\epsilon>0$. 
As in Step \ref{step1abund} in the proof of Theorem \ref{thm--abund-gentype}, after replacing $\epsilon$ by $\frac{\epsilon}{2}$, we see that it is sufficient to prove that $K_{X}+\Delta-\epsilon\llcorner \Delta \lrcorner+\pi^{*}A$ is abundant. 
Set $B=\Delta-\epsilon\llcorner \Delta \lrcorner$, then $(X,B+\pi^{*}A)$ is klt because of the third condition of $(*)$. 
In this paragraph, we focus on $(X,B+\pi^{*}A)$ so we ignore the second and the third condition of $(*)$. 
By Lemma \ref{lem--abundantmmp}, there is a good minimal model $(X',B')$ of $(X,B)$ over $Z$. 
Let $X'\to Z'$ be the contraction over $Z$ induced by $K_{X'}+B'$, and let $h\colon Z'\to Z$ be the induced morphism, which is birational since the restriction of $K_{X'}+B'$ to a general fiber of $X'\to Z$ is numerically trivial. 
Replacing $(X,B)\to Z$ and $A$ with $(X',B')\to Z'$ and $h^{*}A$ respectively, we may assume $K_{X}+B\sim_{\mathbb{R},Z}0$.  
By \cite[Corollary 3.2]{fg-bundle}, there are $\mathbb{R}$-divisors $B_{Z}$ and $\Theta_{Z}$ such that $(Z,B_{Z})$ and $(Z,\Theta_{Z})$ are klt and relations
$K_{X}+B\sim_{\mathbb{R}}\pi^{*}(K_{Z}+B_{Z})$ and $K_{X}+B+\pi^{*}A\sim_{\mathbb{R}}\pi^{*}(K_{Z}+\Theta_{Z})$ hold. 
Fix $s>0$ such that the pair $(Z, s(B_{Z}+A)+(1-s)\Theta_{Z})$ is klt. 
Then 
\begin{equation*}
\begin{split}
K_{X}+B+\pi^{*}A&\sim_{\mathbb{R}}s\pi^{*}(K_{Z}+B_{Z})+s\pi^{*}A+(1-s)\pi^{*}(K_{Z}+\Theta_{Z})\\
&=
\pi^{*}(K_{Z}+s(B_{Z}+A)+(1-s)\Theta_{Z})
\end{split}
\end{equation*}
and the divisor $K_{Z}+s(B_{Z}+A)+(1-s)\Theta_{Z}$ is abundant because $A$ is big (\cite{bchm}). 
Therefore, $K_{X}+B+\pi^{*}A$ is abundant, and so is $K_{X}+\Delta+\pi^{*}A$. 

We may assume that $K_{X}+\Delta-\epsilon\llcorner \Delta \lrcorner+\pi^{*}A$ is not pseudo-effective for any $\epsilon>0$. 
Put $\Psi=\Delta+\pi^{*}A$.
As in steps \ref{step2abund} and \ref{step3abund} in the proof of Theorem \ref{thm--abund-gentype}, we construct 
\begin{itemize}
\item[(1)]
a log resolution $\phi \colon \bar{X} \to X$ of $(X,\Psi)$ and a log smooth lc pair $(\bar{X},\bar{\Psi})$ such that we may write $K_{\bar{X}}+\bar{\Psi}=\phi^{*}(K_{X}+\Psi)+\bar{E}$ with $\bar{E}\geq0$ such that $\bar{E}$ and $\bar{\Psi}$ have no common components, 
\item[(2)]
a sequence of steps of log MMP $(\bar{X},\bar{\Psi})\dashrightarrow (\bar{X}'',\bar{\Psi}'')$ and a contraction $\bar{X}''\to Z''$ such that $K_{\bar{X}''}+\bar{\Psi}''\sim_{\mathbb{R},Z''}0$ and $K_{X}+\Psi$ is abundant if $K_{\bar{X}''}+\bar{\Psi}''$ is abundant, and 
\item[(3)]
a component $S$ of $\llcorner \Delta \lrcorner$ with the birational transform $\bar{S}''$ on $\bar{X}''$ such that the induced morphism $\bar{S}''\to Z''$ is surjective. 
\end{itemize}
Let $\bar{S}$ be the birational transform of $S$ on $\bar{X}$. 
We define $\Psi_{S}$, $\Psi_{\bar{S}}$, and $\Psi_{\bar{S}''}$ by adjunctions $K_{S}+\Psi_{S}=(K_{X}+\Psi)|_{S}$, $K_{\bar{S}}+\Psi_{\bar{S}}=(K_{\bar{X}}+\bar{\Psi})|_{\bar{S}}$, and  $K_{\bar{S}''}+\Psi_{\bar{S}''}=(K_{\bar{X}''}+\bar{\Psi}'')|_{\bar{S}''}$, respectively. 
Since $K_{\bar{X}''}+\bar{\Psi}''\sim_{\mathbb{R},Z''}0$ and $\bar{S}''\to Z''$ is surjective, $K_{\bar{X}''}+\bar{\Psi}''$ is abundant if $K_{\bar{S}''}+\Psi_{\bar{S}''}$ is abundant, hence $K_{X}+\Psi$ is abundant if $K_{\bar{S}''}+\Psi_{\bar{S}''}$ is abundant by $(2)$. 
By (1), construction of $\bar{E}$ and Lemma \ref{lem--adjunction}, we may write 
$K_{\bar{S}}+\Psi_{\bar{S}}=\phi|_{\bar{S}}^{*}(K_{S}+\Psi_{S})+\bar{E}|_{\bar{S}}$,
where $\bar{E}|_{\bar{S}}\geq0$ is a $\phi|_{\bar{S}}$-exceptional divisor. 
The calculation as in Step \ref{step4abund} in the proof of Theorem \ref{thm--abund-gentype} shows $a(Q,S,\Psi_{S})=a(Q,\bar{S},\Psi_{\bar{S}})\leq a(Q,\bar{S}'',\Psi_{\bar{S}''})$ for any prime divisor $Q$ on $S$. 
Finally, by the induction hypothesis and hypothesis of the statement $(*)$, we see that $K_{S}+\Psi_{S}$ is abundant and $\kappa_{\sigma}(S,K_{S}+\Psi_{S})\leq n$. 
Therefore, applying Lemma \ref{lem--abundant-birat-2} to $(S,\Psi_{S})\dashrightarrow(\bar{S}'',\Psi_{\bar{S}''})$ over ${\rm Spec}\mathbb{C}$, we see that $K_{\bar{S}''}+\Psi_{\bar{S}''}$ is abundant, from which the divisor $K_{X}+\Delta+\pi^{*}A=K_{X}+\Psi$ is abundant.  

In this way, we see that the assertion $(*)$ holds. 
So we are done. 
\end{proof}

\section{Proof of Theorem \ref{thm--lcample} and Theorem \ref{thm--lcfano}}\label{sec5}

In this section, we prove Theorem \ref{thm--lcample} by the same idea as in Theorem \ref{thmmain}, and we prove Theorem \ref{thm--lcfano}. 
The key ingredient to prove Theorem \ref{thm--lcample} is the following theorem. 

\begin{thm}\label{thmmain--hu}
Let $\pi\colon X\to Z$ be a morphism of normal projective varieties, and let $(X,B)$ be an lc pair. 
Suppose that there is an $\mathbb{R}$-Cartier divisor $C\geq 0$ on $X$ such that 
\begin{itemize}
\item
the pair $(X,B+C)$ is lc, and  
\item
$K_{X}+B+C\sim_{\mathbb{R},Z}0$.
\end{itemize}
Let $A_{Z}$ be an ample $\mathbb{R}$-divisor on $Z$, and pick $A\sim_{\mathbb{R}}\pi^{*}A_{Z}$ such that $(X,B+A)$ is lc. 

Then $(X,B+A)$ has a good minimal model or a Mori fiber space. 
\end{thm}

The following lemma is a variant of Theorem \ref{thmmain--hu}. 
We use it to prove Theorem \ref{thmmain--hu}. 

\begin{lem}\label{lem--ind-birat}
Assume Theorem \ref{thmmain--hu} for all projective lc pairs of dimension $\leq n-1$. 
Let $\pi\colon X\to Z$, $(X,B)$, $C$, $A_{Z}$ and $A$ be as in Theorem \ref{thmmain--hu} such that ${\rm dim}X\leq n- 1$. 
Let $Y$ be a normal projective variety with a birational morphism $f\colon Y\to X$, and let $(Y,\Gamma)$ be an lc pair such that $(Y,\Gamma+f^{*}A)$ is lc. 
Suppose that the effective part of the divisor $K_{Y}+\Gamma+f^{*}A-f^{*}(K_{X}+B+A)$ is $f$-exceptional. 

Then $(Y,\Gamma+f^{*}A)$ has a good minimal model or a Mori fiber space.  
\end{lem}

\begin{proof}
We put $A_{Y}=f^{*}A$. 
By replacing $(Y,\Gamma+A_{Y})$ with a dlt model, we may assume that $(Y,\Gamma+A_{Y})$ is $\mathbb{Q}$-factorial dlt. 
We may write
\begin{equation*}
K_{Y}+\Gamma+A_{Y}=f^{*}(K_{X}+B+A)+M-N
\end{equation*}
with $M\geq 0$ and $N \geq 0$ such that $M$ and $N$ have no common components and $M$ is $f$-exceptional. 
We run a $(K_{Y}+\Gamma+A_{Y})$-MMP over $X$ with scaling of an ample divisor. 
By argument of very exceptional divisors, we reach a model $f'\colon (Y', \Gamma'+A_{Y'})\to X$ such that $M$ is contracted by the birational map $Y\dashrightarrow Y'$ (\cite[Theorem 3.5]{birkar-flip}). 
Note that $A_{Y'}=f'^{*}A$. 
Let $N'$ be the birational transform of $N$ on $Y'$. 
Since $N'\geq0$ and 
\begin{equation*}
K_{Y'}+\Gamma'+(N'+f'^{*}C)=f'^{*}(K_{X}+B+C)\sim_{\mathbb{R},Z}0,
\end{equation*}
we see that $f'^{*}C+N'\geq0$, $K_{Y'}+\Gamma'+(N'+f'^{*}C)\sim_{\mathbb{R},Z}0$, and the pair $(Y', \Gamma'+(f'^{*}C+N'))$ is lc. 
Now we have $A_{Y'}=f'^{*}A\sim_{\mathbb{R}}(\pi\circ f')^{*}A_{Z}$, and hence
 we can apply Theorem \ref{thmmain--hu} to $\pi\circ f'\colon(Y',\Gamma')\to Z$, $f'^{*}C+N'$, $A_{Z}$ and $A_{Y'}$. 
We see that $(Y',\Gamma'+A_{Y'})$ has a good minimal model or a Mori fiber space, and thus $(Y,\Gamma+A_{Y})$ has a good minimal model or a Mori fiber space.  
\end{proof}

\subsection{Proof of Theorem \ref{thmmain--hu}: Generalized abundance and abundance}

The goal of this subsection is to prove the generalized abundance and  the abundance theorem in the setting of Theorem \ref{thmmain--hu} (Theorem \ref{thm--gen-abund-hu} and Theorem \ref{thm--abund-hu}). 

\begin{lem}[see also {\cite[Corollary 4.3]{bh}}]\label{lem--thmmainklt}
Let $\pi\colon X\to Z$ be a morphism of normal projective varieties, and let $(X,B)$ be a klt pair. 
Assume $\kappa_{\sigma}(X/Z, K_{X}+B)=0$. 
Let $A_{Z}$ be an ample $\mathbb{R}$-divisor on $Z$. 

Then $K_{X}+B+\pi^{*}A_{Z}$ is abundant. 
In particular, if the pair $(X,B+\pi^{*}A_{Z})$ is klt then it has a good minimal model or a Mori fiber space. 
\end{lem}

\begin{proof}
By Lemma \ref{lem--abundantmmp}, we can find a good minimal model $(X',B')$ of $(X,B)$ over $Z$. 
Let $\pi'\colon X'\to Z'$ be the contraction over $Z$ induced by $K_{X'}+B'$, and let $h\colon Z'\to Z$ be the induced morphism. 
By \cite[Corollary 3.2]{fg-bundle} there is an $\mathbb{R}$-divisor $B_{Z'}$ such that $(Z',B_{Z'})$ is klt and 
$K_{X'}+B'\sim_{\mathbb{R}}\pi'^{*}(K_{Z'}+B_{Z'})$. 
Then the divisor $K_{Z'}+B_{Z'}+h^{*}A_{Z}$ is abundant (\cite{bchm}) because $h^{*}A_{Z}$ is nef and big. 
So the divisor $K_{X'}+B'+\pi'^{*}h^{*}A_{Z}$ is abundant, and hence we see that $K_{X}+B+\pi^{*}A_{Z}$ is abundant. 
The second assertion follows from the first assertion and  Lemma \ref{lem--abundantmmp}.  
\end{proof}

\begin{thm}\label{thm--gen-abund-hu}
Assume Theorem \ref{thmmain--hu} for all projective lc pairs of dimension $\leq n-1$. 
Let $\pi\colon X\to Z$, $(X,B)$, $C$, $A_{Z}$ and $A$ be as in Theorem \ref{thmmain--hu} such that ${\rm dim}X= n$. 

Then $K_{X}+B+A$ is abundant. 
\end{thm}

\begin{proof}
The idea is very similar to the proof of Theorem \ref{thm--abund-gentype}. 
We may assume that $K_{X}+B+A$ is pseudo-effective and $\pi$ is a contraction. 
By replacing $A$ with a general one, we may also assume that $B$ and $A$ have no common components and all lc centers of $(X,B+A)$ are lc centers of $(X,B)$. 
By taking a dlt model of $(X,B+A)$, we may assume that $(X,B+A)$ is $\mathbb{Q}$-factorial dlt. 

\begin{step3}\label{step1abund-hu}
In this step, we prove Theorem \ref{thm--gen-abund-hu} in the case when $K_{X}+B-\epsilon \llcorner B\lrcorner+A$ is pseudo-effective for some $\epsilon >0$. 

Renaming $\epsilon$ by $\frac{\epsilon}{2}$, we may assume that $K_{X}+B-2\epsilon \llcorner B\lrcorner+A$ is pseudo-effective. 
Then 
\begin{equation*}
0\leq \kappa_{\sigma}(X/Z, K_{X}+B-2\epsilon \llcorner B\lrcorner)\leq \kappa_{\sigma}(X/Z, K_{X}+B)\leq0,
\end{equation*}
where the final inequality follows from $K_{X}+B+C\sim_{\mathbb{R},Z}0$.  
By Lemma \ref{lem--thmmainklt}, we have
\begin{equation*}
\begin{split}
\kappa_{\iota}(X, K_{X}+B-2\epsilon \llcorner B\lrcorner+A)=\kappa_{\sigma}(X, K_{X}+B-2\epsilon \llcorner B\lrcorner+A)\geq0, 
\end{split}
\end{equation*}
and similarly we see that $K_{X}+B-\epsilon\llcorner B\lrcorner+A$ is abundant. 
Then $K_{X}+B-2\epsilon \llcorner B\lrcorner+A$ is $\mathbb{R}$-linearly equivalent to an effective $\mathbb{R}$-divisor $G$. 
Then
\begin{equation*}
\begin{split}
K_{X}+B+A\sim_{\mathbb{R}}G+2\epsilon\llcorner B\lrcorner \quad {\rm and}\quad K_{X}+B-\epsilon\llcorner B\lrcorner+A\sim_{\mathbb{R}}G+\epsilon\llcorner B\lrcorner.
\end{split}
\end{equation*}
By Remark \ref{remdiv} (1), we have
\begin{equation*}
\begin{split}
\kappa_{\sigma}(X,K_{X}+B+A)&=\kappa_{\sigma}(X,K_{X}+B-\epsilon\llcorner B\lrcorner+A)=\kappa_{\iota}(X,K_{X}+B-\epsilon\llcorner B\lrcorner+A)\\&= \kappa_{\iota}(X,K_{X}+B+A),
\end{split}
\end{equation*}
hence $K_{X}+B+A$ is abundant. 

So we may assume that $K_{X}+B-\epsilon \llcorner B\lrcorner+A$ is not pseudo-effective for any $\epsilon >0$. 
\end{step3}

\begin{step3}\label{step2abund-hu}
By Step \ref{step1abund-hu}, there is a component $S$ of $\llcorner B\lrcorner$ such that $K_{X}+B-\epsilon S+A$ is not pseudo-effective for any $\epsilon>0$. 
For any $\epsilon'>0$, we run a $(K_{X}+B-\epsilon' S+A)$-MMP and get a birational contraction $X\dashrightarrow X'$ to a Mori fiber space $X'\to Z'$. 
Let $B'$, $S'$ and $A'$ be the birational transforms of $B$, $S$ and $A$ on $X'$, respectively. 
Then $(X',B'-\epsilon'S'+A')$ is lc. 
Let $I\subset \mathbb{R}_{\geq 0}$ be the set of coefficients of $B+A$, and consider the set 
\begin{equation*}
\set{\!{\rm lct}(V,\Delta;D)| \text{$(V,\Delta)$ is lc, ${\rm dim}V=n$, coefficients of $\Delta$ belong to $I$,  $D$ is reduced}\!}\!,
\end{equation*}
where ${\rm lct}(V,\Delta;D)$ is the log canonical threshold of $D$ with respect to a pair $(V,\Delta)$.  
By applying the ACC for log canonical thresholds (\cite[Theorem 1.1]{hmx-acc}) to the above set, there is $\epsilon_{0}>0$ such that if $\epsilon'<\epsilon_{0}$ then ${\rm lct}(X',B'-S'+A';S')=1$, which shows that $(X',B'+A')$ is lc. 
With the ACC for numerically trivial pairs (\cite[Theorem 1.5]{hmx-acc})  and by choosing $\epsilon'\in (0,\epsilon_{0})$ sufficiently small, we may assume that $(X',B'+A')$ is lc and $K_{X'}+B'+A'\sim_{\mathbb{R},Z'}0$.
For details, see \cite[Lemma 3.1]{gongyo-nonvanishing}.  
\end{step3}

\begin{step3}\label{step3abund-hu}
Let $\phi\colon \bar{X} \to X$ and $\psi\colon \bar{X}\to X'$ be a common log resolution of the birational map $(X,B)\dashrightarrow (X',B')$. 
Since $A$ is general, we may replace $A$ by an other member of its $\mathbb{R}$-linear system whose coefficients belong to $I$. 
The ACC for log canonical thresholds guarantees that the equality ${\rm lct}(X',B'-S'+A';S')=1$ still holds after we replace $A$, which shows that we can replace $A$ by a general one keeping conditions that $(X',B'+A')$ is lc and $K_{X'}+B'+A'\sim_{\mathbb{R},Z'}0$. 
By replacing $A$, we may assume that $\phi$ is a log resolution of $(X,B+A)$ and that ${\rm Supp}\phi^{*}A$ and ${\rm Supp}\phi_{*}^{-1}B\cup {\rm Ex}(\phi)\cup {\rm Ex}(\psi)$ have no common divisorial components. 
Then $\phi^{*}A=\phi_{*}^{-1}A$. 
Putting $\bar{A}=\phi^{*}A$, we can write 
\begin{equation*}\tag{i}
K_{\bar{X}}+\bar{B}+\bar{A}=\phi^{*}(K_{X}+B+A)+\bar{E}
\end{equation*}
with $\bar{B}\geq0$ and $\bar{E}\geq0$ which have no common components. 
Then $\bar{B}+\bar{A}$ and $\bar{E}$ have no common components and $(\bar{X},\bar{B}+\bar{A})$ is lc. 
We can also write 
\begin{equation*}
K_{\bar{X}}+\bar{B}+\bar{A}=\psi^{*}(K_{X'}+B'+A')+\bar{E}_{+}-\bar{E}_{-}
\end{equation*}
with $\psi$-exceptional $\mathbb{R}$-divisors $\bar{E}_{+}\geq0$ and $\bar{E}_{-}\geq0$ such that $\bar{E}_{+}$ and $\bar{E}_{-}$ have no common components. 
We run a $(K_{\bar{X}}+\bar{B}+\bar{A})$-MMP over $X'$ with scaling of an ample divisor. 
By \cite[Theorem 3.5]{birkar-flip}, we obtain a model $\psi'\colon (\bar{X}',\bar{B}'+\bar{A}')\to X'$ such that $K_{\bar{X}'}+\bar{B}'+\bar{A}'+\bar{E}'_{-}=\psi'^{*}(K_{X'}+B'+A')$, where $\bar{B}'$, $\bar{A}'$, and $\bar{E}'_{-}$ are the birational transforms of  $\bar{B}$, $\bar{A}$, and $\bar{E}_{-}$ on $\bar{X}'$ respectively.  
Then the pair $(\bar{X}',\bar{B}'+\bar{A}'+\bar{E}'_{-})$ is lc and we have 
\begin{equation*}
K_{\bar{X}'}+\bar{B}'+\bar{A}'+\bar{E}'_{-}\sim_{\mathbb{R},Z'}0.
\end{equation*}
By \cite[Theorem 1.1]{has-mmp} and running a $(K_{\bar{X}'}+\bar{B}'+\bar{A}')$-MMP over $Z'$, we obtain a good minimal model $(\bar{X}'',\bar{B}''+\bar{A}'')\to Z'$. 
Let $\bar{X}''\to Z''$ be the contraction over $Z'$ induced by $K_{\bar{X}''}+\bar{B}''+\bar{A}''$. 
Then $K_{\bar{X}''}+\bar{B}''+\bar{A}''\sim_{\mathbb{R},Z''}0$, and $Z''\to Z'$ is birational. 
\end{step3}

\begin{step3}\label{step4abund-hu}
We have the following diagram.
\begin{equation*}
\xymatrix@R=12pt{
\bar{X}\ar[d]_{\phi}\ar@{-->}[r]&\bar{X}''\ar[d]\\
X\ar[d]_{\pi}&Z''\\
Z
}
\end{equation*}

By construction, the birational map $(\bar{X},\bar{B}+\bar{A})\dashrightarrow (\bar{X}'',\bar{B}''+\bar{A}'')$ is a sequence of  steps of the $(K_{\bar{X}}+\bar{B}+\bar{A})$-MMP. 
So it is sufficient to prove that $K_{\bar{X}''}+\bar{B}''+\bar{A}''$ is abundant. 
We recall that $S$ is a component of $\llcorner B\lrcorner$ such that $K_{X}+B-\epsilon S+A$ is not pseudo-effective for any $\epsilon>0$. 
Let $\bar{S}$ be the birational transform of $S$ on $\bar{X}$. 
Then, the divisor $K_{\bar{X}}+\bar{B}-\epsilon \bar{S}+\bar{A}$ is not pseudo-effective for any $\epsilon>0$, hence $\bar{S}$ is not contracted by the map $\bar{X}\dashrightarrow \bar{X}''$. 
Moreover, its birational transform $\bar{S}''$ on $\bar{X}''$ dominates $Z''$. 
It is because $S'$ dominates $Z'$ by construction in Step \ref{step2abund-hu} (\cite[Lemma 3.1]{gongyo-nonvanishing}) and $Z''\to Z'$ is birational.
We define an $\mathbb{R}$-divisor $B_{\bar{S}''}$ by adjunction $K_{\bar{S}''}+B_{\bar{S}''}=(K_{\bar{X}''}+\bar{B}'')|_{\bar{S}''}$. 
We put $A_{\bar{S}''}=\bar{A}''|_{\bar{S}''}$. 
Since $K_{\bar{X}''}+\bar{B}''+\bar{A}''\sim_{\mathbb{R},Z''}0$ and since $\bar{S}''\to Z''$ is surjective, by Remark \ref{remdiv}, we only have to prove that $K_{\bar{S}''}+B_{\bar{S}''}+A_{\bar{S}''}$ is abundant.  

In the rest of the proof, we do not use the pair $(X',B'+A')$ and the morphism $X'\to Z'$, so we forget them.  
Take a common log resolution $\tau''\colon T\to \bar{S}''$ and $\tau\colon T\to \bar{S}$  of the map $(\bar{S},B_{\bar{S}})\dashrightarrow (\bar{S}'', B_{\bar{S}''})$. 
We can replace $A$ by a general member of its $\mathbb{R}$-linear system. 
Indeed, recall that $\phi\colon\bar{X}\to X$ is a log resolution of $(X,B)$ and the birational map $(\bar{X},\bar{B}+\bar{A})\dashrightarrow (\bar{X}'',\bar{B}''+\bar{A}'')$ is a sequence of  steps of the $(K_{\bar{X}}+\bar{B}+\bar{A})$-MMP. 
Recall also the relations $\bar{A}=\phi^{*}A$, $A_{\bar{S}''}=\bar{A}''|_{\bar{S}''}$, and that $\bar{A}''$ is the birational transform of $\bar{A}$ on $\bar{X}''$. 
Since the property of being abundant of $K_{\bar{S}''}+B_{\bar{S}''}+A_{\bar{S}''}$ 
does not depend on $\mathbb{R}$-linear equivalence class of $A_{\bar{S}''}$, we may replace $A$. 
By Lemma \ref{lem--genmember} and Lemma \ref{lem--genmember2} and replacing $A$, we may assume $A_{\bar{S}}\geq0$, $A_{\bar{S}''}\geq0$, $\tau^{*}A_{\bar{S}}\leq \tau_{*}''^{-1}A_{\bar{S}''}$, and $\tau''\colon T\to \bar{S}''$ is a log resolution of $(\bar{S}'', B_{\bar{S}''}+A_{\bar{S}''})$, where $A_{\bar{S}}=\bar{A}|_{\bar{S}}$. 
\end{step3}

\begin{step3}\label{step5abund-hu}
Let $\pi_{S}\colon S\to Z$ be the restriction of $\pi\colon X\to Z$ to $S$, and let $\phi_{\bar{S}}\colon\bar{S}\to S$ be the birational morphism induced by $\phi\colon\bar{X}\to X$. 
Now we have the following diagram.
\begin{equation*}
\xymatrix@R=12pt{
&
T\ar[dl]_{\tau}\ar[dr]^{\tau''}&\\
\bar{S}\ar[d]_{\phi_{\bar{S}}}\ar@{-->}[rr]&&\bar{S}''\ar[d]\\
S\ar[d]_{\pi_{S}}&&Z'' \\
Z&&
}
\end{equation*}
We define an $\mathbb{R}$-divisor $B_{\bar{S}}$ on $\bar{S}$ by adjunction $K_{\bar{S}}+B_{\bar{S}}=(K_{\bar{X}}+\bar{B})|_{\bar{S}}$. 
We also define $A_{S}$ and $B_{S}$ by $A_{S}=A|_{S}$ and $K_{S}+B_{S}=(K_{X}+B)|_{S}$, respectively. 
We have 
\begin{equation*}\tag{ii}
K_{\bar{S}}+B_{\bar{S}}+A_{\bar{S}}=\phi_{\bar{S}}^{*}(K_{S}+B_{S}+A_{S})+\bar{E}|_{\bar{S}}
\end{equation*}
by (i) in Step \ref{step3abund-hu}. 
Since $\phi\colon\bar{X}\to X$ is a log resolution of $(X,B+A)$ and by generality of $A$, we have $B_{\bar{S}}\geq0$ and $\bar{E}|_{\bar{S}}\geq0$, and furthermore $B_{\bar{S}}+A_{\bar{S}}$ and $\bar{E}|_{\bar{S}}$ have no common components and $\bar{E}|_{\bar{S}}$ is $\phi_{\bar{S}}$-exceptional (Lemma \ref{lem--adjunction}). 

Put $A_{T}=\tau^{*}A_{\bar{S}}$, which is equal to $\tau^{*}\phi_{\bar{S}}^{*}A_{S}$.  
Then $A_{T}\leq \tau_{*}''^{-1}A_{\bar{S}''}$. 
We can write 
\begin{equation*}\tag{iii}
K_{T}+\Psi+A_{T}=\tau''^{*}(K_{\bar{S}''}+B_{\bar{S}''}+A_{\bar{S}''})+E_{T}
\end{equation*}
with $\Psi\geq0$ and $E_{T}\geq0$ such that $(T, \Psi+A_{T})$ is a $\mathbb{Q}$-factorial dlt pair and $\Psi+A_{T}$ and $E_{T}$ have no common components. 
Then it is sufficient to prove that $K_{T}+\Psi+A_{T}$ is abundant. 
\end{step3}

\begin{step3}\label{step6abund-hu}
We may write
\begin{equation*}
K_{T}+\Psi+A_{T}=\tau^{*}\phi_{\bar{S}}^{*}(K_{S}+B_{S}+A_{S})+M-N,
\end{equation*}
with $M\geq0$ and $N\geq0$ which have no common components. 
If $M$ is $(\phi_{\bar{S}}\circ \tau)$-exceptional, by Lemma \ref{lem--ind-birat}, we see that $K_{T}+\Psi+A_{T}$ is abundant.  
So, to complete the proof, we only need to prove that $M$ is $(\phi_{\bar{S}}\circ \tau)$-exceptional. 

Suppose by contradiction that there is a component $Q$ of $M$ such that $Q$ is not contracted by $\phi_{\bar{S}}\circ \tau$. 
Then $a(Q,S, B_{S}+A_{S})$ satisfies
\begin{equation*}
\begin{split}
a(Q, T, \Psi+A_{T}) <a(Q,S, B_{S}+A_{S})\leq0. 
\end{split}
\end{equation*}
Since $B_{\bar{S}}+A_{\bar{S}}$ and $\bar{E}|_{\bar{S}}$ in the equation (ii) in Step \ref{step5abund-hu} have no common components and $\bar{E}|_{\bar{S}}$ is $\phi_{\bar{S}}$-exceptional, we have $a(Q,S, B_{S}+A_{S})=a(Q, \bar{S}, B_{\bar{S}}+A_{\bar{S}})$.
Moreover, since the map $(\bar{X},\bar{B}+\bar{A})\dashrightarrow (\bar{X}'',\bar{B}''+\bar{A}'')$ is a sequence of steps of the $(K_{\bar{X}}+\bar{B}+\bar{A})$-MMP, we have $a(Q,\bar{S}, B_{\bar{S}}+A_{\bar{S}})\leq a(Q, \bar{S}'', B_{\bar{S}''}+A_{\bar{S}''})$ (see \cite[Lemma 4.2.10]{fujino-sp-ter}). 
We also have $a(Q, T, \Psi+A_{T}) ={\rm min}\{0, a(Q, \bar{S}'', B_{\bar{S}''}+A_{\bar{S}''}) \}$ by (iii) in Step \ref{step5abund-hu} and Lemma \ref{lem--discre}. 
But we get a contradiction because
\begin{equation*} 
\begin{split}
a(Q,S, B_{S}+A_{S})&={\rm min}\{0, a(Q,S, B_{S}+A_{S}) \}={\rm min}\{0, a(Q, \bar{S}, B_{\bar{S}}+A_{\bar{S}}) \}\\&\leq {\rm min}\{0, a(Q, \bar{S}'', B_{\bar{S}''}+A_{\bar{S}''}) \}\\
&=a(Q, T, \Psi+A_{T})\\
&<a(Q,S, B_{S}+A_{S}).
\end{split} 
\end{equation*}
Thus, we see that $M$ is $(\phi_{\bar{S}}\circ \tau)$-exceptional. 
By Lemma \ref{lem--ind-birat}, $K_{T}+\Psi+A_{T}$ is abundant. 
\end{step3}
So $K_{\bar{X}''}+\bar{B}''+\bar{A}''$ is abundant, and so is $K_{X}+B+A$. 
We complete the proof. 
\end{proof}

\begin{thm}\label{thm--abund-hu}
Assume Theorem \ref{thmmain--hu} for all projective lc pairs of dimension $\leq n-1$. 
Let $\pi\colon X\to Z$, $(X,B)$, $C$, $A_{Z}$ and $A$ be as in Theorem \ref{thmmain--hu} such that ${\rm dim}X= n$. 

If $(X,B+A)$ has a log minimal model, then it has a good minimal model. 
\end{thm}

\begin{proof}
By taking a dlt blow-up of $(X,B)$ and by 
replacing $(X,B)$ and $A$, we can assume that $(X,B+A)$ is $\mathbb{Q}$-factorial dlt. 
We run a $(K_{X}+B+A)$-MMP and get a log minimal model $(X,B+A)\dashrightarrow(X',B'+A')$. 
We prove that $K_{X'}+B'+A'$ is semi-ample. 
By \cite[Lemma 3.4]{has-class}, it is sufficient to prove that $K_{X'}+B'+A'$ is nef and log abundant with respect to $(X',B'+A')$. 
By replacing $A$ if necessary, we can assume all lc centers of $(X,B+A)$ are lc centers of $(X,B)$. 

By Theorem \ref{thm--gen-abund-hu}, $K_{X'}+B'+A'$ is abundant. 
Pick any lc center $S'$ of $(X',B'+A')$. 
Then there is an lc center $S$ of $(X,B+A)$ such that the birational map $X\dashrightarrow X'$ induces a birational map $S\dashrightarrow S'$. 
We put $A_{S}=A|_{S}$ and $A_{S'}=A'|_{S'}$. 
Then $A_{S}\geq 0$ and $A_{S'}\geq0$. 
We define $B_{S}$ by adjunction $K_{S}+B_{S}=(K_{X}+B)|_{S}$. 
Similarly, we define $B_{S'}$ by $K_{S'}+B_{S'}=(K_{X'}+B')|_{S'}$. 
Let $\tau'\colon T\to S'$ and $\tau\colon T\to S$ be a common log resolution of the map $(S,B_{S})\dashrightarrow(S',B_{S'})$. 
By Lemma \ref{lem--genmember} and Lemma \ref{lem--genmember2} and by replacing $A$ with a general one, we can assume that $\tau'$ is a log resolution of $(S',B_{S'}+A_{S'})$ and $\tau^{*}A_{S}\leq \tau_{*}'^{-1}A_{S'}$. 
Put $A_{T}=\tau^{*}A_{S}$. 
Then we can write
\begin{equation*}
K_{T}+\Psi+A_{T}=\tau'^{*}(K_{S'}+B_{S'}+A_{S'})+E_{T}
\end{equation*}
with $\Psi\geq0$ and $E_{T}\geq0$ such that $(T,\Psi+A_{T})$ is $\mathbb{Q}$-factorial dlt and $\Psi+A_{T}$ and $E_{T}$ have no common components. 
We can also write 
\begin{equation*}
K_{T}+\Psi+A_{T}=\tau^{*}(K_{S}+B_{S}+A_{S})+M_{+}-M_{-},
\end{equation*}
where $M_{+}\geq0$ and $M_{-}\geq0$ have no common components. 
Then $M_{+}$ is $\tau$-exceptional. 
Indeed, any component $Q$ of $M_{+}$ satisfies $a(Q,S,B_{S}+A_{S})>a(Q,T,\Psi+A_{T})$. 
On the other hand, we have $a(Q,S,B_{S}+A_{S})\leq a(Q,S',B_{S'}+A_{S'})$ (\cite[Lemma 4.2.10]{fujino-sp-ter}). 
We have $a(Q,T,\Psi+A_{T})={\rm min}\{0, a(Q,S',B_{S'}+A_{S'})\}$ by construction and Lemma \ref{lem--discre}. 
Then
\begin{equation*}
{\rm min}\{0, a(Q,S',B_{S'}+A_{S'})\}< a(Q,S,B_{S}+A_{S})\leq a(Q,S',B_{S'}+A_{S'})\end{equation*}
and hence we have $a(Q,S,B_{S}+A_{S})>0$. 
So $Q$ is $\tau$-exceptional, and hence we see that $M_{+}$ is $\tau$-exceptional.  

By Lemma \ref{lem--ind-birat}, we see that $K_{T}+\Psi+A_{T}$ is abundant, then 
\begin{equation*}
K_{S'}+B_{S'}+A_{S'}=(K_{X'}+B'+A')|_{S'}
\end{equation*}
is abundant. 
So $K_{X'}+B'+A'$ is nef and log abundant, and \cite[Lemma 3.4]{has-class} shows that $K_{X'}+B'+A'$ is semi-ample. 
So we are done. 
\end{proof}

\subsection{Proof of Theorem \ref{thmmain--hu}: A special case}

The goal of this subsection is to prove the following. 

\begin{lem}\label{lem--reduction}
Assume Theorem \ref{thmmain--hu} for all projective lc pairs of dimension $\leq n-1$. 
Let $\pi\colon X\to Z$, $(X,B)$, $C$, $A_{Z}$ and $A$ be as in Theorem \ref{thmmain--hu} such that ${\rm dim}X=n$ and $(X,B)$ is $\mathbb{Q}$-factorial dlt. 
If $K_{X}+B-\epsilon \llcorner B \lrcorner+A$ is not pseudo-effective for any real number $\epsilon>0$, then $(X,B+A)$ has a good minimal model or a Mori fiber space. 
\end{lem}

Before the proof, we show the following lemma.

\begin{lem}[cf.~{\cite[Theorem 4.2]{bh}}]\label{lemzariski}
Let $\pi\colon X\to Z$ be a surjective morphism of normal projective varieties, and let $(X,\Delta)$ be a $\mathbb{Q}$-factorial dlt pair such that $K_{X}+\Delta$ is pseudo-effective. 
Suppose that $K_{X}+\Delta\sim_{\mathbb{R}}\pi^{*}D$ for an $\mathbb{R}$-Cartier divisor  $D$ on $Z$, and there is a component $S$ of $\llcorner \Delta \lrcorner$ dominating $Z$. 

Then $(X,\Delta)$ has a log minimal model if $(S,{\rm Diff}(\Delta-S))$ has a log minimal model.  
\end{lem}

To prove Lemma \ref{lemzariski}, recall \cite[Theorem 1.1]{bhzariski}.

\begin{thm}[cf.~{\cite[Theorem 1.1]{bhzariski}}]\label{thmzariski}
Let $(X,\Delta)$ be a projective lc pair. 
Then $(X,\Delta)$ has a log minimal model if and only if $K_{X}+\Delta$ is pseudo-effective and there is a resolution $f\colon Y\to X$ of $X$ such that the Nakayama--Zariski decomposition of $f^{*}(K_{X}+\Delta)$ has nef positive part. 
\end{thm}

\begin{proof}[Proof of Lemma \ref{lemzariski}]
The proof is similar to that of \cite[Theorem 4.2]{bh}. 
In this proof, $P_{\sigma}(\,\cdot\,)$ denotes the positive part of Nakayama--Zariski decomposition. 
We note that we have $P_{\sigma}(D_{1})\sim_{\mathbb{R}}P_{\sigma}(D_{2})$ for any two $\mathbb{R}$-Cartier divisors $D_{1}$ and $D_{2}$ such that $D_{1}\sim_{\mathbb{R}}D_{2}$. 

We set $\pi_{S}=\pi|_{S}\colon S\to Z$. 
Then $K_{S}+{\rm Diff}(\Delta-S)\sim_{\mathbb{R}}\pi_{S}^{*}D$ by hypothesis. 
Assume that $(S,{\rm Diff}(\Delta-S))$ has a log minimal model. 
By Theorem \ref{thmzariski}, there is a resolution $\tau\colon T\to S$ such that the divisor $P_{\sigma}(\tau^{*}\pi_{S}^{*}D)$ is nef. 
Let $h\colon Z'\to Z$ be a resolution. 
Thanks to \cite[III, 5.17 Corollary]{nakayama}, we may replace $T$ with a higher birational model, so we can assume that the induced map $\pi_{T}\colon T\dashrightarrow Z'$ is a morphism. 
Since $P_{\sigma}(\pi_{T}^{*}h^{*}D)$ is nef, by \cite[III, 5.18 Corollary]{nakayama}, there is a projective birational morphism $h'\colon Z''\to Z'$ from a smooth variety $Z''$ such that $P_{\sigma}(h'^{*}h^{*}D)$ is nef. 
Put $D''=h'^{*}h^{*}D$.

Take a resolution $f\colon \widetilde{X} \to X$ such that the induced map $\widetilde{\pi}\colon \widetilde{X}\dashrightarrow  Z''$ is a morphism. 
Then $P_{\sigma}(\widetilde{\pi}^{*}D'')$ is nef by \cite[III, 5.17 Corollary]{nakayama}. 
Since $f^{*}(K_{X}+\Delta)\sim_{\mathbb{R}}f^{*}\pi^{*}D=\widetilde{\pi}^{*}D''$, we see that $P_{\sigma}(f^{*}(K_{X}+\Delta))$ is nef. 
From this fact and by Theorem \ref{thmzariski}, we see that $(X,\Delta)$ has a log minimal model. 
\end{proof}

Finally, we prove Lemma \ref{lem--reduction}. 

\begin{proof}[Proof of Lemma \ref{lem--reduction}]
We may assume that $K_{X}+B+A$ is pseudo-effective. 
Suppose that $K_{X}+B-\epsilon \llcorner B \lrcorner+A$ is not pseudo-effective for any $\epsilon>0$. 
Thanks to Theorem \ref{thm--abund-hu}, we only have to prove the existence of a log minimal model of $(X,B+A)$. 
But then we can apply Step \ref{step2abund-hu}--\ref{step6abund-hu} in the proof of Theorem \ref{thm--gen-abund-hu}. 
As in Step  \ref{step2abund-hu}--\ref{step4abund-hu}, we construct
\begin{itemize}
\item[(1)]
a log resolution $\phi \colon \bar{X} \to X$ of $(X,B)$ and a log smooth lc pair $(\bar{X},\bar{B}+\bar{A})$, where $\bar{A}=\phi^{*}A$, such that we may write $K_{\bar{X}}+\bar{B}+\bar{A}=\phi^{*}(K_{X}+B+A)+\bar{E}$ with $\bar{E}\geq0$ such that $\bar{E}$ and $\bar{B}+\bar{A}$ have no common components, 
\item[(2)]
a sequence of steps of log MMP $(\bar{X},\bar{B}+\bar{A})\dashrightarrow (\bar{X}'',\bar{B}''+\bar{A}'')$ with a contraction $\bar{X}''\to Z''$ such that $K_{\bar{X}''}+\bar{B}''+\bar{A}''\sim_{\mathbb{R},Z''}0$ and $(X, B+A)$ has a log minimal model if $(\bar{X}'',\bar{B}''+\bar{A}'')$ has a log minimal model, and 
\item[(3)]
a component $S$ of $\llcorner B \lrcorner$ with the birational transform $\bar{S}''$ on $\bar{X}''$ such that the induced morphism $\bar{S}''\to Z''$ is surjective. 
\end{itemize}
Let $\bar{S}$ be the birational transform of $S$ on $\bar{X}$. 
We define $B_{S}$, $B_{\bar{S}}$, and $B_{\bar{S}''}$ by adjunctions $K_{S}+B_{S}=(K_{X}+B)|_{S}$, $K_{\bar{S}}+B_{\bar{S}}=(K_{\bar{X}}+\bar{B})|_{\bar{S}}$, and  $K_{\bar{S}''}+B_{\bar{S}''}=(K_{\bar{X}''}+\bar{B}'')|_{\bar{S}''}$, respectively. 
Set $A_{S}=A|_{S}$, $A_{\bar{S}}=\bar{A}|_{\bar{S}}$, and $A_{\bar{S}''}=\bar{A}''|_{\bar{S}''}$. 
Since $K_{\bar{X}''}+\bar{B}''+\bar{A}''\sim_{\mathbb{R},Z''}0$ and $\bar{S}''\to Z''$ is surjective, by Lemma \ref{lemzariski}, $(\bar{X}'',\bar{B}''+\bar{A}'')$ has a log minimal model if $(\bar{S}'',B_{\bar{S}''}+A_{\bar{S}''})$ has a log minimal model, so it is sufficient to prove that $(\bar{S}'',B_{\bar{S}''}+A_{\bar{S}''})$ has a log minimal model by $(2)$. 
By (1), construction of $\bar{E}$ and Lemma \ref{lem--adjunction}, we have 
$K_{\bar{S}}+B_{\bar{S}}+A_{\bar{S}}=\phi|_{\bar{S}}^{*}(K_{S}+B_{S}+A_{S})+\bar{E}|_{\bar{S}}$
and $\bar{E}|_{\bar{S}}$ is an effective $\phi|_{\bar{S}}$-exceptional divisor. 
Let $\tau''\colon T\to \bar{S}''$ and $\tau\colon T\to \bar{S}$ be a common log resolution of $(\bar{S},B_{\bar{S}})\dashrightarrow (\bar{S}'', B_{\bar{S}''})$. 
As in the latter part of Step \ref{step4abund-hu}, we can replace $A$ by a general member of its $\mathbb{R}$-linear system. 
By Lemma \ref{lem--genmember} and Lemma \ref{lem--genmember2} and by replacing $A$, we may assume $A_{\bar{S}}\geq0$, $A_{\bar{S}''}\geq0$, $\tau^{*}A_{\bar{S}}\leq \tau_{*}''^{-1}A_{\bar{S}''}$, and $\tau''\colon T\to \bar{S}''$ is a log resolution of $(\bar{S}'', B_{\bar{S}''}+A_{\bar{S}''})$. 
Putting $A_{T}=\tau^{*}A_{\bar{S}}$, we may write 
$
K_{T}+\Psi+A_{T}=\tau''^{*}(K_{\bar{S}''}+B_{\bar{S}''}+A_{\bar{S}''})+E_{T}$
with $\Psi\geq0$ and $E_{T}\geq0$ such that $(T, \Psi+A_{T})$ is a $\mathbb{Q}$-factorial dlt pair and $\Psi+A_{T}$ and $E_{T}$ have no common components. 
Calculations as in Step \ref{step6abund-hu} show that the effective part of  
$K_{T}+\Psi+A_{T}-\tau^{*}\phi|_{\bar{S}}^{*}(K_{S}+B_{S}+A_{S})$ is $(\phi|_{\bar{S}} \circ \tau)$-exceptional. 
Finally, by the induction hypothesis of Theorem \ref{thmmain--hu} and by Lemma \ref{lem--ind-birat}, we see that  $(T, \Psi+A_{T})$ has a log minimal model. 
Then $(\bar{S}'',B_{\bar{S}''}+A_{\bar{S}''})$ also has a log minimal model, and so does $(\bar{X}'',\bar{B}''+\bar{A}'')$. 
Existence of a log minimal model of $(X,B+A)$ follows from this fact. 
\end{proof}

\subsection{Proof of Theorem \ref{thmmain--hu}: General case}

From now on, we prove Theorem \ref{thmmain--hu} in full generality. 

\begin{proof}[Proof of Theorem \ref{thmmain--hu}]
We prove the theorem by induction on ${\rm dim}X$. 
Put $n={\rm dim}X$, and suppose that Theorem \ref{thmmain--hu} holds for all projective lc pairs of dimension $\leq n-1$. By Lemma \ref{lem--thmmainklt}, we may assume that $(X,B)$ is not klt. 
By Theorem \ref{thm--abund-hu}, it is sufficient to show the existence of log minimal model of $(X,B+A)$ under the assumption that $K_{X}+B+A$ is pseudo-effective. 
By taking a dlt model, we may assume that $(X,B+A)$ is $\mathbb{Q}$-factorial dlt. 
By Theorem \ref{thm--gen-abund-hu}, the divisor $K_{X}+B+A$ is abundant. 
We prove Theorem \ref{thmmain--hu} with several steps. 

\begin{step4}\label{step1main-hu}
By Lemma \ref{lem--reduction}, we may assume there is $u>0$ such that $K_{X}+B-u \llcorner B\lrcorner +A$ is pseudo-effective. 
By using Lemma \ref{lem--thmmainklt}, we can find $D\geq0$ such that $D\sim_{\mathbb{R}}K_{X}+B+A$ and ${\rm Supp}D\supset {\rm Supp}\llcorner B\lrcorner$. 
Let $f\colon \widetilde{X}\to X$ be a log resolution of $(X,{\rm Supp}(B+D))$, and let $(\widetilde{X},\widetilde{B})$ be a log smooth model of $(X,B)$ (for the definition of log smooth model, see \cite[Definition 2.9]{has-trivial}). 
We can write $f^{*}D=\widetilde{G}+\widetilde{H}$ with $\widetilde{G}\geq0$ and $\widetilde{H}\geq0$ such that $\widetilde{G}$ and $\widetilde{H}$ have no common components and ${\rm Supp}\widetilde{G}={\rm Supp}\llcorner \widetilde{B}\lrcorner$. 
Because $\widetilde{H}$ and $\llcorner \widetilde{B} \lrcorner$ have no common components and $(\widetilde{X},{\rm Supp}(\widetilde{B}+\widetilde{H}))$ is log smooth, there is $\epsilon>0$ such that $(\widetilde{X},\widetilde{B}+\epsilon  \widetilde{H})$ is dlt. 
We construct a dlt blow-up $f'\colon (\widetilde{X}',\widetilde{B}')\to (X,B)$ by running a $(K_{\widetilde{X}}+\widetilde{B})$-MMP over $X$. 
Let $\widetilde{G}'$ and $\widetilde{H}'$ be the birational transforms of $\widetilde{G}$ and $\widetilde{H}$ on $\widetilde{X}'$, respectively. 
By replacing $\epsilon$ with a smaller one, we may assume that $\widetilde{X}\dashrightarrow \widetilde{X}'$ is a sequence of steps of the $(K_{\widetilde{X}}+\widetilde{B}+\epsilon \widetilde{H})$-MMP. 
Then $(\widetilde{X}',\widetilde{B}'+\epsilon  \widetilde{H}')$ is dlt. 
It is easy to check that we may replace $X\to Z$, $(X,B)$, $C$, $A_{Z}$ and $A$ by $\widetilde{X}'\to Z$, $(\widetilde{X}',\widetilde{B}')$, $f'^{*}C$, $A_{Z}$ and $f'^{*}A$, respectively. 

In this way, we may assume that there is an effective $\mathbb{R}$-divisors $G$ and $H$ such that 
\begin{itemize}
\item
$K_{X}+B+A\sim_{\mathbb{R}}G+H$, 
\item
${\rm Supp}G= {\rm Supp}\llcorner B \lrcorner$, 
\item
there is a real number $\epsilon>0$ such that $(X,B+\epsilon H)$ is $\mathbb{Q}$-factorial dlt. 
\end{itemize}
\end{step4}

\begin{step4}\label{step2main-hu}

By choosing $\epsilon>0$ in Step \ref{step1main-hu} sufficiently small, we can assume that $B-\epsilon G\geq0$. 
By replacing $A$, we may also assume that $A$ and $B+H$ have no common components,  $(X,B+A+\epsilon H)$ is dlt and any lc center of the pair is an lc center of $(X,B)$. 
In this step, we construct a strictly decreasing sequence $\{e_{i}\}_{i\geq1}$ of real numbers and a sequence of birational maps
\begin{equation*}
X \dashrightarrow X_{1}\dashrightarrow X_{2}\dashrightarrow \cdots \dashrightarrow X_{i}\dashrightarrow \cdots
\end{equation*}
such that
\begin{enumerate}
\item
$0<e_{i}<\epsilon$ for any $i$ and ${\rm lim}_{i\to \infty}e_{i}=0$, 
\item
$X \dashrightarrow X_{1}$ is a sequence of steps of a $(K_{X}+B+A+e_{1}H)$-MMP to a good minimal model, 
\item
for any $i\geq1$, the pair $(X_{i},B_{i}+A_{i}+e_{i}H_{i})$ is a good minimal model of both pairs $(X_{1},B_{1}+A_{1}+e_{i}H_{1})$ and $(X,B+A+e_{i}H)$, and
\item
for each $i\geq1$, the map $X_{i}\dashrightarrow X_{i+1}$ is an isomorphism or a sequence of steps of a $(K_{X_{i}}+B_{i}+A_{i}+e_{i+1}H_{i})$-MMP with scaling of $(e_{i}-e_{i+1})H_{i}$. 
\end{enumerate}
Here $B_{i}$, $A_{i}$ and $H_{i}$ are the birational transforms of $B$, $A$ and $H$ on $X_{i}$, respectively. 
By condition (4), the sequence of birational maps $X_{1}\dashrightarrow \cdots$ is a sequence of steps of a $(K_{X_{1}}+B_{1}+A_{1})$-MMP with scaling of $e_{1}H_{1}$.  

Pick a strictly decreasing infinite sequence $\{e_{i}\}_{i\geq1}$ of positive real numbers such that $e_{i}<\epsilon$ for any $i\geq1$ and ${\rm lim}_{i\to \infty}e_{i}=0$. 
Then $(X,B+A+e_{i}H)$ is dlt, $(X,B+A-\frac{e_{i}}{1+e_{i}}G)$ is klt and 
\begin{equation*}
K_{X}+B+A+e_{i}H\sim_{\mathbb{R}}(1+e_{i})\Bigl(K_{X}+B+A-\frac{e_{i}}{1+e_{i}}G\Bigr).
\end{equation*}
By Lemma \ref{lem--thmmainklt} and running a $(K_{X}+B+A+e_{i}H)$-MMP we get a good minimal model
$(X,B+A+e_{i}H)\dashrightarrow (X_{i},B_{i}+A_{i}+e_{i}H_{i})$. 
Then, any prime divisor contracted by the map $X\dashrightarrow X_{i}$ is a component of $G+H$, which does not depend on $i$. 
Therefore, by replacing $\{e_{i}\}_{i\geq1}$ with a subsequence, we may assume that all $X \dashrightarrow X_{i}$ contract the same divisors. 
Then all $X_{i}$ are isomorphic in codimension one. 

We show that $(X_{1},B_{1}+A_{1}+tH_{1})$ has a good minimal model for any $t\in(0,e_{1})$. 
Since $K_{X_{1}}+B_{1}+A_{1}+tH_{1}\sim_{\mathbb{R}}(1+t)(K_{X_{1}}+B_{1}+A_{1}-\frac{t}{1+t}G_{1}),$ it is sufficient to show that the pair $(X_{1},B_{1}+A_{1}-\frac{t}{1+t}G_{1})$ has a good minimal model, where $G_{1}$ is the birational transform of $G$ on $X_{1}$. 
Note that $(X_{1},B_{1}+A_{1}-\frac{t}{1+t}G_{1})$ is klt. 
By Lemma \ref{lem--abundantmmp} and the above relation, it is sufficient to show that $K_{X_{1}}+B_{1}+A_{1}+tH_{1}$ is abundant. 
We have $K_{X_{1}}+B_{1}+A_{1}+tH_{1}\sim_{\mathbb{R}}G_{1}+(1+t)H_{1}$. 
By Remark \ref{remdiv} (1), $K_{X_{1}}+B_{1}+A_{1}+tH_{1}$ is abundant if and only if $K_{X_{1}}+B_{1}+A_{1}+e_{1}H_{1}$ is abundant. 
Since the map $X\dashrightarrow X_{1}$ is a sequence of steps of the $(K_{X}+B+A+e_{1}H)$-MMP, it is sufficient to show that $K_{X}+B+A+e_{1}H$ is abundant. 
Using the relation $K_{X}+B+A+e_{1}H\sim_{\mathbb{R}}G+(1+e_{1})H$ and Remark \ref{remdiv} (1), we see that $K_{X}+B+A+e_{1}H$ is abundant if and only if $K_{X}+B+A$ is abundant, which is true by Theorem \ref{thm--gen-abund-hu}. 
Therefore, we see that $K_{X_{1}}+B_{1}+A_{1}+tH_{1}$ is abundant, and 
$(X_{1},B_{1}+A_{1}+tH_{1})$ has a good minimal model for any $t\in (0, e_{1})$. 

Put $X'_{1}=X_{1}$ (resp.~$B'_{1}=B_{1}$, $A'_{1}=A_{1}$, $H'_{1}=H_{1}$). 
By \cite[Lemma 2.14]{has-mmp}, we can run a $(K_{X'_{1}}+B'_{1}+A'_{1})$-MMP with scaling of $e_{1}H'_{1}$
\begin{equation*}
(X'_{1},B'_{1}+A'_{1})\dashrightarrow \cdots \dashrightarrow (X'_{j},B'_{j}+A'_{j})\dashrightarrow \cdots
\end{equation*}
such that if we set $\lambda_{j}={\rm inf}\!\set{\!\mu\in\mathbb{R}_{\geq0} | \text{$K_{X'_{j}}+B'_{j}+A'_{j}+\mu H'_{j}$ is nef}\!}$, where $H'_{j}$ is the birational transform of $H'_{1}$ on $X'_{j}$, then the $(K_{X'_{1}}+B'_{1}+A'_{1})$-MMP terminates after finitely many steps or we have ${\rm lim}_{j\to \infty}\lambda_{j}=0$ when it does not terminate. 

For any $i\geq1$, pick the minimum $k_{i}$ such that $K_{X'_{k_{i}}}+B'_{k_{i}}+A'_{k_{i}}+e_{i}H'_{k_{i}}$ is nef. 
Such $k_{i}$ exists since ${\rm lim}_{j\to \infty}\lambda_{j}=0$, and we have $k_{1}=1$. 
By the same argument as in Step \ref{step4relative} in the proof of Theorem \ref{thmmain}, we see that the pair $(X'_{k_{i}},B'_{k_{i}}+A'_{k_{i}}+e_{i}H'_{k_{i}})$ is a good minimal model of $(X'_{1}, B'_{1}+A'_{1}+e_{i}H'_{1})$ and it is also a good minimal model of $(X, B+A+e_{i}H)$. 

By abuse of notations, we put $X_{i}=X'_{k_{i}}$ (resp.~$B_{i}=B'_{k_{i}}$, $A_{i}=A'_{k_{i}}$, $H_{i}=H'_{k_{i}}$). 
Then we can check that  $\{e_{i}\}_{i\geq1}$ and 
\begin{equation*}
X \dashrightarrow X_{1}\dashrightarrow X_{2}\dashrightarrow \cdots \dashrightarrow X_{i}\dashrightarrow \cdots
\end{equation*}
satisfy (1), (2), (3) and (4) stated at the start of this step. 
Indeed, (1) and (2) follow from $k_{1}=1$ and the argument in the second paragraph.  
The condition (4) follows from the argument in the fourth paragraph and $X_{i}=X'_{k_{i}}$. 
The condition (3) follows from the fifth paragraph. 
\end{step4}

\begin{step4}\label{step3main-hu}
As in Step \ref{step5relative} in the proof of Theorem \ref{thmmain}, if the $(K_{X_{1}}+B_{1}+A_{1})$-MMP with scaling of $e_{1}H_{1}$ terminates, then $(X,B+A)$ has a log minimal model. 
Thus we only have to prove termination of the $(K_{X_{1}}+B_{1}+A_{1})$-MMP. 

Suppose by contradiction that the $(K_{X_{1}}+B_{1}+A_{1})$-MMP does not terminate. 
Since $K_{X_{1}}+B_{1}+A_{1}+e_{1}H_{1}\sim_{\mathbb{R}}(1+e_{1})\bigl(K_{X_{1}}+B_{1}+A_{1}-\frac{e_{1}}{1+e_{1}}G_{1}\bigr)$ and ${\rm Supp}G_{1}={\rm Supp}\llcorner B_{1}\lrcorner$, this log MMP occurs only in ${\rm Supp}\llcorner B_{1}\lrcorner$ (cf.~\cite[Step 2 in the proof of Proposition 5.4]{has-trivial}). 
We will get a contradiction by the argument of the special termination as in \cite{fujino-sp-ter}. 
\end{step4}

\begin{step4}\label{step4main-hu}
The goal of this step is to prove inequality $(\spadesuit)$ stated below. 
The argument is almost same as the second paragraph of Step \ref{step7relative} in the proof of Theorem \ref{thmmain}. 

Note that $(X_{1},B_{1}+A_{1}+e_{1}H_{1})$ is $\mathbb{Q}$-factorial dlt and any lc center of the pair is an lc center of $(X_{1},B_{1})$. 
So $(X_{i},B_{i}+A_{i}+e_{i}H_{i})$ is $\mathbb{Q}$-factorial dlt and any lc center of the pair is an lc center of $(X_{i},B_{i})$ for any $i\geq 1$. 
For any lc center $S$ of $(X,B)$, we set $A_{S}=A|_{S}$ and $H_{S}=H|_{S}$, and we define $B_{S}$ by adjunction $K_{S}+B_{S}=(K_{X}+B)|_{S}$. 
Note that $A_{S}$ and $H_{S}$ are effective $\mathbb{R}$-Cartier divisors, and $K_{S}+B_{S}$ is $\mathbb{R}$-Cartier.
Similarly, for any $i\geq 1$ and any lc center $S_{i}$ of $(X_{i},B_{i})$, we set $A_{S_{i}}=A_{i}|_{S_{i}}$ and $H_{S_{i}}=H_{i}|_{S_{i}}$, and we define $B_{S_{i}}$ by adjunction $K_{S_{i}}+B_{S_{i}}=(K_{X_{i}}+B_{i})|_{S_{i}}$. 
Then $A_{S_{i}}$ and $H_{S_{i}}$ are both effective $\mathbb{R}$-Cartier divisors, and $K_{S_{i}}+B_{S_{i}}$ is $\mathbb{R}$-Cartier. 
Moreover, the pairs $(S,B_{S}+A_{S}+e_{i}H_{S})$ and $(S_{i}, B_{S_{i}}+A_{S_{i}}+e_{i}H_{S_{i}})$ are all dlt. 

By construction of the map $(X,B+A+e_{i}H)\dashrightarrow(X_{i},B_{i}+A_{i}+e_{i}H_{i})$ (see (2) and (4) in Step \ref{step2main-hu}), for any $i\geq 1$ and any lc center $S_{i}$ of $(X_{i},B_{i})$, there is an lc center $S$ of $(X,B)$ such that the map $X\dashrightarrow X_{i}$  induces a birational map $S\dashrightarrow S_{i}$. 
By (2) and (4) in Step \ref{step2main-hu}, there is a common resolution $\overline{X}\to X$ and $\overline{X}\to {X_{i}}$ of the map $X\dashrightarrow X_{i}$ and a subvariety $\overline{S}\subset \overline{X}$ birational to $S$ and $S_{i}$ such that the induced morphisms $\overline{S}\to S$ and $\overline{S}\to S_{i}$ form a common resolution of $S\dashrightarrow S_{i}$. 
Because of condition (3) in Step \ref{step2main-hu}, by restricting the pullbacks of $K_{X}+B+A+e_{i}H$ and $K_{X_{i}}+B_{i}+A_{i}+e_{i}H_{i}$ to $\overline{S}$ and comparing coefficients of those restricted divisors, we obtain 
\begin{equation*}\tag{$\spadesuit$}
a(Q,S,B_{S}+A_{S}+e_{i}H_{S})\leq a(Q,S_{i},B_{S_{i}}+A_{S_{i}}+e_{i}H_{S_{i}})
\end{equation*}
 for any prime divisor $Q$ over $S$. 
\end{step4}

\begin{step4}\label{step5main-hu}
Now we can write $A=\sum_{l}r_{l}A^{(l)}$, where $0<r_{l}<1$ and $A^{(l)}$ are base point free Cartier divisors and general in their linear systems. 
The goal of this step is to replace $A$ with a sufficiently general one by replacing each $A^{(l)}$ with a sufficiently general member of its linear system so that all dlt pairs $(S_{i},B_{S_{i}}+A_{S_{i}}+e_{i}H_{S_{i}})$ have good properties. 
We note that all $(X_{i},B_{i}+A_{i}+e_{i}H_{i})$ are $\mathbb{Q}$-factorial dlt pairs with the same lc centers as lc centers of $(X_{i},B_{i})$, and $B+H$ and $A$ have no common components as far as all $A^{(l)}$ are general. 
Also, we note that all the log MMP and the good minimal models constructed in Step \ref{step2main-hu} and the inequalities of discrepancies in Step \ref{step4main-hu} do not depend on $\mathbb{R}$-linear equivalence class of $A$.
Therefore, after replacing $A$ and replacing $A_{i}$, $A_{S}$ and $A_{S_{i}}$ accordingly, we need not change other divisors, varieties and pairs. 

We fix an lc center $S_{i}$ of $(X_{i},B_{i})$, and let $S$ be an lc center of $(X,B)$ such that the map $X\dashrightarrow X_{i}$  induces a birational map $S\dashrightarrow S_{i}$. 
Let $\mathcal{C}$ be the set of components of $H_{S}$. 
For any boundary $\mathbb{R}$-divisor $\Delta_{S_{i}}$ such that $(S_{i},\Delta_{S_{i}})$ is dlt, we consider the set
\begin{equation*} 
\mathcal{C}_{\Delta_{S_{i}}}=\Set{ Q' 
| \begin{array}{l}\text{$Q'$ is a component of $H_{S}$ which is exceptional over $S_{i}$}\\ \text{such that $-1<a(Q',S_{i}, \Delta_{S_{i}})<0$} 
 \end{array}\!}.
 \end{equation*}
Note that we may have $\mathcal{C}_{\Delta_{S_{i}}}=\emptyset$. 
Since $\mathcal{C}_{\Delta_{S_{i}}}\subset \mathcal{C}$, 
there are only finitely many subsets $\mathcal{C}_{1},\cdots, \mathcal{C}_{j_{0}}$ of $\mathcal{C}$ such that for any dlt pair $(S_{i},\Delta_{S_{i}})$, there is $1\leq j\leq j_{0}$ satisfying $\mathcal{C}_{\Delta_{S_{i}}}=\mathcal{C}_{j}$. 
By discarding some $\mathcal{C}_{j}$, we may assume that for any $1\leq j\leq j_{0}$ there is a dlt pair $(S_{i},\Delta^{(j)})$ such that $\mathcal{C}_{j}=\mathcal{C}_{\Delta^{(j)}}$. 
Applying Lemma \ref{lem--extraction} to the dlt pairs $(S_{i},\Delta^{(j)})$, we obtain finitely many projective birational morphisms $\psi_{S_{i},j}\colon \bar{S}_{i,j}\to S_{i}$ such that $(\bar{S}_{i,j},0)$ are $\mathbb{Q}$-factorial klt and $\psi_{S_{i},j}^{-1}$ exactly extracts elements of $\mathcal{C}_{j}$. 
By construction, for any dlt pair $(S_{i},\Delta_{S_{i}})$, there is $1\leq j\leq j_{0}$ such that $\psi_{S_{i},j}^{-1}$ exactly extracts elements of $\mathcal{C}_{\Delta_{S_{i}}}$. 
For each $j$, let $E_{\bar{S}_{i,j}}$ be the sum of all $\psi_{S_{i},j}$-exceptional prime divisors and all exceptional prime divisors of the induced map $\bar{S}_{i,j}\dashrightarrow S$. 
We fix a log resolution $\tau_{S_{i},j}\colon \widetilde{S}_{i,j} \to \bar{S}_{i,j}$ of $\bigl(\bar{S}_{i,j}, \psi_{S_{i},j*}^{-1}(B_{S_{i}}+H_{S_{i}})+E_{\bar{S}_{i,j}}\bigr)$ such that the induced birational map $\sigma_{S_{i},j}\colon \widetilde{S}_{i,j}\dashrightarrow S$ is a morphism. 
\begin{equation*}
\xymatrix
{
\widetilde{S}_{i,j}\ar[d]_{\sigma_{S_{i},j}} \ar[r]^{\tau_{S_{i},j}}&\bar{S}_{i,j} \ar[d]^{\psi_{S_{i},j}}\\
S\ar@{-->}[r]&S_{i}
}
\end{equation*}
Note that the morphisms $\psi_{S_{i},j}$, $\tau_{S_{i},j}$ and $\sigma_{S_{i},j}$ do not depend on $A$. 
By Lemma \ref{lem--genmember} and Lemma \ref{lem--genmember2} and by replacing $A$ with a general one, for any $j$
we may assume that 
\begin{itemize}
\item
$\sigma_{S_{i},j}^{*}A_{S}=\sigma_{S_{i},j*}^{-1}A_{S}$ and 
$\sigma_{S_{i},j}^{*}A_{S}\leq (\psi_{S_{i},j}\circ \tau_{S_{i},j})_{*}^{-1}A_{S_{i}}$, and 
\item
$\tau_{S_{i},j}$ is a log resolution of $\bigl(\bar{S}_{i,j}, \psi_{S_{i},j*}^{-1}(B_{S_{i}}+A_{S_{i}}+H_{S_{i}})+E_{\bar{S}_{i,j}}\bigr)$. 
\end{itemize}

We apply the above discussion to all $S_{i}$ and we replace $A$ with a sufficiently general one. 
Then $A_{S}$ and $A_{S_{i}}$ satisfy the above two conditions for all $S_{i}$ and $j$. 
For each $S_{i}$ and the set $\mathcal{C}_{B_{S_{i}}+A_{S_{i}}+e_{i}H_{S_{i}}}$, we choose the corresponding index $j$ and we set 
\begin{equation*}
\begin{split}
T_{i}=\bar{S}_{i,j},&\quad \psi_{S_{i}}=\psi_{S_{i},j}\colon T_{i}\to S_{i}, \\
\widetilde{T}_{i}=\widetilde{S}_{i,j},&\quad\tau_{S_{i}}=\tau_{S_{i},j}\colon \widetilde{T}_{i} \to T_{i}\quad {\rm and} \quad \sigma_{S_{i}}=\sigma_{S_{i},j}\colon \widetilde{T}_{i}\to S.
\end{split}
\end{equation*} 
We define $\Psi_{T_{i}}$ by equation $K_{T_{i}}+\Psi_{T_{i}}=\psi_{S_{i}}^{*}(K_{S_{i}}+B_{S_{i}}+A_{S_{i}}+e_{i}H_{S_{i}})$. 
By construction, we have $\Psi_{T_{i}}\geq 0$, $(T_{i},0)$ is $\mathbb{Q}$-factorial klt and $(T_{i},\Psi_{T_{i}})$ is lc. 
By the second condition stated above, the morphism $\tau_{S_{i}}\colon \widetilde{T}_{i} \to T_{i}$ is a log resolution of $(T_{i},\Psi_{T_{i}})$. 
We note that the inclusion ${\rm Supp}\Psi_{T_{i}}\subset {\rm Supp}\bigl(\psi_{S_{i}*}^{-1}(B_{S_{i}}+A_{S_{i}}+e_{i}H_{S_{i}})+E_{T_{i}}\bigr)$ holds, where $E_{T_{i}}$ is the sum of all $\psi_{S_{i}}$-exceptional prime divisors and all exceptional prime divisors of $T_{i}\dashrightarrow S$. 
By the first condition stated above, we have $\sigma_{S_{i}}^{*}A_{S}=\sigma_{S_{i}*}^{-1}A_{S}$ and $\sigma_{S_{i}}^{*}A_{S}\leq (\psi_{S_{i}}\circ \tau_{S_{i}})_{*}^{-1}A_{S_{i}}$. 

From the above argument, by replacing $A$, we may assume that for any $i\geq1$ and any lc center $S_{i}$ of $(X_{i},B_{i})$ with corresponding lc center $S$ of $(X,B)$, there is a diagram  
\begin{equation*}
\xymatrix
{
\widetilde{T}_{i}\ar[d]_{\sigma_{S_{i}}} \ar[r]^{\tau_{S_{i}}}&T_{i} \ar[d]^{\psi_{S_{i}}}\\
S\ar@{-->}[r]&S_{i}
}
\end{equation*}
such that 
\begin{itemize}
\item[(a)]
$\psi_{S_{i}}^{-1}$ exactly extracts all components $Q'$ of $H_{S}$ which are exceptional over $S_{i}$ such that $-1<a(Q',S_{i},B_{S_{i}}+A_{S_{i}}+e_{i}H_{S_{i}} )<0$,
\item[(b)]
$(T_{i},0)$ is $\mathbb{Q}$-factorial klt,
\item[(c)]
if we define $\Psi_{T_{i}}$ by an equation $K_{T_{i}}+\Psi_{T_{i}}=\psi_{S_{i}}^{*}(K_{S_{i}}+B_{S_{i}}+A_{S_{i}}+e_{i}H_{S_{i}})$, then $(T_{i},\Psi_{T_{i}})$ is lc and $\tau_{S_{i}}$ is a log resolution of $(T_{i},\Psi_{T_{i}})$, and
\item[(d)]
$\sigma_{S_{i}}^{*}A_{S}=\sigma_{S_{i}*}^{-1}A_{S}$ and $\sigma_{S_{i}}^{*}A_{S}\leq (\psi_{S_{i}}\circ \tau_{S_{i}})_{*}^{-1}A_{S_{i}}$. 
\end{itemize}
This diagram and conditions will be used in steps \ref{step7main-hu}, \ref{step8main-hu} and  \ref{step9main-hu}. 
\end{step4}

\begin{step4}\label{step6main-hu}
We start the argument of the special termination. 
The basic strategy is similar to \cite[Proof of Theorem 1.2]{birkar-09}. 

There is $m>0$ such that for any lc center $S_{m}$ of $(X_{m},B_{m})$ and any $i \geq m$, the indeterminacy locus of the birational map $X_{m}\dashrightarrow X_{i}$ does not contain $S_{m}$ and the restriction of the map to $S_{m}$ induces a birational map $S_{m}\dashrightarrow S_{i}$ to an lc center $S_{i}$ of $(X_{i},B_{i})$. 
From now on, we prove that for any lc center $S_{m}$ of $(X_{m},B_{m})$, there is $i_{0}\geq m$ such that the induced map $(S_{i}, B_{S_{i}}+A_{S_{i}})\dashrightarrow (S_{i+1}, B_{S_{i+1}}+A_{S_{i+1}})$ is an isomorphism for any $i\geq i_{0}$. 
Assuming this, then we see that the $(K_{X_{1}}+B_{1}+A_{1})$-MMP terminates by the same argument as in \cite{fujino-sp-ter}, and we get a contradiction. 
So we prove the assertion, and we prove it by induction on the dimension of $S_{m}$. 
By arguments as in \cite{fujino-sp-ter}, replacing $m$, we may assume that the induced birational map $S_{m}\dashrightarrow S_{i}$ is small and the birational transforms of $B_{S_{m}}$ and $A_{S_{m}}$ on $S_{i}$ are equal to $B_{S_{i}}$ and $A_{S_{i}}$, respectively. 
Then we can easily check that it is sufficient to prove that $K_{S_{i}}+B_{S_{i}}+A_{S_{i}}$ is nef for some $i\geq m$ (see \cite[Step 5 in the proof of Theorem 3.5]{has-class}). 
\end{step4}

\begin{step4}\label{step7main-hu}
In the rest of the proof, unless otherwise stated all $i$ are assumed to be $i\geq m$. 
In this step, we define some varieties and divisors used in the proof. 

For each $i$, let $\psi_{S_{i}}\colon (T_{i},\Psi_{T_{i}})\to (S_{i},B_{S_{i}}+A_{S_{i}}+e_{i}H_{S_{i}})$ be the model constructed in Step \ref{step5main-hu}. 
By (4) in Step \ref{step2main-hu} and by \cite[Lemma 4.2.10]{fujino-sp-ter}, we have 
\begin{equation*}
\begin{split}
a(Q, S_{i},B_{S_{i}}+A_{S_{i}}+e_{i}H_{S_{i}})&\leq a(Q,S_{i},B_{S_{i}}+A_{S_{i}}+e_{i+1}H_{S_{i}})\\
&\leq a(Q,S_{i+1},B_{S_{i+1}}+A_{S_{i+1}}+e_{i+1}H_{S_{i+1}})
\end{split}
\end{equation*}
for any prime divisor $Q$ over $S_{i}$. 
By (a) in Step \ref{step5main-hu}, the induced birational map $T_{i}\dashrightarrow T_{i+1}$ is a birational contraction. 
Replacing $m$, we may assume that $T_{m}$ and $T_{i}$ are isomorphic in codimension one for all $i$. 
Let $\Psi_{i}^{(m)}$ be the birational transform of $\Psi_{T_{i}}$ on $T_{m}$. 
By the above relation and since we have $K_{T_{i}}+\Psi_{T_{i}}=\psi_{S_{i}}^{*}(K_{S_{i}}+B_{S_{i}}+A_{S_{i}}+e_{i}H_{S_{i}})$, we have 
\begin{equation*}
\Psi_{T_{m}}\geq \Psi_{m+1}^{(m)}\geq \cdots \geq \Psi_{i}^{(m)} \geq \cdots \geq0.
\end{equation*}
Thus the limit $\Psi_{\infty}^{(m)}:={\rm lim}_{i \to \infty}\Psi_{i}^{(m)}$ exists as an $\mathbb{R}$-divisor. 
Then the pair $(T_{m},\Psi_{\infty}^{(m)})$ is $\mathbb{Q}$-factorial lc. 

By (4) in Step \ref{step2main-hu}, we see that $(X_{m},B_{m}+A_{m}+e_{i}H_{m})\dashrightarrow (X_{i},B_{i}+A_{i}+e_{i}H_{i})$ is a sequence of steps of the $(K_{X_{m}}+B_{m}+A_{m}+e_{i}H_{m})$-MMP with scaling of $(e_{m}-e_{i})H_{m}$ to a good minimal model. 
By \cite[Lemma 4.2.10]{fujino-sp-ter}, we have 
\begin{equation*}
\begin{split}
a(Q, S_{m},B_{S_{m}}+A_{S_{m}}+e_{i}H_{S_{m}})&\leq a(Q,S_{i},B_{S_{i}}+A_{S_{i}}+e_{i}H_{S_{i}})=a(Q,T_{i},\Psi_{T_{i}})
\end{split}
\end{equation*}
for any prime divisor $Q$ over $S_{m}$. 
Since the map $T_{m}\dashrightarrow T_{i}$ is isomorphic in codimension one, we obtain $\psi_{S_{m}}^{*}(K_{S_{m}}+B_{S_{m}}+A_{S_{m}}+e_{i}H_{S_{m}})\geq K_{T_{m}}+\Psi_{i}^{(m)}$ for any $i$. 
By considering the limit $i\to \infty$, for any prime divisor $Q$ over $S_{m}$, we have 
\begin{equation*}\tag{$\clubsuit$}
a(Q,S_{m},B_{S_{m}}+A_{S_{m}})\leq  a(Q,T_{m}, \Psi_{\infty}^{(m)}).
\end{equation*}
\end{step4}

\begin{step4}\label{step8main-hu}
In this step, we prove that $(T_{m}, \Psi_{\infty}^{(m)})$ has a good minimal model. 
Note that the divisor $K_{T_{i}}+\Psi_{T_{i}}=\psi_{S_{i}}^{*}(K_{S_{i}}+B_{S_{i}}+A_{S_{i}}+e_{i}H_{S_{i}})$ is semi-ample, and therefore the divisor $K_{T_{m}}+\Psi_{\infty}^{(m)}$, which is the limit of movable divisors $K_{T_{m}}+\Psi_{i}^{(m)}$, is pseudo-effective. 

We use the diagram  
\begin{equation*}
\xymatrix
{
\widetilde{T}_{m}\ar[d]_{\sigma_{S_{m}}} \ar[r]^{\tau_{S_{m}}}&T_{m} \ar[d]^{\psi_{S_{m}}}\\
S\ar@{-->}[r]&S_{m}
}
\end{equation*}
and conditions (a)--(d) in Step \ref{step5main-hu}. 
Since the birational transform of $A_{S_{m}}$ on $S_{i}$ is $A_{S_{i}}$, by construction of $\Psi_{\infty}^{(m)}$, we have $\Psi_{\infty}^{(m)}\geq \psi_{S_{m}*}^{-1}A_{S_{m}}$. 
By (d) in Step \ref{step5main-hu}, we can write
\begin{equation*}
K_{\widetilde{T}_{m}}+\Theta+\sigma_{S_{m}}^{*}A_{S}=\tau_{S_{m}}^{*}(K_{T_{m}}+\Psi_{\infty}^{(m)})+E
\end{equation*}
with $\Theta\geq 0$ and a $\tau_{S_{m}}$-exceptional $\mathbb{R}$-divisor $E\geq 0$ such that $\Theta+\sigma_{S_{m}}^{*}A_{S}$ and $E$ have no common components. 
It is sufficient to prove that $(\widetilde{T}_{m}, \Theta+\sigma_{S_{m}}^{*}A_{S})$ has a good minimal model. 
Note that $(\widetilde{T}_{m}, \Theta+\sigma_{S_{m}}^{*}A_{S})$ is an lc pair by (c) in Step \ref{step5main-hu}. 
We may write
\begin{equation*}\tag{$*$}
K_{\widetilde{T}_{m}}+\Theta+\sigma_{S_{m}}^{*}A_{S}=\sigma_{S_{m}}^{*}(K_{S}+B_{S}+A_{S})+M-N
\end{equation*}
with $M\geq 0$ and $N\geq0$ which have no common components. 
Thanks to Lemma \ref{lem--ind-birat}, it is sufficient to prove that $M$ is $\sigma_{S_{m}}$-exceptional, which is equivalent to that any prime divisor $Q$ on $S$ satisfies $a(Q,S,B_{S}+A_{S})\leq a(Q, \widetilde{T}_{m},\Theta+\sigma_{S_{m}}^{*}A_{S})$. 

Pick any prime divisor $Q$ on $S$. 
Then $a(Q, \widetilde{T}_{m},\Theta+\sigma_{S_{m}}^{*}A_{S})={\rm min}\{0, a(Q,T_{m}, \Psi_{\infty}^{(m)})\}$ by construction of $(\widetilde{T}_{m}, \Theta+\sigma_{S_{m}}^{*}A_{S})$ and Lemma \ref{lem--discre}, and $a(Q,S,B_{S}+A_{S})\leq 0$. 
So, we only need to prove $a(Q,S,B_{S}+A_{S})\leq a(Q,T_{m}, \Psi_{\infty}^{(m)})$. 
The calculation is the same as in the second paragraph of Step \ref{step8relative} in the proof of Theorem \ref{thmmain}. 
If $Q$ is not a component of $H_{S}$, we have 
\begin{equation*}
\begin{split}
a(Q,S,B_{S}+A_{S})&=a(Q,S,B_{S}+A_{S}+e_{m}H_{S})\\
&\leq a(Q,S_{m},B_{S_{m}}+A_{S_{m}}+e_{m}H_{S_{m}})\leq a(Q,S_{m},B_{S_{m}}+A_{S_{m}})\\
&\leq  a(Q,T_{m}, \Psi_{\infty}^{(m)}).
\end{split}
\end{equation*}
Here, the first equality follows from that $Q$ is a divisor on $S$ and not a component of $H_{S}$, the second inequality follows from ($\spadesuit$) in Step \ref{step4main-hu}, 
the third inequality is clear from a property of discrepancies, and the final inequality follows from ($\clubsuit$) in Step \ref{step7main-hu}. 
So we may assume that $Q$ is a component of $H_{S}$. 
Then $-1<a(Q,S,B_{S}+A_{S}+e_{m}H_{S})$, and by using this and ($\spadesuit$) in Step \ref{step4main-hu} we obtain $-1<a(Q,S_{m},B_{S_{m}}+A_{S_{m}}+e_{m}H_{S_{m}})$. 
Suppose that $Q$ is exceptional over $T_{m}$. 
This implies that $Q$ is exceptional over $S_{m}$ and it is not extracted by $\psi_{S_{m}}^{-1}$. 
By (a) in Step \ref{step5main-hu}, we have $a(Q,S_{m},B_{S_{m}}+A_{S_{m}}+e_{m}H_{S_{m}})\geq0$. 
Since $a(Q,S,B_{S}+A_{S})\leq 0$, we have 
\begin{equation*}
\begin{split}
a(Q,S,B_{S}+A_{S})&\leq 0\leq a(Q,S_{m},B_{S_{m}}+A_{S_{m}}+e_{m}H_{S_{m}})\leq a(Q,S_{m},B_{S_{m}}+A_{S_{m}})\\&\leq  a(Q,T_{m}, \Psi_{\infty}^{(m)}).
\end{split}
\end{equation*}
Here we used ($\clubsuit$) in Step \ref{step7main-hu} to obtain the final inequality. 
Suppose that $Q$ is a divisor on $T_{m}$. Since $T_{m}\dashrightarrow T_{i}$ is small for any $i$, we have 
\begin{equation*}
\begin{split}
a(Q,S,B_{S}+A_{S}+e_{i}H_{S})&\leq a(Q,S_{i},B_{S_{i}}+A_{S_{i}}+e_{i}H_{S_{i}})=a(Q,T_{i},\Psi_{T_{i}})\\&=a(Q,T_{m}, \Psi_{i}^{(m)})
\end{split}
\end{equation*}
where, the first inequality follows from ($\spadesuit$) in Step \ref{step4main-hu}, the second equality follows from (c) in Step \ref{step5main-hu}, 
and the third equality follows from that $Q$ is a divisor on $T_{m}$ and $T_{i}$.  
By considering the limit $i\to \infty$, we have $a(Q,S,B_{S}+A_{S})\leq a(Q,T_{m}, \Psi_{\infty}^{(m)})$. 

Therefore, in any case, we have $a(Q,S,B_{S}+A_{S})\leq a(Q,T_{m}, \Psi_{\infty}^{(m)})$ for any divisor $Q$ on $S$. 
Then the divisor $M$ in ($*$) in this step is exceptional over $S$. 
By Lemma \ref{lem--ind-birat}, the pair $(\widetilde{T}_{m}, \Theta+\sigma_{S_{m}}^{*}A_{S})$ has a good minimal model, and so does $(T_{m}, \Psi_{\infty}^{(m)})$. 
\end{step4}

\begin{step4}\label{step9main-hu}
With this step we complete the proof. 
The argument is the same as Step \ref{step9relative} in the proof of Theorem \ref{thmmain}. 
But we write details for the reader's convenience. 

Since $(T_{m},0)$ is $\mathbb{Q}$-factorial klt ((b) in Step \ref{step5main-hu}), by Step \ref{step8main-hu} and Theorem \ref{thmtermi}, there is a sequence of steps of a $(K_{T_{m}}+\Psi_{\infty}^{(m)})$-MMP to a good minimal model 
\begin{equation*}
(T_{m}, \Psi_{\infty}^{(m)})\dashrightarrow (T', \Psi'_{\infty}).
\end{equation*}
Then $(T',0)$ is also $\mathbb{Q}$-factorial klt since the map $T_{m}\dashrightarrow T'$ is also a sequence of steps of the $(K_{T_{m}}+t\Psi_{\infty}^{(m)})$-MMP for some $t<1$. 
Moreover, this log MMP contains only flips because $K_{T_{m}}+\Psi_{\infty}^{(m)}$ is the limit of movable divisors $K_{T_{m}}+\Psi_{i}^{(m)}$ (see the second sentence of Step \ref{step8main-hu}). 
In particular, $T'$ and $T_{m}$ are isomorphic in codimension one. 
Let $\Psi'_{i}$ be the birational transform of $\Psi_{T_{i}}$ on $T'$. 
Then $\Psi'_{i}\geq \Psi'_{i+1}\geq 0$ for any $i$ and $\Psi'_{\infty}={\rm lim}_{i \to \infty}\Psi'_{i}$.  
Since the map $T_{m}\dashrightarrow T'$ is also a sequence of steps of the $(K_{T_{m}}+\Psi_{i}^{(m)})$-MMP for any $i\gg m$, the pair $(T',\Psi'_{i})$ is lc for any $i\gg m$. 
Recall that $K_{T_{i}}+\Psi_{T_{i}}$ is semi-ample, which follows from construction of the morphism $\psi_{S_{i}}\colon T_{i}\to S_{i}$ in Step \ref{step5main-hu}. 
Because $K_{T'}+\Psi'_{i}$ is the birational transform of $K_{T_{i}}+\Psi_{T_{i}}$ and since $T_{i}$ and $T'$ are isomorphic in codimension one, the pair $(T_{i}, \Psi_{T_{i}})$ is a weak lc model of $(T',\Psi'_{i})$ with semi-ample log canonical divisor. 
Therefore, the pair $(T',\Psi'_{i})$ has a good minimal model for any $i\gg m$. 
By Theorem \ref{thmtermi}, there is a sequence of steps of a $(K_{T'}+\Psi'_{i})$-MMP to a good minimal model. 
Fix an $i\gg m$ such that 
$\Psi'_{i}-\Psi'_{\infty}$ is sufficiently small so that $\Psi'_{i}$ satisfies the property of Lemma \ref{lem--polytope-1}, that is, for any sequence of steps of any $(K_{T'}+\Psi'_{i})$-MMP to a good minimal model 
\begin{equation*}
(T',\Psi'_{i})\dashrightarrow(T'',\Psi''_{i})
\end{equation*}
and the contraction $T''\to W$ induced by $K_{T''}+\Psi''_{i}$, the divisor $K_{T''}+\Psi''_{\infty}$ is nef and $K_{T''}+\Psi''_{\infty}\sim_{\mathbb{R},W}0$, where $\Psi''_{\infty}$ is the birational transform of $\Psi'_{\infty}$ on $T''$. 
In particular, it is $\mathbb{R}$-linearly equivalent to the pullback of a nef divisor on $W$. 
Note that this log MMP also contains only flips since $K_{T'}+\Psi'_{i}$ is movable. 
Therefore, $T''$ and $T'$ are isomorphic in codimension one, hence $T''$ and $T_{i}$ are isomorphic in codimension one. 

We focus on the following diagram
\begin{equation*}
\xymatrix
{
(T'',\Psi''_{i})\ar[d]&(T_{i},\Psi_{T_{i}})\ar[d]^{\psi_{S_{i}}} \ar@{-->}[l]\\
W&(S_{i},B_{S_{i}}+A_{S_{i}}+e_{i}H_{S_{i}})\ar@{-->}[l]
}
\end{equation*}
and recall $K_{T_{i}}+\Psi_{T_{i}}=\psi_{S_{i}}^{*}(K_{S_{i}}+B_{S_{i}}+A_{S_{i}}+e_{i}H_{S_{i}})$ is semi-ample. 
Here $S_{i}\dashrightarrow W$ is the induced rational map. 
Since $T''$ and $T_{i}$ are isomorphic in codimension one, by the negativity lemma, we see that the rational map $S_{i}\dashrightarrow W$ is a morphism. 
Furthermore, by definition of $\Psi_{\infty}^{(m)}$ and since ${\rm lim}_{i \to \infty}e_{i}=0$, the birational transform of $K_{T''}+\Psi''_{\infty}$ on $S_{i}$ is $K_{S_{i}}+B_{S_{i}}+A_{S_{i}}$. 
Since $K_{T''}+\Psi''_{\infty}$ is $\mathbb{R}$-linearly equivalent to the pullback of a nef divisor on $W$, its birational transform $K_{S_{i}}+B_{S_{i}}+A_{S_{i}}$ is also $\mathbb{R}$-linearly equivalent to the pullback of a nef divisor on $W$. 
In particular, $K_{S_{i}}+B_{S_{i}}+A_{S_{i}}$ is nef. 

By the argument in Step \ref{step6main-hu}, we see that $(X,B+A)$ has a log minimal model. 
Therefore, the pair $(X,B+A)$ has a good minimal model by Theorem \ref{thm--abund-hu}. 
\end{step4}
So we are done. 
\end{proof}

\subsection{From Theorem \ref{thmmain--hu} to Theorem \ref{thm--lcample}, and proof of Theorem \ref{thm--lcfano}}
\begin{thm}\label{thmrelative}
Let $(X,B)$ be an lc pair, and let $\pi \colon X \to Y$ and $\varphi\colon Y\to Z$ be projective morphisms of normal quasi-projective varieties. 
Let $A_{Y}$ be a $\varphi$-ample $\mathbb{R}$-divisor on $Y$ such that $(X,B+\pi^{*}A_{Y})$ is lc. 
Suppose that $(X,B)$ has a good minimal model over $Y$. 

If $K_{X}+B+\pi^{*}A_{Y}$ is pseudo-effective over $Z$, then $(X,B+\pi^{*}A_{Y})$ has a good minimal model over $Z$. 
\end{thm}

\begin{proof}
Let $Z\hookrightarrow Z^{c}$ be an open immersion to a normal projective variety $Z^{c}$. 
Then there is an open immersion $Y\hookrightarrow Y^{c}$ to a normal projective variety $Y^{c}$, a projective morphism $\varphi^{c}\colon Y^{c}\to Z^{c}$, and a $\varphi^{c}$-ample $\mathbb{R}$-divisor $A_{Y^{c}}$ on $Y^{c}$ such that $A_{Y^{c}}|_{Y}\sim_{\mathbb{R},Z}A_{Y}$. 
We take an lc closure $(X^{c},B^{c})$ of $(X,B)$ as in \cite[Corollary 1.3]{has-mmp}. 
Then any lc center of $(X^{c},B^{c})$ intersects $X$. 
By replacing $A_{Y^{c}}$ if necessary, we may assume that $(X^{c},B^{c}+\pi^{c*}A_{Y^{c}})$ is lc, where $\pi^{c}\colon X^{c}\to Y^{c}$ is the induced morphism. 
By \cite[Theorem 1.2]{has-mmp}, $(X^{c},B^{c})$ has a good minimal model over $Y^{c}$. 
If $(X^{c},B^{c}+\pi^{c*}A_{Y^{c}})$ has a good minimal model over $Z^{c}$, by restricting it over $Z$, we see that $(X,B+\pi^{*}A_{Y})$ has a good minimal model over $Z$. 
Thus, replacing $\pi \colon X\to Y$, $\varphi\colon Y\to Z$, $(X,B)$ and $A_{Y}$ by $\pi^{c}\colon X^{c}\to Y^{c}$, $\varphi^{c}\colon Y^{c}\to Z^{c}$, $(X^{c},B^{c})$ and $A_{Y^{c}}$ respectively, we may assume that $X$, $Y$ and $Z$ are projective. 

Let $A_{Z}$ be a sufficiently ample divisor on $Z$ such that $G_{Y}:=A_{Y}+\varphi^{*}A_{Z}$ is ample. 
We may assume that $(X,B+\pi^{*}G_{Y})$ is lc and $(X,B+\pi^{*}A_{Y})$ has a good minimal model over $Z$ if and only if $(X,B+\pi^{*}G_{Y})$ has a good minimal model. 
Hence, by replacing $A_{Y}$ with $G_{Y}$, we may assume that $Z$ is a point. 

By replacing $(X,B)$ with a good minimal model over $Y$, we may assume that $K_{X}+B$ is semi-ample over $Y$. 
Let $\pi'\colon X \to Y'$ be the contraction over $Y$ induced by $K_{X}+B$, and let $g\colon Y'\to Y$ be the natural morphism. 
Let $A'$ be a $g$-ample $\mathbb{R}$-divisor on $Y'$ such that $K_{X}+B\sim_{\mathbb{R}}\pi'^{*}A'$. 
Then there is $\delta\in (0,1)$ such that $\delta A'+g^{*}A_{Y}$ is ample. 
Pick $H\sim_{\mathbb{R}}\delta A'+g^{*}A_{Y}$ such that $(X,B+\tfrac{1}{1-\delta}\pi'^{*}H)$ is lc. 
With this $\delta$ and $H$, we may write
\begin{equation*}
K_{X}+B+\pi^{*}A_{Y}\sim_{\mathbb{R}}(1-\delta)(K_{X}+B)+\delta \pi'^{*}A'+ \pi'^{*}g^{*}A_{Y}\sim_{\mathbb{R}}(1-\delta)(K_{X}+B+\tfrac{1}{1-\delta}\pi'^{*}H).
\end{equation*}
So $(X,B+\pi^{*}A_{Y})$ has a good minimal model if and only if $(X,B+\tfrac{1}{1-\delta}\pi'^{*}H)$ has a good minimal model. 
Apply Theorem \ref{thmmain--hu} to $\pi'\colon X \to Y'$, $(X,B)$, $\tfrac{1}{1-\delta}H$ and $\tfrac{1}{1-\delta}\pi'^{*}H$ (put $C=0$ in the notations of Theorem \ref{thmmain--hu}). 
Then $(X,B+\tfrac{1}{1-\delta}\pi'^{*}H)$ has a good minimal model, and therefore $(X,B+\pi^{*}A_{Y})$ has a good minimal model. 
\end{proof}

\begin{proof}[Proof of Theorem \ref{thm--lcample}]
If $K_{X}+B+A$ is not pseudo-effective over $Z$, then $(X,B+A)$ has a Mori fiber space over $Z$. 
So we may assume $K_{X}+B+A$ is pseudo-effective over $Z$. 
Then the theorem follows from Theorem \ref{thmrelative} by putting $Y=X$ and $\pi={\rm id}_{X}$. 
\end{proof}

\begin{proof}[Proof of Theorem \ref{thm--lcfano}]
We may assume $B$ is a $\mathbb{Q}$-divisor. 
Let $A_{Y}$ be an ample $\mathbb{Q}$-divisor on $Y$ such that $-(K_{X}+B)-f^{*}A_{Y}$ is ample. 
We pick $A\sim_{\mathbb{Q}}-(K_{X}+B)-f^{*}A_{Y}$ such that $(X,B+A)$ is lc. 
Applying the canonical bundle formula to $(X,B+A)\to Y$, there is a generalized lc pair $(Y,\Delta_{Y}+M_{Y})$, where $\Delta_{Y}$ is the discriminant part and $M_{Y}$ is the moduli part, such that 
$K_{X}+B+A\sim_{\mathbb{Q}}f^{*}(K_{Y}+\Delta_{Y}+M_{Y}).$  
By replacing $A_{Y}$ with a general one, we may assume that $A_{Y}\geq0$ and $(Y,\Delta_{Y}+A_{Y}+M_{Y})$ is a generalized lc pair with the discriminant part $\Delta_{Y}+A_{Y}$. 
Note that $K_{Y}+\Delta_{Y}+M_{Y}+A_{Y}\sim_{\mathbb{Q}}0$. 

We first treat the case when $X$ is $\mathbb{Q}$-factorial. 
By \cite[Appendix]{bh} by Jinhyung Park, $X$ is a Mori dream space, therefore we see that $Y$ is a  (non-$\mathbb{Q}$-factorial) Mori dream space by \cite[Theorem 9.3]{okawa}. Moreover, by \cite[Theorem 2.3]{coxrings} there is a small $\mathbb{Q}$-factorialization $Y'\to Y$. Note that $Y'$ is a Mori dream space. 
Let $M_{Y'}$ be the birational transform of $M_{Y}$ on $Y'$. 
The divisor $M_{Y}$ is the birational transform of a nef divisor by a property of the canonical bundle formula, and, since $Y\dashrightarrow Y'$ is small $M_{Y'}$ is also the birational transform of a nef divisor. 
From this we see that $M_{Y'}$ is the limit of movable divisors, so there is a small birational map $Y'\dashrightarrow Y''$ to a projective $\mathbb{Q}$-factorial variety $Y''$ such that the birational transform $M_{Y''}$ of $M_{Y'}$ on $Y''$ is semi-ample. 
Then $(Y'', \Delta_{Y''}+A_{Y''}+M_{Y''})$ is generalized lc, where $\Delta_{Y''}$ and $A_{Y''}$ are birational transforms of $\Delta_{Y}$ and $A_{Y}$ on $Y''$, respectively. 
Then $(Y'',\Delta_{Y''}+A_{Y''})$ is lc, so there is $0\leq R_{Y''}\sim_{\mathbb{Q}}M_{Y''}$ such that $(Y'',\Delta_{Y''}+R_{Y''}+A_{Y''})$ is lc. 
By construction, we have $K_{Y''}+\Delta_{Y''}+R_{Y''}+A_{Y''}\sim_{\mathbb{Q}}0$. 
Let $R_{Y}$ be the birational transform of $R_{Y''}$ on $Y$, and set $B_{Y}=\Delta_{Y}+R_{Y}$. 
Then $B_{Y}$ is the desired boundary $\mathbb{Q}$-divisor. 
Indeed, by taking a common resolution of $Y\dashrightarrow Y''$, which is small by construction, and by the negativity lemma, we see that 
$K_{Y}+B_{Y}+A_{Y}\sim_{\mathbb{Q}}0$ and $(Y,B_{Y}+A_{Y})$ is lc. 
So $(Y,B_{Y})$ is lc and $-(K_{Y}+B_{Y})$ is ample.

Next we treat the case when $Y$ is $\mathbb{Q}$-factorial. 
By assumption, $M_Y$ is $\mathbb{Q}$-Cartier. 
Put $M=f^*M_Y$, and pick a rational number $\delta>0$ such that $-(K_{X}+B)+\delta M$ is ample. 
Then we can find an effective $\mathbb{Q}$-divisor $G\sim_{\mathbb{Q}}-(K_{X}+B)+\delta M$ such that $(X,B+G)$ is lc. 
By Theorem \ref{thm--lcample} and since we have $\delta M \sim_{\mathbb{Q}} K_X+B+G$, we deduce that $M$ birationally has a Nakayama--Zariski decomposition with semi-ample positive part. By \cite[III, 5.18 Corollary]{nakayama} and \cite[III, 5.17 Corollary]{nakayama}, we see that $M_Y$ birationally has a Nakayama--Zariski decomposition with semi-ample positive part. 

Let $\pi\colon Y' \to Y$ be a resolution on which $\pi^*M_Y=P+N$ where $P=P_\sigma(\pi^*M_Y)$ is a semi-ample $\mathbb{Q}$-divisor and $N=N_\sigma(\pi^*M_Y)$. 
By definition and possibly by replacing $Y'$ we can write $K_{Y'}+\Delta_{Y'}+M_{Y'}=\pi^*(K_{Y}+\Delta_{Y}+M_{Y})$ such that $(Y',\Delta_{Y'})$ is sub-lc and $M_{Y'}$ is nef. 
By the negativity lemma, we can find a $\pi$-exceptional $\mathbb{Q}$-divisor $E\geq 0$ such that $\pi^{*}M_{Y}=M_{Y'}+E$. 
By \cite[III, 1.14 (2) Proposition]{nakayama}, we have $E- N\geq 0$, so
\begin{equation*}
K_{Y'}+\Delta_{Y'}+P+(N-E)=\pi^*(K_{Y}+\Delta_{Y}+M_{Y})
\end{equation*}
and the pair $(Y',\Delta_{Y'}+N-E)$ is sub-lc, and $\pi_{*}(\Delta_{Y'}+N-E)=\Delta_{Y}$ because $E$ is $\pi$-exceptional and $E- N\geq 0$. 
We pick $0\leq P'\sim_{\mathbb{Q}}P$ such that $(Y',\Delta_{Y'}+N-E+P')$ is sub-lc. 
Setting $B_{Y}$ as $\pi_{*}(\Delta_{Y'}+N-E+P')$, we see that $B_{Y}\geq 0$, $B_{Y}\sim_{\mathbb{Q}}\Delta_{Y}+M_{Y}$, and $K_{Y'}+\Delta_{Y'}+P'+(N-E)=\pi^*(K_{Y}+B_{Y})$. 
The second and the third condition show $K_{Y}+B_{Y}+A_{Y}\sim_{\mathbb{Q}}0$ and that $(Y,B_{Y})$ is lc. 
So we complete the proof. 
\end{proof}

A contraction $X\to Y$ of normal projective varieties is said to be a {\em Fano fibration} if $-K_{X}$ is ample over $Y$ and ${\rm dim}X >{\rm dim}Y$. 
We close this section with a result on existence of a birational contraction to a Fano fibration. 

\begin{thm}\label{thm--cont-fanofib}
Let $(X,0)$ be a projective lc pair such that $K_X$ is not pseudo-effective. 
Let $A$ be an ample $\mathbb{R}$-divisor on $X$, and let $t$ be the pseudo-effective threshold of $A$ with respect to $K_X$, that is, the unique real number $t>0$ such that $K_{X}+tA$ is pseudo-effective but not big. 

Then there is a birational contraction $X\dashrightarrow Y$ and a Fano fibration $Y\to Z$ with  ${\rm dim} Z = \kappa_{\sigma} (X,K_X+tA)$.
\end{thm}

\begin{proof}
By rescaling and replacing $A$, we may assume that $t=1$ and $(X,(1+\epsilon)A)$ is lc for some $\epsilon>0$. 
By Theorem \ref{thm--lcample}, $(X,A)$ and $(X,(1+\epsilon)A)$ have good minimal models. 
Let $(Y',E')$ be a dlt blow-up of $(X,0)$, and set $G'$ and $G'_{\epsilon}$ as the pullbacks of $A$ and $(1+\epsilon) A$, respectively. 
Note that $(1+\epsilon)G'=G'_{\epsilon}$. 
By Theorem \ref{thmtermi}, there is a sequence of steps of a $(K_{Y'}+E'+G')$-MMP to a good minimal model $(Y'',E''+G'')$, and there is a contraction $Y''\to Z$ induced by $K_{Y''}+E''+G''$. 
Then ${\rm dim}Z=\kappa_{\sigma} (X,K_X+A)<{\rm dim}Y''$. 
Let $G''_{\epsilon}$ be the birational transform of $G'_{\epsilon}$ on $Y''$. 
By choosing $\epsilon>0$ sufficiently small, we may assume that the map $Y'\dashrightarrow Y''$ is a sequence of steps of a $(K_{Y'}+E'+G'_{\epsilon})$-MMP and any $(K_{Y''}+E''+G''_{\epsilon})$-MMP is a $(K_{Y''}+E''+G''_{\epsilon})$-MMP  over $Z$. 
By Theorem \ref{thmtermi}, we can run a $(K_{Y''}+E''+G''_{\epsilon})$-MMP, which is also a $(K_{Y''}+E''+G''_{\epsilon})$ over $Z$, to a good minimal model $(Y''', E'''+G'''_{\epsilon})$. 
Since $K_{Y''}+E''+G''\sim_{\mathbb{R},Z}0$ and $G''_{\epsilon}=(1+\epsilon) G''$, we see that $-(K_{Y''}+E'')$ is big over $Z$ and the log MMP is a $(-(K_{Y''}+E''))$-MMP over $Z$. 
Let $Y'''\to Y$ be the contraction over $Z$ induced by $K_{Y'''}+E'''+G'''_{\epsilon}\sim_{\mathbb{R},Z}-\epsilon(K_{Y'''}+E''')$. 
Since $\epsilon$ is sufficiently small, by construction of $(Y',E'+G'_{\epsilon})$ and $Y'\dashrightarrow Y$, we see that $(Y,E+G_{\epsilon})$ is the log canonical model of $(X,(1+\epsilon) A)$, so $E'$ is contracted by $Y'\dashrightarrow Y$. 
Then the map $X\dashrightarrow Y$ is a birational contraction and $-K_{Y}$ is ample over $Z$ since $K_{Y}+G_{\epsilon}\sim_{\mathbb{R},Z}-\epsilon K_{Y}$. 
Therefore $Y\to Z$ is the desired contraction. 
\end{proof}

\section{Termination of log MMP with scaling}\label{sec6}

In this section, we prove Theorem \ref{thm--mmp-scaling}. 
In this section, we do not assume log minimal models and Mori fiber spaces to be $\mathbb{Q}$-factorial. 
We recall basic facts on non-$\mathbb{Q}$-factorial log MMP. 
For definition of non-$\mathbb{Q}$-factorial log MMP, see \cite[4.9]{fujino-book}. 

\begin{rem}\label{remmmp}
Let $(X_{1},\Delta_{1})$ be an lc pair, and let $X_{1}\to Z$ be a projective morphism of  normal quasi-projective varieties. 
Let
\begin{equation*}
(X_{1},\Delta_{1})\dashrightarrow (X_{2},\Delta_{2})\dashrightarrow \cdots \dashrightarrow (X_{i},\Delta_{i})\dashrightarrow \cdots
\end{equation*}
be a sequence of steps of a $(K_{X_{1}}+\Delta_{1})$-MMP over $Z$. 
\begin{enumerate}
\item[(1)]
The map $X_{i}\dashrightarrow X_{i+1}$ is a birational contraction for any $i\geq 1$. 
\item[(2)]
By the cone theorem, for any $i$ and any $\mathbb{Q}$-Cartier divisor $D_{i}$ on $X_{i}$, the birational transform $D_{i+1}$ on $X_{i+1}$ is $\mathbb{Q}$-Cartier. 
\item[(3)]
There is a diagram
\begin{equation*}
\xymatrix
{
(Y_{1},\Gamma_{1})\ar[d]_{f_{1}}\ar@{-->}[r] &\cdots\ar@{-->}[r]&(Y_{k_{i}},\Gamma_{k_{i}})\ar[d]_{f_{i}}\ar@{-->}[r] &\cdots\\
(X_{1},\Delta_{1})\ar@{-->}[r] &\cdots\ar@{-->}[r]&(X_{i},\Delta_{i})\ar@{-->}[r] &\cdots
}
\end{equation*}
where $f_{i}$ are dlt blow-ups and upper horizontal sequence of birational maps is a sequence of steps of a $(K_{Y_{1}}+\Gamma_{1})$-MMP over $Z$ (see also \cite[Section 2]{bh}).
\end{enumerate}
\end{rem}

First, we prove a weaker result than Theorem \ref{thm--mmp-scaling}. 

\begin{prop}\label{prop--mmplc}
Let $\pi\colon X\to Z$ be a projective morphism of normal quasi-projective varieties, and let $(X,B)$ be an lc pair. 
Suppose that $(X,B)$ has a log minimal model over $Z$ or $K_{X}+B$ is not pseudo-effective over $Z$.  
  
Then there is a sequence of birational contractions 
\begin{equation*}
(X,B)=(X_{1},B_{1})\dashrightarrow (X_{2},B_{2})\dashrightarrow \cdots \dashrightarrow (X_{l},B_{l})
\end{equation*}
of a non-$\mathbb{Q}$-factorial $(K_{X}+B)$-MMP over $Z$ that terminates. 
\end{prop}

\begin{proof}
For any variety $V$ with a projective morphsim $V\to Z$, we denote the $\mathbb{R}$-vector space of relative numerical classes of $\mathbb{R}$-Cartier divisors by $N^{1}(V/Z)_{\mathbb{R}}$, in other words, we set $N^{1}(V/Z)_{\mathbb{R}}:={\rm Pic}(V)_{\mathbb{R}}/\equiv_{Z}$. 
We denote the relative Picard number ${\rm dim}N^{1}(V/Z)_{\mathbb{R}}$ by $\rho(V/Z)$. 

Fix a dlt blow-up $(Y,\Gamma)\to (X,B)$. 
Suppose that there exists a sequence of steps of a $(K_{X}+B)$-MMP $(X,B)\dashrightarrow \cdots \dashrightarrow(X_{i},B_{i})$ over $Z$ such that $X_{i-1}\dashrightarrow X_{i}$ contracts a divisor. 
Then we can construct a diagram
\begin{equation*}
\xymatrix
{
(Y,\Gamma)\ar[d]\ar@{-->}[r] &\cdots\ar@{-->}[r]&(Y_{k_{i-1}},\Gamma_{k_{i-1}})\ar[d]\ar@{-->}[r] &(Y_{k_{i}},\Gamma_{k_{i}})\ar[d]\\
(X,B)\ar@{-->}[r] &\cdots\ar@{-->}[r]&(X_{i-1},B_{i-1})\ar@{-->}[r] &(X_{i},B_{i})
}
\end{equation*}
as in Remark \ref{remmmp} (3). 
Then the log MMP $(Y_{k_{i-1}},\Gamma_{k_{i-1}})\dashrightarrow(Y_{k_{i}},\Gamma_{k_{i}})$ contracts a divisor. 
Therefore, we have $\rho(Y/Z)\geq \rho(Y_{k_{i-1}}/Z)>\rho(Y_{k_{i}}/Z)$. 
We replace $(X,B)$ and $(Y,\Gamma)$ by $(X_{i},B_{i})$ and $(Y_{k_{i}},\Gamma_{k_{i}})$ respectively, and repeat the above argument. 
Since the relative Picard number $\rho(Y/Z)$ strictly decreases, the process eventually stops. 
Thus, replacing $(X,B)$ we may assume that any $(K_{X}+B)$-MMP 
$(X,B)\dashrightarrow \cdots \dashrightarrow(X_{i},B_{i})\dashrightarrow \cdots$ over $Z$ contains only small birational maps, in other words, the varieties $X_{i}$ and $X_{i+1}$ are isomorphic in codimension one. 

By Remark \ref{remmmp} (2) and the negativity lemma, all steps $\phi_{i}\colon (X_{i},B_{i})\dashrightarrow (X_{i+1},B_{i+1})$ of any $(K_{X}+B)$-MMP over $Z$ induce linear maps $\phi_{i*}\colon N^{1}(X_{i}/Z)_{\mathbb{R}}\to N^{1}(X_{i+1}/Z)_{\mathbb{R}}$, and all $\phi_{i*}$ are injective. 
Suppose that there is a sequence of steps of a $(K_{X}+B)$-MMP $(X,B)\dashrightarrow \cdots \dashrightarrow(X_{i},B_{i})\dashrightarrow \cdots$ over $Z$ such that $\rho(X/Z)<\rho(X_{i}/Z)$ for some $i$. 
Then we replace $(X,B)$ with $(X_{i},B_{i})$ for the $i$. 
With notations as in the diagram in the previous paragraph, we can check that $\rho(Y/Z)-\rho(X/Z)$ strictly decreases after replacing $(X,B)$. 
So this process eventually stops. 
Replacing $(X,B)$, we may assume $\rho(X_{i}/Z)=\rho(X_{i+1}/Z)$ for all steps $\phi_{i}\colon(X_{i},B_{i})\dashrightarrow (X_{i+1},B_{i+1})$ of any $(K_{X}+B)$-MMP over $Z$.  
Then the linear map $\phi_{i*}\colon N^{1}(X_{i}/Z)_{\mathbb{R}}\to N^{1}(X_{i+1}/Z)_{\mathbb{R}}$ is an isomorphism. 

By the standard argument of convex geometry, we can find positive real numbers $r_{1},\cdots,r_{m}$  and $\mathbb{Q}$-divisors $B^{(1)},\cdots, B^{(m)}$ on $X$ such that $\sum_{j=1}^{m}r_{j}=1$, all $(X,B^{(j)})$ are lc and $B=\sum_{j=1}^{m}r_{j}B^{(j)}$. 
We put $d=\rho(X/Z)$. 
Fix positive real numbers $\alpha_{1},\cdots,\alpha_{d}$ such that $ \alpha_{1},\cdots, \alpha_{d}$ are linearly independent over $\mathbb{Q}(r_{1},\cdots, r_{m})$, where $\mathbb{Q}(r_{1},\cdots, r_{m})$ is the field over $\mathbb{Q}$ generated by $r_{1},\cdots, r_{m}$. 
We pick sufficiently ample $\mathbb{Q}$-divisors $A^{(1)},\cdots ,A^{(d)}$ so that they form a basis of $N^{1}(X/Z)_{\mathbb{R}}$, and we put $A=\sum_{k=1}^{d}\alpha_{k}A^{(k)}$. 
Since $A^{(1)},\cdots ,A^{(d)}$ are sufficiently ample, we may assume $K_{X}+B+A$ is nef over $Z$. 

We run a $(K_{X}+B)$-MMP over $Z$ with scaling of $A$
\begin{equation*}
(X,B)=(X_{1},B_{1})\dashrightarrow \cdots \dashrightarrow(X_{i},B_{i})\dashrightarrow \cdots
\end{equation*}
and set $\lambda_{i}={\rm inf}\!\set{\!\mu \in \mathbb{R}_{\geq0} | \text{$K_{X_{i}}+B_{i}+\mu A_{i}$ is nef over $Z$}\!}$, where $A_{i}$ is the birational transform of $A$ on $X_{i}$. 
We show that $\lambda_{i}>\lambda_{i+1}$ for any $i\geq1$. 
Suppose by contradiction that $\lambda_{i}=\lambda_{i+1}$ for an $i$. 
Let $X_{i}\to V_{i}$ and $X_{i+1}\to V_{i+1}$ be the extremal contractions of the log MMP. 
Then $X_{i+1}$ is not isomorphic to $V_{i}$ because $\rho(X_{i}/Z)=\rho(X_{i+1}/Z)$ by the argument of the third paragraph. 
So there is a curve $\xi$ on $X_{i+1}$ contracted by $X_{i+1}\to V_{i}$. 
Pick any curve $\xi'$ on $X_{i+1}$ contracted by $X_{i+1}\to V_{i+1}$. 
By Remark \ref{remmmp} (2), all $K_{X_{i+1}}+B_{i+1}^{(j)}$ are $\mathbb{Q}$-Cartier. 
Since $K_{X}+B=\sum_{j=1}^{m}r_{j}(K_{X}+B^{(j)})$, the intersection numbers $(K_{X_{i+1}}+B_{i+1})\cdot \xi$ and $(K_{X_{i+1}}+B_{i+1})\cdot \xi'$ are elements of $\mathbb{Q}(r_{1},\,\cdots, r_{m})$. 
Furthermore, we have $A_{i+1}=\sum_{k=1}^{d}\alpha_{k}A_{i+1}^{(k)}$, where $A_{i+1}^{(k)}$ is the birational transform of $A^{(k)}$ on $X_{i+1}$, and $\{A_{i+1}^{(k)}\}_{k}$ form a basis of $N^{1}(X_{i+1}/Z)_{\mathbb{R}}$. 
Since $A^{(k)}$ are $\mathbb{Q}$-divisors, by Remark \ref{remmmp} (2), the intersection numbers $(A_{i+1}^{(k)}\cdot \xi)$ and $(A_{i+1}^{(k)}\cdot \xi')$ are rational numbers for any $1\leq k\leq d$. 
By the definition of $\lambda_{i}$ and $\lambda_{i+1}$, we have 
\begin{equation*}
\begin{split}
&(K_{X_{i+1}}+B_{i+1}+\lambda_{i}A_{i+1})\cdot \xi=0\quad {\rm and}\quad (K_{X_{i+1}}+B_{i+1}+\lambda_{i+1}A_{i+1})\cdot \xi'=0. 
\end{split}
\end{equation*}
Since we assume $\lambda_{i}=\lambda_{i+1}$ and since $(K_{X_{i+1}}+B_{i+1})\cdot \xi'<0$, 
we have 
\begin{equation*}
\frac{(A_{i+1}\cdot \xi)}{(A_{i+1}\cdot \xi')}=\frac{(K_{X_{i+1}}+B_{i+1})\cdot \xi}{(K_{X_{i+1}}+B_{i+1})\cdot \xi'}.
\end{equation*}
We put $\beta=\frac{(K_{X_{i+1}}+B_{i+1})\cdot \xi}{(K_{X_{i+1}}+B_{i+1})\cdot \xi'}\in \mathbb{Q}(r_{1},\cdots, r_{m})$. 
Then $A_{i+1}^{(k)}\cdot(\xi-\beta \xi')$ is an element of $\mathbb{Q}(r_{1}, \cdots, r_{m})$ for any $k$ and 
$\sum_{k=1}^{d}\alpha_{k}A_{i+1}^{(k)}\cdot(\xi-\beta \xi')=0.$
Since $ \alpha_{1},\cdots, \alpha_{d}$ are linearly independent over $\mathbb{Q}(r_{1},\cdots, r_{m})$, we have $A_{i+1}^{(k)}\cdot(\xi-\beta \xi')=0$ for any $1\leq k\leq d$. 
Since $A_{i+1}^{(k)}$ are the basis of $N^{1}(X_{i+1}/Z)_{\mathbb{R}}$, we see that $\xi-\beta \xi'=0$ in $N_{1}(X_{i+1}/Z)_{\mathbb{R}}$. 
So $\xi$ and $\xi'$ generate the same half line in $N_{1}(X_{i+1}/Z)_{\mathbb{R}}$. 
But it is impossible because $\xi$ is contracted by $X_{i+1}\to V_{i}$ and $\xi'$ is contracted by $X_{i+1}\to V_{i+1}$. 
Hence we have $\lambda_{i}>\lambda_{i+1}$. 

Finally, we prove that the $(K_{X}+B)$-MMP over $Z$ with scaling of $A$ must terminate. 
If it does not terminate, set $\lambda={\rm lim}_{i\to \infty}\lambda_{i}$. 
Then $\lambda\neq \lambda_{i}$ for any $i$ by the above argument, and the divisor $K_{X}+B+\lambda A$ is pseudo-effective over $Z$. 
Moreover, the pair $(X,B+\lambda A)$ has a log minimal model over $Z$ by Theorem \ref{thm--lcample} or hypothesis of Proposition \ref{prop--mmplc}. 
But it contradicts \cite[Theorem 4.1 (iii)]{birkar-flip}. 
So the log MMP terminates. 
\end{proof}

From now on, we prove Theorem \ref{thm--mmp-scaling}. 

\begin{proof}[Proof of Theorem \ref{thm--mmp-scaling}]
The strategy is the same as \cite[Proof of Lemma 2.14]{has-mmp}. 
In this proof, for any $\mathbb{R}$-divisor $D$ on $X$, any $i\in \mathbb{Z}_{>0}$ and any birational contraction $X\dashrightarrow X_{i}$, $D_{i}$ denotes the birational transform of $D$ on $X_{i}$. 

Put $(X,B)=(X_{1},B_{1})$ and set 
\begin{equation*}
\lambda_{1}={\rm inf}\!\set{\!\mu\in\mathbb{R}_{\geq 0}| \text{$K_{X_{1}}+B_{1}+\mu A_{1}$ is nef over $Z$}\!}.
\end{equation*}
If $\lambda_{1}=0$, there is nothing to prove. 
If $\lambda_{1}>0$, by the argument of length of extremal rays, we can find $0\leq \lambda'_{1}<\lambda_{1}$ such that the birational transform of $K_{X_{1}}+B_{1}+\lambda_{1}A_{1}$ is trivial over extremal contraction in each step of any $(K_{X_{1}}+B_{1}+\lambda'_{1}A_{1})$-MMP over $Z$ (cf.~\cite[Proposition 3.2 (5)]{birkar-existII} or \cite[Lemma 2.12]{has-mmp}). 
We note that non-$\mathbb{Q}$-factorial lc case of \cite[Proposition 3.2 (5)]{birkar-existII} and \cite[Lemma 2.12]{has-mmp} hold because we can apply results in \cite[Section 18]{fujino-fund}. 
From this, it follows that any $(K_{X_{1}}+B_{1}+\lambda'_{1}A_{1})$-MMP over $Z$ is also a $(K_{X_{1}}+B_{1})$-MMP over $Z$ with scaling of $A_{1}$. 
By Proposition \ref{prop--mmplc}, there exists a sequence of steps of the $(K_{X_{1}}+B_{1}+\lambda'_{1}A)$-MMP over $Z$
\begin{equation*}
(X_{1},B_{1}+\lambda'_{1}A_{1})\dashrightarrow (X_{2},B_{2}+\lambda'_{1}A_{2})\dashrightarrow \cdots \dashrightarrow (X_{k_{1}},B_{k_{1}}+\lambda'_{1}A_{k_{1}})
\end{equation*}
to a log minimal model or a Mori fiber space $(X_{k_{1}},B_{k_{1}}+\lambda'_{1}A_{k_{1}})$ over $Z$. 
Note that this sequence of birational maps is a sequence of step of a $(K_{X_{1}}+B_{1})$-MMP over $Z$ with scaling of $A_{1}$. 
If $(X_{k_{1}},B_{k_{1}}+\lambda'_{1}A_{k_{1}})$ is a Mori fiber space over $Z$, then we complete the proof because $(X_{k_{1}},B_{k_{1}})$ is also a Mori fiber space over $Z$ by construction of the above log MMP. 
If $(X_{k_{1}},B_{k_{1}}+\lambda'_{1}A_{k_{1}})$ is a log minimal model over $Z$, we set 
\begin{equation*}
\lambda_{2}={\rm inf}\!\set{\!\mu\in\mathbb{R}_{\geq 0} | \text{$K_{X_{k_{1}}}+B_{k_{1}}+\mu A_{k_{1}}$ is nef over $Z$}\!}.
\end{equation*}
Then $\lambda_{2}\leq \lambda'_{1}<\lambda_{1}$ by construction. 
If $\lambda_{2}=0$, we stop the discussion. 
If $\lambda_{2}>0$, then there is $0\leq \lambda'_{2}<\lambda_{2}$ and a sequence of steps of the $(K_{X_{k_{1}}}+B_{k_{1}}+\lambda'_{2}A_{k_{1}})$-MMP over $Z$
\begin{equation*}
(X_{k_{1}},B_{k_{1}}+\lambda'_{2}A_{k_{1}})\dashrightarrow (X_{k_{2}},B_{k_{2}}+\lambda'_{2}A_{k_{2}})
\end{equation*}
to a log minimal model or a Mori fiber space $(X_{k_{2}},B_{k_{2}}+\lambda'_{2}A_{k_{2}})$ over $Z$ such that $K_{X_{k_{1}}}+B_{k_{1}}+\lambda_{2} A_{k_{1}}$ is trivial over each extremal contraction of the log MMP. 
Then the map $X_{1}\dashrightarrow X_{k_{2}}$ is a sequence of steps of the $(K_{X_{1}}+B_{1})$-MMP over $Z$ with scaling of $A_{1}$. 
By repeating this discussion, we obtain a sequence of steps of a $(K_{X_{1}}+B_{1})$-MMP over $Z$ with scaling of $A_{1}$
\begin{equation*}
(X_{1},B_{1})\dashrightarrow \cdots \dashrightarrow (X_{i},B_{i})\dashrightarrow \cdots.
\end{equation*}
We show that this log MMP terminates.
Suppose that this log MMP does not terminate. 
Set $\tilde{\lambda}_{i}={\rm inf}\{\mu\in \mathbb{R}_{\geq0}|K_{X_{i}}+B_{i}+\mu A_{i}{\rm \; is\;nef\; over\;}Z\}$
and $\lambda:={\rm lim}_{i\to \infty}\tilde{\lambda}_{i}$. 
By construction, $\lambda\neq \tilde{\lambda}_{j}$ for any $j$, and the divisor $K_{X}+B+\lambda A$ is pseudo-effective over $Z$. 
Furthermore, the pair $(X,B+\lambda A)$ has a log minimal model over $Z$ by Theorem \ref{thm--lcample} or hypothesis of Theorem \ref{thm--mmp-scaling}. 
But it contradicts \cite[Theorem 4.1 (iii)]{birkar-flip}. 
So the log MMP terminates. 
\end{proof}


\end{document}